\documentclass[psamsfont,a4paper,11pt]{article}

\usepackage[english]{babel}   %loads english

\usepackage{setspace}     %controls line-spacing
\usepackage{enumerate} 	%allows lists
\usepackage{afterpage} 	%keeps tables and figures where they belong
\usepackage[usenames,dvipsnames]{color}
\usepackage{amsmath} 	%just apparently the most important thing ever
\usepackage{amssymb}	%symbols in math
\usepackage{mathrsfs} 	%really fancy script font
\usepackage{amsfonts} 	%standard fonts
\usepackage{latexsym} 	%more symbols
\usepackage[latin1]{inputenc} %sets up proper special characters

\usepackage{amsthm} 	%sets up theorem styling
\usepackage{stackrel}       %allows stacking in mathmode  

\usepackage{tabularx}	%automatically widening tables
\usepackage{longtable}	%for tables longer than a page

\usepackage{verbatim} %comment environment?
\usepackage{blindtext} %something about dummy text

\allowdisplaybreaks

\def\GL{\mathop{\rm GL}}

\renewcommand{\to}{\longrightarrow}

\def\und{\underline}

\theoremstyle{definition}

\theoremstyle{remark}

\newcommand{\bbmx}{\begin{bmatrix}}
\newcommand{\ebmx}{\end{bmatrix}}

\newtheorem{ind}[]{{\rm\it Indice}}

\usepackage{multicol}
\usepackage{mathrsfs}
\usepackage{tikz-cd}
\usepackage{adjustbox,amsfonts,amssymb,array,graphicx,amscd,amsmath,amsthm,tikz,booktabs,tabularx,multirow,multicol,stmaryrd,tabu,xcolor}
\usepackage[flushleft]{threeparttable}
\usepackage[fixamsmath]{mathtools}
\usepackage{stackengine}
\usepackage{ mathdots }
\usepackage{subcaption}
\usepackage{cite}
\usepackage{float}
\usepackage{hyperref}

\usetikzlibrary{shapes,arrows,decorations.pathreplacing}
\usepackage{enumitem}
\setlist[enumerate]{label=(\alph*)}

\newcommand{\twid}[1]{\widetilde{#1}}

\theoremstyle{plain}
\newtheorem{theorem}{Theorem}[section]
\newtheorem*{maintheorem}{Main Theorem}
\newtheorem{proposition}[theorem]{Proposition}
\newtheorem{lemma}[theorem]{Lemma}
\newtheorem{corollary}[theorem]{Corollary}
\theoremstyle{definition}
\newtheorem{definition}[theorem]{Definition}
\newtheorem{example}[theorem]{Example}
\newtheorem{remark}[theorem]{Remark}

\newcommand{\gammaice}[6]{
\begin{tikzpicture}[scale=0.77, every node/.style={scale=0.8}]
\coordinate (a) at (-.75, 0);
\coordinate (b) at (0, .75);
\coordinate (c) at (.75, 0);
\coordinate (d) at (0, -.75);
\coordinate (aa) at (-.75,.5);
\coordinate (cc) at (.75,.5);
\draw (a)--(c);
\draw (b)--(d);
\draw[fill=white] (a) circle (.25);
\draw[fill=white] (b) circle (.25);
\draw[fill=white] (c) circle (.25);
\draw[fill=white] (d) circle (.25);
\node at (0,1) { };
\node at (a) {$#1$};
\node at (b) {$#2$};
\node at (c) {$#3$};
\node at (d) {$#4$};
\node at (aa) {$#5$};
\node at (cc) {$#6$};
\end{tikzpicture}}
\title{Frozen Pipes: Lattice Models for Grothendieck Polynomials}
\author{Ben Brubaker, Claire Frechette, Andrew Hardt, Emily Tibor, and Katherine Weber}
\date{\today}

\begin{document}

\maketitle

\begin{abstract} We introduce families of two-parameter multivariate polynomials indexed by pairs of partitions $v,w$ -- {\it biaxial double}
 $(\beta,q)$-{\it Grothendieck polynomials} -- which specialize at $q=0$ and $v=1$ to double $\beta$-Grothendieck polynomials from torus-equivariant connective K-theory. Initially defined recursively via divided difference operators, our main result is that these new polynomials arise as partition functions of solvable lattice models. Moreover, the associated quantum group of the solvable model for polynomials in $n$ pairs of variables is a Drinfeld twist of the $U_q(\widehat{\mathfrak{sl}}_{n+1})$ $R$-matrix. By leveraging the resulting Yang-Baxter equations of the lattice model, we show that these polynomials simultaneously generalize double $\beta$-Grothendieck polynomials and dual double $\beta$-Grothendieck polynomials for arbitrary permutations. We then use properties of the model and Yang-Baxter equations to reprove Fomin-Kirillov's Cauchy identity for $\beta$-Grothendieck polynomials, generalize it to a new Cauchy identity for biaxial double $\beta$-Grothendieck polynomials, and prove a new branching rule for double $\beta$-Grothendieck polynomials.
\end{abstract}

\section{Introduction}\label{intro}

In statistical mechanics, solvable lattice models are used to infer global behavior of a system from the local properties of nearest-neighbor particle interactions. Here we use the adjective ``solvable'' to mean that there exists a family of (quantum) Yang-Baxter equations which allow one to \emph{solve the model}; that is, to determine sufficiently many recursive relations to explicitly evaluate the generating function made from weighted sums of potential particle configurations. This generating function is called the \emph{partition function} of the model and it is summed over certain \emph{admissible states}. This terminology and additional tools for (two-dimensional) solvable lattice models are reviewed at the outsets of Sections~\ref{Pipes} and~\ref{solvability}. 

A growing collection of recent papers (see for example \cite{Zinn-Justin-LRSchur, ZJ-Wheeler, KnutsonZJ-schubert-puzzles, KnutsonZJ-motivic, MotegiSakai, BW-coloured, BBBGcolor}) has demonstrated the utility of solvable lattice models to represent important polynomial functions in Schubert calculus and its many generalizations. Roughly speaking, solvable lattice models are effective tools because, on one hand, the admissible states are in bijection with certain combinatorial objects -- such as tableaux, Gelfand-Tsetlin-type patterns, or, in our case, pipe dreams -- and are thus equipped with a rich set of tools for studying algebraic invariants. On the other hand, their solutions to the Yang-Baxter equation, called \emph{$R$-matrices}, are known to satisfy braid and quadratic relations which give rise to an action of a Hecke algebra on the lattice model. As we will see, manipulating the lattice model according to a certain diagrammatic calculus allows one to move between these interpretations, via a sort of graphical analogue of Schur-Weyl duality, leading to new identities as well as new proofs of known identities involving the partition function. In the present paper, we add to the list by showing that the double $\beta$-Grothendieck polynomials and their duals are expressible as partition functions of solvable lattice models, and we present a one-parameter deformation of both of them to properly explain the solvability in terms of associated quantum groups. We then use these models and their Yang-Baxter equations to prove Cauchy-type identities and a branching rule. 

Before describing our results in more detail, we give a concise overview of Grothendieck polynomials, with precise definitions to follow in Section~\ref{grotpoly}. 
Let $X$ be the set of complete flags in $\mathbb{C}^n$, a smooth, projective
complex variety with an action of $\GL_n(\mathbb{C})$ induced from the standard action
on complex vectors. Let $K(X)$ denote the Grothendieck ring of algebraic vector bundles over $X$.
It has an additive basis given by $K$-theoretic Schubert classes $[\mathcal{O}_{X_w}]$ with $w \in S_n$.
Here $X_w$ denotes the closure of a corresponding $B$-orbit in $X$, where $B$ denotes the Borel subgroup of lower
triangular matrices in $\GL_n(\mathbb{C})$, and $\mathcal{O}_X$ is the structure sheaf of $X$.
The generalized Littlewood-Richardson problem is to determine an explicit formula for the structure constants for multiplication with respect to this basis.
That is, one would like to determine the integer coefficients $C_{u,v}^w$ appearing in
$$ [\mathcal{O}_{X_u}] \cdot [\mathcal{O}_{X_v}] = \sum_{w \in S_n} C_{u,v}^w [\mathcal{O}_{X_w}], $$ 
analogous to the Littlewood-Richardson coefficients for the ordinary cohomology of $X$. Brion \cite{Brion-positivity} famously showed that these
structure constants had a predictable sign. If $\ell(w)$ denotes the length of $w$ as a reduced word in simple reflections,
then he showed that $(-1)^{\ell(w) - \ell(u) - \ell(v)} C_{u,v}^w$ is non-negative for all $u,v,w$ in $S_n$.

Lascoux and Sch\"utzenberger \cite{LS-symmetry} introduced polynomial representatives for these classes known as Grothendieck polynomials, denoted
$\mathcal{G}_w$ to any permutation $w$. That is, Grothendieck polynomials model the multiplication in $K(X)$ according to
$$  \mathcal{G}_u \cdot \mathcal{G}_v = \sum_{w \in S_n} C_{u,v}^w \mathcal{G}_w. $$
These were later generalized by Fomin and Kirillov \cite{FominKirillov-gpybe}, who defined $\beta$-Grothendieck polynomials depending on a deformation parameter $\beta$, which simultaneously
generalize the prior Grothendieck polynomials ($\beta=-1$) and Schubert polynomials ($\beta = 0$). We will denote these by $\mathcal{G}^{(\beta)}_w$.

Much later, a geometric interpretation of $\beta$-Grothendieck polynomials was described by Hudson in \cite{Hudson}.  In fact, Hudson works in slightly greater generality: just as Schubert and Grothendieck polynomials admit generalizations to two sets of variables, there are also \emph{double} $\beta$-Grothendieck polynomials. These represent connective $K$-theory classes of degeneracy loci of flag bundles, or, equivalently, Schubert classes in the $T$-equivariant connective $K$-theory ring of the flag variety, where $T$ is the torus of diagonal matrices. (See e.g. \cite{anderson2012introduction} for a nice example of how results about degeneracy loci translate into statements about equivariant cohomology.) As is also true in the single variable case, these have a natural recursive description in terms of pullback and pushforward maps in the Bott-Samelson resolution of singularities in the flag variety, which leads to the definition of $\beta$-Grothendieck polynomials in terms of divided difference operators described in Section~\ref{grotpoly}, and which will manifest in the action of $R$-matrices on our lattices in later sections.

One of the most striking examples of the interplay between the Yang-Baxter equation and combinatorial identities is the paper \cite{Zinn-Justin-LRSchur} of Zinn-Justin which, by interpreting Knuston-Tao puzzles in the context of solvable lattice models, gives a new proof of the Littlewood-Richardson rule for the structure constants of Schur functions. This method was subsequently generalized by Wheeler and Zinn-Justin \cite{ZJ-Wheeler} to Grassmannian Grothendieck polynomials, and by Knutson and Zinn-Justin to Grothendieck polynomials whose associated permutations have few descents \cite{KnutsonZJ-schubert-puzzles}. As noted above, finding and proving a Littlewood-Richardson type rule for the structure constants of general ($\beta$-) Grothendieck polynomials is a long-standing open problem, and based on the success in the certain special cases described above, one might hope that this general problem can also be approached from the lattice models defined in the present paper. 

Indeed, our lattice models are defined for arbitrary permutations, and therefore in the ``crystal limit'' $q=0$ specialize to Grothendieck polynomials for arbitrary permutations. The key insight to treating arbitrary permutations is to use a generalization of the ice-type \textit{six-vertex model} inspired by earlier work of Borodin and Wheeler \cite{BW-coloured}. Unlike Borodin-Wheeler, who use color to refine certain partition functions of lattice models from symmetric functions into their nonsymmetric pieces, our models use color to move from permutations with one descent to those with arbitrarily many descents. As a nod to this similarity, we refer to our models as ``chromatic'' rather than ``colo(u)red.'' This distinction is also seen in the quantum interpretation of these two approaches: the associated quantum group module for our solutions of the Yang-Baxter equation is a Drinfeld twist of the standard $U_q(\widehat{\mathfrak{sl}}_{n+1})$ module, while those in~\cite{BW-coloured} arise from symmetric powers of $U_q(\widehat{\mathfrak{sl}}_{2})$ modules. 

We conclude by outlining the structure of the subsequent sections. In Section~\ref{grotpoly}, we define several new families of polynomials, including our most general $q$-deformed ``biaxial'' polynomials, which specialize to both double $\beta$-Grothendieck polynomials and their duals. In that section, we also present two different definitions for double $\beta$-Grothendieck polynomials which represent the Hecke algebra point of view and the combinatorial point of view in terms of generating functions discussed earlier. Then in Section~\ref{Pipes}, we introduce our lattice model, whose partition function we will later show is the biaxial polynomial, and show how to obtain the weights from a Drinfeld twist of the standard $U_q(\widehat{\mathfrak{sl}}_{n+1})$ module. In Section~\ref{solvability}, we describe the Yang-Baxter equations and $R$-matrices associated to the model, as well as a ``rhombus'' Yang-Baxter equation associated to mixing two different orientations of the model.

Then in Section~\ref{modelproof}, we use the solvability of the model to prove the following main result. 

\begin{maintheorem}
The biaxial double $(\beta,q)$-Grothendieck polynomials in Definition \ref{biaxial-polynomials} are realized as the partition function of a solvable lattice model (see Theorem \ref{modelgivesgrotpolys}). Furthermore, they simultaneously generalize both the double $\beta$-Grothendieck and double dual $\beta$-Grothendieck polynomials (see Theorems \ref{modelgivesgrotpolys} and~\ref{modelgivesdual}).
\end{maintheorem}

This allows us to find alternate recursive definitions for both the $\beta$-Grothendieck polynomials and their duals and show that the biaxial polynomials specialize to both of these sets. The last three sections concern these specializations of the model. In Section~\ref{pipedreamcorrespondencesection}, we move to the generating function approach and describe the correspondence between the states of the models and pipe dreams. In Section~\ref{cauchy}, we prove Cauchy-type identities by stacking our models appropriately and calculating the partition functions in two ways. Finally, in Section~\ref{branching}, we provide a branching rule that describes how to express a double $\beta$-Grothendieck polynomial for any permutation in $S_n$ as a sum over double $\beta$-Grothendieck polynomials for permutations in $S_{n-1}.$

\bigskip

\noindent {\bf Acknowledgements:} This work was supported by NSF grants DMS-1801527 and DMS-2101392 and NSF RTG grant DMS-1745638. We thank Valentin Buciumas and Travis Scrimshaw for comments on an earlier draft of the paper, and Daniel Bump and Henrik Gustafsson for making code available to us that was useful in computing solutions to Yang-Baxter equations.

\bigskip

\noindent {\bf Authors' Note:} In the final stages of the first version of this paper ({\tt arXiv:2007.04310}), we were informed of another, independent work by Buciumas and Scrimshaw~\cite{BuciumasScrimshaw-bumpless-model} which provides solvable lattice models for Grothendieck polynomials. Their models are naturally connected to so-called ``bumpless'' pipe dreams of Lam, Lee, and Shimozono \cite{LamLeeShimozono-bumpless}. So while their Boltzmann weights look similar in nature, there are important differences which prevent direct comparison of their lattice model to ours. Indeed, it is likewise difficult on the combinatorial generating function side to compare Fomin-Kirillov pipe dreams to bumpless ones. This provides yet another example of the mutable nature of these solvable lattice models - the methods are flexible enough to apply to a wide array of special functions, but the analysis of each particular model presents unique challenges.

\section{Grothendieck Polynomials and Their Generalizations}\label{grotpoly}

In this section, we build up to an inductive definition of the biaxial double $(\beta,q)$-Grothendieck polynomials via divided difference operators. These polynomials are a $q$-deformation of another new family of polynomials, the biaxial double $\beta$-Grothendieck polynomials, which in turn are a simultaneous generalization of double $\beta$-Grothendieck polynomials and dual double $\beta$-Grothendieck polynomials.

In the process, we review the special case of (double $\beta$-)Grothendieck polynomials as in the original definition of Fomin and Kirillov~\cite{FominKirillov-gpybe}. Let $\beta$ and $q$ be formal parameters and let $\boldsymbol{x} = (x_1, \ldots, x_n)$ and $\boldsymbol{y} = (y_1, \ldots, y_n)$ denote $n$-tuples of formal variables. Then for the polynomial ring $\mathbb{C}(\beta, q)[\boldsymbol{x}, \boldsymbol{y}] := \mathbb{C}(\beta, q)[x_1, \ldots, x_n, y_1, \ldots, y_n]$ we define the divided difference operators $\pi_i := \pi_i^{(\beta,q)}$ by
\begin{equation} \pi_{i}^{(\beta,q)}(f)(\boldsymbol{x},\boldsymbol{y}) = \frac{ (1-q^2)(1+\beta x_{i+1})f(\boldsymbol{x},\boldsymbol{y}) - (1+\beta x_i -q^2 - q^2\beta x_{i+1})f(s_i\boldsymbol{x},\boldsymbol{y})}{x_i - x_{i+1}} \label{isobaricdd} \end{equation}
for $i=1, \ldots, n-1.$ Here $s_i$ denotes the simple reflection in $S_n$ acting on $\boldsymbol{x}$ by interchanging the variables $x_i$ and $x_{i+1}$. Recalling the Newton divided difference operator $\partial_i$ given by
\[ \partial_i(f)(\boldsymbol{x}) = \frac{f(\boldsymbol{x}) - f(s_i \boldsymbol{x})}{x_i - x_{i+1}}, \]
we may write the identity of operators
\[ \pi_i^{(\beta,q)} = \partial_i^x (1+\beta x_{i+1}) - q^2 (1+\beta x_{i+1}) \partial_i^x, \]
where we've emphasized in the notation $\partial_i^x$ that the operator $\partial_i$ acts only on the variables $\boldsymbol{x}$ for any polynomial $f(\boldsymbol{x}, \boldsymbol{y}).$
There is a second corresponding set of operators $\twid{\pi}_i := \twid{\pi}_i^{(\beta, q)}$ for $i=1,\ldots,n$ defined by
\[\twid{\pi}_i^{(\beta,q)}(f) = \frac{1}{\beta^2 q^2} \left( \partial_i^y (1+\beta y_{i+1}) - q^2 (1+\beta y_{i+1}) \partial_i^y \right), \]
where $\partial_i^y$ now acts on the $\boldsymbol{y}$ variables, paralleling the operator $\pi_i$ up to normalization. 

Recalling that $\partial_i^2 = 0$, $x_i \partial_i = \partial_i x_{i+1} + 1$, and the braid relations for $\partial_i$, it is straightforward to check that the operators $\pi_i$ and $\twid{\pi}_i$ also satisfy the braid relations. For example, in the case of $\pi_i$:
\[ \pi_i \pi_{i+1} \pi_i = \pi_{i+1} \pi_i \pi_{i+1} \quad \textrm{and} \quad \pi_i \pi_j = \pi_j \pi_i \quad \textrm{for all $i,j$ with $|i-j|>1$.} \] Additionally, a somewhat-lengthy computation shows that each $\pi_i$ and $\twid{\pi}_i$ satisfy the quadratic relations \[(\pi_i+\beta)(\pi_i-\beta q^2)f = 0, \hspace{30pt} (\twid{\pi}_i + \beta^{-1}q^{-2})(\twid{\pi}_i-\beta^{-1})f = 0\] for any function $f\in\mathbb{C}(\beta,q)[\boldsymbol{x},\boldsymbol{y}]$. Therefore, the operators $\{\pi_i\}_{i=1}^{n-1}$ and the operators $\{\twid{\pi}_i\}_{i=1}^{n-1}$ each generate Hecke algebras. Section~2 of \cite{FominKirillov-schubert} explores solutions to the Yang-Baxter equation via such Hecke algebras as we vary over certain specializations of quadratic relations of the form $\pi_i^2 = a \pi_i + b$ (in particular when $a=0$ and $b=0$). As we will show in subsequent sections, our operators also arise from solutions to Yang-Baxter equations, despite their complicated quadratic relations. This connection will arise in a much more circuitous way by making use of solvable lattice models to demonstrate this fact and derive further applications.

For later use, we also record the explicit inverses of each operator: \[\pi^{-1}_{i}(f) = \frac{1}{\beta^2 q^2} \left( (1+\beta x_i)\partial_i^x - q^2 \partial_i^x (1+\beta x_i) \right) \hspace{18pt} \text{and} \hspace{18pt} \twid{\pi}^{-1}_{i}(f) = \left( (1+\beta y_i)\partial_i^y - q^2 \partial_i^y (1+\beta y_i) \right).\]

We may now define double $(\beta,q)$-Grothendieck 
polynomials recursively with respect to length $\ell$ in the symmetric group, beginning from the long element $w_0$. We often use the following notation, suggesting a formal group law addition but suppressing the dependence on $\beta$, in our definitions:
$$ x \oplus y := x + y + \beta xy. $$

\begin{definition} Given a fixed positive integer $n$ and a permutation $w$, define the {\it double $(\beta,q)$-Grothendieck polynomials} $\mathcal{G}^{(\beta,q)}_w(\boldsymbol{x}; \boldsymbol{y}) := \mathcal{G}^{(\beta,q)}_w(x_1, \ldots, x_n; y_1, \ldots, y_n)$ as follows. \begin{itemize}
\item Set $$ \mathcal{G}^{(\beta,q)}_{w_0}(\boldsymbol{x};\boldsymbol{y}) = (1-q^2)^n \prod_{i+j< n+1}(x_i \oplus y_j) \prod_{i+j>n+1} (1-q^2(1+\beta (x_i\oplus y_j))).$$
\item For any $w$ and a simple reflection $s_i = (i \; i+1)$ in $S_n$ such that $\ell(ws_i) = \ell(w) - 1$, set
$$ \mathcal{G}^{(\beta,q)}_{ws_i}(\boldsymbol{x}; \boldsymbol{y}) = \pi_i^{(\beta,q)} \mathcal{G}^{(\beta,q)}_{w} (\boldsymbol{x}; \boldsymbol{y}). $$
\end{itemize}
\label{divdiffdef}
\end{definition}
Note the above polynomials are well-defined since the operators $\pi_i^{(\beta,q)}$ satisfy the braid relations above. 

Setting $q=0$ in both bullets above recovers the formula for double $\beta$-Grothendieck polynomials, as in the original definition in Fomin and Kirillov~\cite{FominKirillov-gpybe}. These authors were motivated, in part, by attempts to classify exponential solutions to the Yang-Baxter equation. Their definition of $\mathcal{G}^{(\beta)}_w$ is given in terms of $\beta$-deformed divided difference operators $\pi_i^{(\beta)} := \pi_{i}^{(\beta,0)}$, obtained from setting $q=0$ in (\ref{isobaricdd}).
Once again, we may succinctly write $\pi_i^{(\beta)} = \partial_i \circ (1+\beta x_{i+1})$ where $\partial_i$ is the usual divided difference operator. The single variable $\beta$-Grothendieck polynomials are recovered by the further specialization $\mathcal{G}^{(\beta)}_{w} (\boldsymbol{x}; 0)$. See the appendix to \cite{Hudson} for a nice exposition of the relation between Fomin-Kirillov's definition of double Grothendieck polynomials via the exponential Yang-Baxter equation solutions and the divided difference operator definition given above.

We caution the reader that other notions of Grothendieck polynomials appear in the literature. For example in \cite{Kirillov-notes, Kirillov-quantum}, Kirillov sets $\mathcal{G}^{(\beta)}_{w_0}(\boldsymbol{x};\boldsymbol{y})$ equal to $\prod_{i+j \leq n} (x_i + y_j)$, using the additive formal group as opposed to our multiplicative one depending on $\beta$. Moreover, many authors work only with {\it Grassmannian} Grothendieck polynomials, for which the corresponding permutation $w$ has a unique descent and so may be recorded as a partition, and the adjective ``Grassmannian'' is often dropped for brevity.

A second definition of $\beta$-Grothendieck polynomials as a weighted sum over a combinatorially defined set may be given in terms of ``pipe dreams.'' Fomin and Kirillov \cite{FominKirillov-schubert} defined reduced pipe dreams (under the name ``generalized configurations'') to represent products in a Yang-Baxter algebra, and used them to study Schubert polynomials. Reduced pipe dreams are also known as rc-graphs, as in \cite{BergeronBilley-rc}.

Knutson and Miller \cite{KnutsonMiller-geometry,KnutsonMiller-subwords} coined the name ``pipe dreams,'' defined nonreduced pipe dreams, and gave a formula for Grothendieck polynomials in terms of these more general objects. Lenart, Robinson, and Sottile \cite{LenartRobinsonSottile} gave a correspondence between pipe dreams and chains in the $k$-Bruhat order, which they used to find formulas for both $\mathcal{G}_w$ and $\mathcal{H}_w$.

We mostly follow \cite{PechenikSearles} for our definition. A pipe dream is a tiling of an $n\times n$ grid with the tiles \raisebox{-1 ex}{\begin{tikzpicture}[scale=.5]
    %grid
    \draw (0,0) -- (1,0);
    \draw (0,1) -- (1,1);
    \draw (0,0) -- (0,1);
    \draw (1,0) -- (1,1);
    %pipes
    \draw[rounded corners = 2mm, line width = .5mm] (0,.5) -- (.5,.5)-- (.5,1);
    \draw[rounded corners = 2mm, line width = .5mm] (.5,0) --(.5,.5)-- (1,.5);
\end{tikzpicture}}
and 
\raisebox{-1 ex}{\begin{tikzpicture}[scale=.5]
    %grid
    \draw (0,0) -- (1,0);
    \draw (0,1) -- (1,1);
    \draw (0,0) -- (0,1);
    \draw (1,0) -- (1,1);
    %pipes
    \draw[rounded corners = 2mm, line width = .5mm] (0,.5)--(1,.5);
    \draw[rounded corners = 2mm, line width = .5mm] (.5,0) --(.5,1);
\end{tikzpicture}}, such that every 
\raisebox{-1 ex}{\begin{tikzpicture}[scale=.5]
    %grid
    \draw (0,0) -- (1,0);
    \draw (0,1) -- (1,1);
    \draw (0,0) -- (0,1);
    \draw (1,0) -- (1,1);
    %pipes
    \draw[rounded corners = 2mm, line width = .5mm] (0,.5)--(1,.5);
    \draw[rounded corners = 2mm, line width = .5mm] (.5,0) --(.5,1);
\end{tikzpicture}}
appears above the anti-diagonal. Any undrawn tiles are assumed to be of the type 
\raisebox{-1 ex}{\begin{tikzpicture}[scale=.5]
    %grid
    \draw (0,0) -- (1,0);
    \draw (0,1) -- (1,1);
    \draw (0,0) -- (0,1);
    \draw (1,0) -- (1,1);
    %pipes
    \draw[rounded corners = 2mm, line width = .5mm] (0,.5) -- (.5,.5)-- (.5,1);
    \draw[rounded corners = 2mm, line width = .5mm] (.5,0) --(.5,.5)-- (1,.5);
\end{tikzpicture}}. 
These tiles form a collection of pipes originating on the left boundary of the grid and terminating along the top boundary. We number the pipes according to the row of their origination, increasing from top to bottom. A pipe dream is said to be {\it reduced} if no two pipes cross more than once. Given a nonreduced pipe dream $P$ we can form an associated reduced pipe dream, called the {\it reduction of $P$}, in the following way: pipe by pipe, starting with pipe 1, remove any multiple crossings with another pipe after the initial crossing. We denoted the reduction of $P$ by $\text{red}(P)$. The number of crossings removed in this process is the {\it excess of $P$}, written $\text{ex}(P).$ 

To each pipe dream we can associate a permutation. The permutation of a reduced pipe dream is the permutation sending $i$ to the column (indexed in increasing order from left to right) of the termination of pipe $i$. The permutation of a nonreduced pipe dream is the permutation associated to its reduction; it is often difficult to read the permutation from a nonreduced pipe dream. Given a permutation $w$, let $PD(w)$ be the collection of all pipe dreams with permutation $w$. The weight of a pipe dream $P$ is $\text{wt}(P) := \prod (x_i\oplus y_j) $, where the product is over all \raisebox{-1 ex}{\begin{tikzpicture}[scale=.5]
    %grid
    \draw (0,0) -- (1,0);
    \draw (0,1) -- (1,1);
    \draw (0,0) -- (0,1);
    \draw (1,0) -- (1,1);
    %pipes
    \draw[rounded corners = 2mm, line width = .5mm] (0,.5)--(1,.5);
    \draw[rounded corners = 2mm, line width = .5mm] (.5,0) --(.5,1);
\end{tikzpicture}} tiles, and the variable subscripts $i$ (resp. $j$) refer to the row (resp. column) in which the crossing tile appears (see Figure \ref{pipedreamreduction}).

\begin{figure}[h]
\centering
\scalebox{.8}{\begin{tikzpicture}
%grid
\draw (0,0) -- (4,0);
\draw (0,1) -- (4,1);
\draw (0,2) -- (4,2);
\draw (0,3) -- (4,3);
\draw (0,4) -- (4,4);
\draw (0,0) -- (0,4);
\draw (1,0) -- (1,4);
\draw (2,0) -- (2,4);
\draw (3,0) -- (3,4);
\draw (4,0) -- (4,4);
%pipes
\draw[rounded corners = 3mm, line width = .5mm] (0,3.5) -- (.5,3.5) -- (.5, 4);
\draw[rounded corners= 3mm, line width = .5mm] (0, 2.5) -- (1.5,2.5) -- (1.5,4);
\draw[rounded corners = 3mm, line width = .5mm] (0,1.5) -- (1.5,1.5) -- (1.5,2.5)  -- (2.5,2.5) -- (2.5,4);
\draw[rounded corners = 3mm, line width = .5mm] (0,.5) -- (.5,.5) -- (.5,3.5)  -- (3.5,3.5) -- (3.5,4);
\end{tikzpicture}}
\hspace{1cm}
\scalebox{.8}{\begin{tikzpicture}
%grid
\draw (0,0) -- (4,0);
\draw (0,1) -- (4,1);
\draw (0,2) -- (4,2);
\draw (0,3) -- (4,3);
\draw (0,4) -- (4,4);
\draw (0,0) -- (0,4);
\draw (1,0) -- (1,4);
\draw (2,0) -- (2,4);
\draw (3,0) -- (3,4);
\draw (4,0) -- (4,4);
%pipes
\draw[rounded corners = 3mm, line width = .5mm] (0,3.5) -- (.5,3.5) -- (.5, 4);
\draw[rounded corners= 3mm, line width = .5mm] (0, 2.5) -- (1.5,2.5) -- (1.5,3.5)--(3.5,3.5)--(3.5,4);
\draw[rounded corners = 3mm, line width = .5mm] (0,1.5) -- (1.5,1.5) -- (1.5,2.5)  -- (2.5,2.5) -- (2.5,4);
\draw[rounded corners = 3mm, line width = .5mm] (0,.5) -- (.5,.5) -- (.5,3.5)  -- (1.5,3.5) -- (1.5,4);
\end{tikzpicture}}
\caption{Two pipe dreams that both correspond to the permutation $w=1432.$ On the left is an unreduced pipe dream, and on the right is its reduction. The weight of the nonreduced pipe dream  is $(x_1\oplus y_2)(x_1\oplus y_3)(x_2\oplus y_1)(x_3\oplus y_1)$, and its excess is 1. The weight of the reduced pipe dream is $(x_1\oplus y_3)(x_2\oplus y_1)(x_3\oplus y_1)$.}
    \label{pipedreamreduction}
\end{figure}
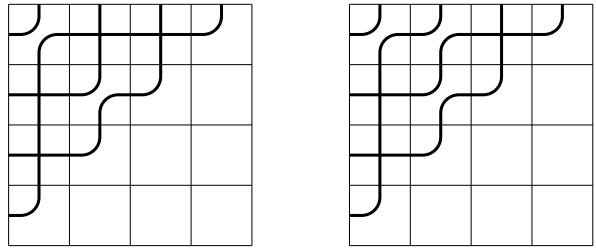

It is shown (see for example \cite{KnutsonMiller-subwords,LenartRobinsonSottile}) that double $\beta$-Grothendieck polynomials can be written as a sum over pipe dreams using the statistics defined above: \begin{equation} \mathcal{G}_w^{(\beta)}(\boldsymbol{x};\boldsymbol{y}) = \sum_{P\in PD(w)} \text{wt}(P) \beta^{\text{ex}(P)}. \label{pipedreamequation} \end{equation}

\begin{remark}
Our definition of pipe dreams is actually a slight generalization of others we've found in the literature since our pipe dreams have both row and column variables, and general $\beta$. The formula in \cite{LenartRobinsonSottile} has $\beta=-1$, and the formula in \cite{PechenikSearles} has only row variables $\boldsymbol{x}$. We have not been able to find the formula (\ref{pipedreamequation}) in the literature for double $\beta$-Grothendieck polynomials, although surely it is well-known. Our Proposition \ref{pipedreamproposition} can be taken as a proof of (\ref{pipedreamequation}).
\end{remark}

\begin{remark}
Pipe dreams can be seen as combinatorial realizations of products in the Yang-Baxter algebras of \cite{FominKirillov-schubert,FominKirillov-gpybe}. Each tile represents a generator of the Yang-Baxter algebra, and \raisebox{-1ex}{\begin{tikzpicture}[scale=.5]
    %grid
    \draw (0,0) -- (1,0);
    \draw (0,1) -- (1,1);
    \draw (0,0) -- (0,1);
    \draw (1,0) -- (1,1);
    %pipes
    \draw[rounded corners = 2mm, line width = .5mm] (0,.5)--(1,.5);
    \draw[rounded corners = 2mm, line width = .5mm] (.5,0) --(.5,1);
\end{tikzpicture}} tiles correspond to generators of a 0-Hecke algebra. Reduction of a pipe dream corresponds to applying quadratic relations in the 0-Hecke algebra, and the braid relations in the Yang-Baxter algebra ensure that the reduction of a pipe dream is well-defined and associated to a unique permutation.
\end{remark}

A different approach to the one we use here is to consider {\it bumpless pipe dreams}, which are pipe dreams with some extra allowed tiles. These were introduced by Lam, Lee, and Shimozono \cite{LamLeeShimozono-bumpless}. Lascoux \cite{Lascoux-chern-yang} proved a formula for Grothendieck polynomials as a sum over alternating sign matrices. Weigandt \cite{Weigandt-bumpless} then showed that Lascoux's formula implies a summation formula similar to (\ref{pipedreamequation}) for double $\beta$-Grothendieck polynomials in terms of bumpless pipe dreams. However, there is no weight-preserving bijection between Weigandt's formula and (\ref{pipedreamequation}) in the double set of variables.

These bumpless pipe dreams and the work of Weigandt were the motivation for the aforementioned contemporaneous work of Buciumas and Scrimshaw~\cite{BuciumasScrimshaw-bumpless-model} on lattice models for Grothendieck polynomials. As our model will be seen to relate to (ordinary) pipe dreams, this is a further indication of the fundamental difference between the models that will appear here and those of~\cite{BuciumasScrimshaw-bumpless-model}. 

We conclude this section by defining a set of ``dual'' polynomials, and a generalized set of ``biaxial'' polynomials.

\begin{definition} Given a fixed positive integer $n$ and a permutation $w$, define the {\it double dual $(\beta,q)$-Grothendieck polynomials} $\mathcal{H}^{(\beta,q)}_w(\boldsymbol{x}; \boldsymbol{y}) := \mathcal{H}^{(\beta,q)}_w(x_1, \ldots, x_n; y_1, \ldots, y_n)$ as follows. \begin{itemize}
\item Set $$ \mathcal{H}^{(\beta,q)}_{w_0}(\boldsymbol{x};\boldsymbol{y}) = (1-q^2)^n \prod_{i+j< n+1}(x_i \oplus y_j) \prod_{i+j>n+1} (1-q^2(1+\beta (x_i\oplus y_j))).$$
\item For any $w$ and a simple reflection $s_i = (i \; i+1)$ in $S_n$ such that $\ell(s_iw) = \ell(w) - 1$, set
$$ \mathcal{H}^{(\beta,q)}_{s_iw}(\boldsymbol{x}; \boldsymbol{y}) = \twid{\pi_i}^{-1} \mathcal{H}^{(\beta,q)}_{w} (\boldsymbol{x}; \boldsymbol{y}). $$
\end{itemize}
\label{dualdef}
\end{definition}

We will later show that when $q=0$ our dual polynomials specialize to the dual $\beta$-Grothendieck polynomials defined by Lascoux and Sch\"utzenberger.

\begin{definition}[Lascoux-Sch\"utzenberger \cite{LS-symmetry}] \label{dual-polynomials-definition} Given a fixed positive integer $n$ and a permutation $w$, define the {\it dual double $\beta$-Grothendieck polynomials} $\mathcal{H}^{(\beta)}_w(\boldsymbol{x}; \boldsymbol{y})$ recursively by setting $\mathcal{H}^{(\beta)}_{w_0} = \mathcal{G}^{(\beta)}_{w_0}$.

Then, for any $w$ and a simple reflection $s_i = (i \; i+1)$ in $S_n$ such that $\ell(ws_i) = \ell(w) - 1$, set
$$ \mathcal{H}^{(\beta)}_{ws_i}(\boldsymbol{x}; \boldsymbol{y}) = \mu_i \, \mathcal{H}^{(\beta)}_{w} (\boldsymbol{x}; \boldsymbol{y}) $$
where $\mu_i:= (1+\beta x_i)\partial_i$. 
\end{definition}

Note that this action is a right action involving the variable $x$, as opposed to the left action using the variable $y$ involved in Definition \ref{dualdef}. Thus, showing that the double dual $(\beta,q)$-Grothendieck polynomials indeed $q$-deform the dual $\beta$-Grothendieck polynomials requires additional information on the symmetries of these polynomials (see Section 5).

These dual $\beta$-Grothendieck polynomials are sometimes referred to as $\mathcal{H}$-polynomials. 
Equivalently, we may define them by the formula
$$ \mathcal{H}^{(\beta)}_w(\boldsymbol{x}; \boldsymbol{y}) := \sum_{w_0 \geq v \geq w} \beta^{\ell(v)-\ell(w)} \mathcal{G}^{(\beta)}_v(\boldsymbol{x}; \boldsymbol{y}). $$

The $\mathcal{H}$-polynomials may be seen to be adjoint to the $\mathcal{G}^{(\beta)}_{w}$ with respect to an inner product on the Grothendieck ring $K(X)$ defined in terms of the $\pi_i^{(\beta)}$ operators (see Equation 7.1 in~\cite{LS-symmetry}). Strictly speaking, this exact definition seems to have not appeared explicitly in the literature, but see Section~6 of \cite{LenartRobinsonSottile} for a detailed survey of their properties in the special case $\beta = -1$ and \cite{Kirillov-notes} for a definition that agrees except for the initial $\mathcal{H}^{(\beta)}_{w_0}$.

We define the biaxial polynomials to be a generalization of the double $(\beta,q)$-Grothendieck polynomials. They will turn out to generalize the dual polynomials as well.

\begin{definition}
Given a fixed positive integer $n$ and two permutations $v,w$, define the {\it biaxial double $(\beta,q)$-Grothendieck polynomials} $\mathcal{G}^{(\beta,q)}_{v,w}(\boldsymbol{x}; \boldsymbol{y}) := \mathcal{G}^{(\beta,q)}_{v,w}(x_1, \ldots, x_n; y_1, \ldots, y_n)$ as follows. \begin{itemize}
\item Set $$\mathcal{G}^{(\beta,q)}_{1,w}(\boldsymbol{x};\boldsymbol{y}) = \mathcal{G}^{(\beta,q)}_w(\boldsymbol{x};\boldsymbol{y}).$$
\item If $s_i$ is a simple reflection with $\ell(vs_i) = \ell(v)+1$, set
$$ \mathcal{G}^{(\beta,q)}_{vs_i,w}(\boldsymbol{x}; \boldsymbol{y}) = \twid{\pi}_i^{-1}\mathcal{G}^{(\beta,q)}_{v,w}(\boldsymbol{x}; \boldsymbol{y}). $$
\end{itemize}
\label{biaxial-polynomials}
\end{definition}

These polynomials will turn out to be partition functions of our model with two nontrivial boundary axes; hence the name. Note that the operators $\pi$ and $\twid{\pi}$ commute, so we can apply $\pi$ and $\twid{\pi}^{-1}$ in any order we wish to obtain $\mathcal{G}^{(\beta,q)}_{v,w}$ from $\mathcal{G}^{(\beta,q)}_{1,w_0}$.

\section{Constructing Chromatic Lattice Models}\label{Pipes}

In this section, we define a \emph{chromatic lattice model} whose partition function gives a polynomial family that simultaneously generalizes (double) $\beta$-Grothendieck polynomials and their duals. In Section \ref{modelproof}, we will demonstrate that the solvability of this model will provide the connection to $\beta$-Grothendieck polynomials via Demazure operators. Before detailing the specifics of the model, we briefly review the terminology associated to lattice models.

Given a fixed positive integer $n$, we form a square grid with $n$ rows and $n$ columns, whose intersection points are called {\it vertices}. In particular, each vertex has four adjacent edges. The vertices are labeled by parameters $x_i,y_j$ depending on the row ($i$) and column ($j$) in which they appear. Columns will always be numbered from left to right starting with column 1, and rows from top to bottom starting with row 1. (For diagrammatic readability and to nod to their potential quantum group connections, we will often abbreviate the vertex labels $x_i,y_j$ as $T_{i,j}$.) Each edge has a label taken from the set $\{+,1,2,...,n\}$. This strange choice of label set comes from its connection to the six-vertex model, where edges are labeled with $+$ or $-$ and we think of $\{1,\ldots,n\}$ as expanding the set of possible labels on six-vertex model states formerly labeled with a $-$. In figures, the labels $\{1,\ldots,n\}$ on edges will be depicted with colors while edges labeled with a $+$ label will be considered \emph{uncolored}. This is where the name ``chromatic lattice model'' comes from, and we may think of the classical six-vertex model as the special case of a monochrome (i.e., one color) model. Color has been used before to generalize lattice models, and we caution the reader that the use of color here is different from the so-called ``colo(u)red lattice models'' in \cite{BW-coloured, BBBGcolor, BSWatoms}. While our diagrams appear visually similar to those of colored lattice models, the colored models refine the uncolored partition functions into smaller ``atoms,'' while our labels generalize the monochrome models. Thus we have chosen the adjective \emph{chromatic} to reflect this difference.

\begin{figure}[h]
\begin{center}
\scalebox{0.8}{
\begin{tikzpicture}
  \coordinate (ab) at (1,0);
  \coordinate (ad) at (3,0);
  \coordinate (af) at (5,0);
  \coordinate (ah) at (7,0);
  \coordinate (ba) at (0,1);
  \coordinate (bc) at (2,1);
  \coordinate (be) at (4,1);
  \coordinate (bg) at (6,1);
  \coordinate (bi) at (8,1);
  \coordinate (cb) at (1,2);
  \coordinate (cd) at (3,2);
  \coordinate (cf) at (5,2);
  \coordinate (ch) at (7,2);
  \coordinate (da) at (0,3);
  \coordinate (dc) at (2,3);
  \coordinate (de) at (4,3);
  \coordinate (dg) at (6,3);
  \coordinate (di) at (8,3);
  \coordinate (eb) at (1,4);
  \coordinate (ed) at (3,4);
  \coordinate (ef) at (5,4);
  \coordinate (eh) at (7,4);
  \coordinate (fa) at (0,5);
  \coordinate (fc) at (2,5);
  \coordinate (fe) at (4,5);
  \coordinate (fg) at (6,5);
  \coordinate (fi) at (8,5);
  \coordinate (gb) at (1,6);
  \coordinate (gd) at (3,6);
  \coordinate (gf) at (5,6);
  \coordinate (gh) at (7,6);
  \coordinate (ha) at (0,7);
  \coordinate (hc) at (2,7);
  \coordinate (he) at (4,7);
  \coordinate (hg) at (6,7);
  \coordinate (hi) at (8,7);
  \coordinate (ib) at (1,8);
  \coordinate (id) at (3,8);
  \coordinate (if) at (5,8);
  \coordinate (ih) at (7,8);
  \coordinate (bb) at (1,1);
  \coordinate (bd) at (3,1);
  \coordinate (bf) at (5,1);
  \coordinate (bh) at (7,1);
  \coordinate (db) at (1,3);
  \coordinate (dd) at (3,3);
  \coordinate (df) at (5,3);
  \coordinate (dh) at (7,3);
  \coordinate (fb) at (1,5);
  \coordinate (fd) at (3,5);
  \coordinate (ff) at (5,5);
  \coordinate (fh) at (7,5);
  \coordinate (hb) at (1,7);
  \coordinate (hd) at (3,7);
  \coordinate (hf) at (5,7);
  \coordinate (hh) at (7,7);
  \coordinate (bax) at (0,1.5);
  \coordinate (bcx) at (2,1.5);
  \coordinate (bex) at (4,1.5);
  \coordinate (bgx) at (6,1.5);
  \coordinate (bix) at (8,1.5);
  \coordinate (dax) at (0,3.5);
  \coordinate (dcx) at (2,3.5);
  \coordinate (dex) at (4,3.5);
  \coordinate (dgx) at (6,3.5);
  \coordinate (dix) at (8,3.5);
  \coordinate (fax) at (0,5.5);
  \coordinate (fcx) at (2,5.5);
  \coordinate (fex) at (4,5.5);
  \coordinate (fgx) at (6,5.5);
  \coordinate (fix) at (8,5.5);
  \coordinate (hax) at (0,7.5);
  \coordinate (hcx) at (2,7.5);
  \coordinate (hex) at (4,7.5);
  \coordinate (hgx) at (6,7.5);
  \coordinate (hix) at (8,7.5);
  \draw (ab)--(ib);
  \draw (ad)--(id);
  \draw (af)--(if);
  \draw (ah)--(ih);
  \draw (ba)--(bi);
  \draw (da)--(di);
  \draw (fa)--(fi);
  \draw (ha)--(hi);
  \draw[fill=white] (ab) circle (.25);
  \draw[fill=white] (ad) circle (.25);
  \draw[fill=white] (af) circle (.25);
  \draw[fill=white] (ah) circle (.25);
  \draw[line width=0.5mm, green,fill=white] (ba) circle (.25);
  \draw[fill=white] (bc) circle (.25);
  \draw[fill=white] (be) circle (.25);
  \draw[fill=white] (bg) circle (.25);
  \draw[fill=white] (bi) circle (.25);
  \draw[fill=white] (cb) circle (.25);
  \draw[fill=white] (cd) circle (.25);
  \draw[fill=white] (cf) circle (.25);
  \draw[fill=white] (ch) circle (.25);
  \draw[line width=0.5mm, red,fill=white] (da) circle (.25);
  \draw[fill=white] (dc) circle (.25);
  \draw[fill=white] (de) circle (.25);
  \draw[fill=white] (dg) circle (.25);
  \draw[fill=white] (di) circle (.25);
  \draw[fill=white] (eb) circle (.25);
  \draw[fill=white] (ed) circle (.25);
  \draw[fill=white] (ef) circle (.25);
  \draw[fill=white] (eh) circle (.25);
  \draw[line width=0.5mm, blue,fill=white] (fa) circle (.25);
  \draw[fill=white] (fc) circle (.25);
  \draw[fill=white] (fe) circle (.25);
  \draw[fill=white] (fg) circle (.25);
  \draw[fill=white] (fi) circle (.25);
  \draw[fill=white] (gb) circle (.25);
  \draw[fill=white] (gd) circle (.25);
  \draw[fill=white] (gf) circle (.25);
  \draw[fill=white] (gh) circle (.25);
  \draw[line width=0.5mm, violet,fill=white] (ha) circle (.25);
  \draw[fill=white] (hc) circle (.25);
  \draw[fill=white] (he) circle (.25);
  \draw[fill=white] (hg) circle (.25);
  \draw[fill=white] (hi) circle (.25);
  \draw[line width=0.5mm, red,fill=white] (ib) circle (.25);
  \draw[line width=0.5mm, green,fill=white] (id) circle (.25);
  \draw[line width=0.5mm, blue,fill=white] (if) circle (.25);
  \draw[line width=0.5mm, violet,fill=white] (ih) circle (.25);
  \path[fill=white] (bb) circle (.3);
  \path[fill=white] (bd) circle (.3);
  \path[fill=white] (bf) circle (.3);
  \path[fill=white] (bh) circle (.3);
  \path[fill=white] (db) circle (.3);
  \path[fill=white] (dd) circle (.3);
  \path[fill=white] (df) circle (.3);
  \path[fill=white] (dh) circle (.3);
  \path[fill=white] (fb) circle (.3);
  \path[fill=white] (fd) circle (.3);
  \path[fill=white] (ff) circle (.3);
  \path[fill=white] (fh) circle (.3);
   \path[fill=white] (hb) circle (.3);
  \path[fill=white] (hd) circle (.3);
  \path[fill=white] (hf) circle (.3);
  \path[fill=white] (hh) circle (.3);
  \node at (bb) {$T_{4,1}$};
  \node at (bd) {$T_{4,2}$};
  \node at (bf) {$T_{4,3}$};
  \node at (bh) {$T_{4,4}$};
  \node at (db) {$T_{3,1}$};
  \node at (dd) {$T_{3,2}$};
  \node at (df) {$T_{3,3}$};
  \node at (dh) {$T_{3,4}$};
  \node at (fb) {$T_{2,1}$};
  \node at (fd) {$T_{2,2}$};
  \node at (ff) {$T_{2,3}$};
  \node at (fh) {$T_{2,4}$};
  \node at (hb) {$T_{1,1}$};
  \node at (hd) {$T_{1,2}$};
  \node at (hf) {$T_{1,3}$};
  \node at (hh) {$T_{1,4}$};
  \node at (-2,7) {row:};
  \node at (-1.2,7) {1};
  \node at (-1.2,5) {2};
  \node at (-1.2,3) {3};
  \node at (-1.2,1) {4};
  \node at (ib) {$1$};
  \node at (id) {$2$};
  \node at (if) {$3$};
  \node at (ih) {$4$};
  \node at (ha) {$4$};
  \node at (hi) {$+$};
  \node at (fa) {$3$};
  \node at (fi) {$+$};
  \node at (da) {$1$};
  \node at (di) {$+$};
  \node at (ba) {$2$};
  \node at (bi) {$+$};
  \node at (ab) {$+$};
  \node at (ad) {$+$};
  \node at (af) {$+$};
  \node at (ah) {$+$};
  \node at (7,9) {$4$};
  \node at (5,9) {$3$};
  \node at (3,9) {$2$};
  \node at (1,9) {$1$};
  \node at (0,9.04) {column:};
\end{tikzpicture}}
\end{center}
\caption{Boundary conditions for the system $\mathfrak{S}_{1234,4312}$ in the chromatic model.}
\label{iceboundary1}
\end{figure}
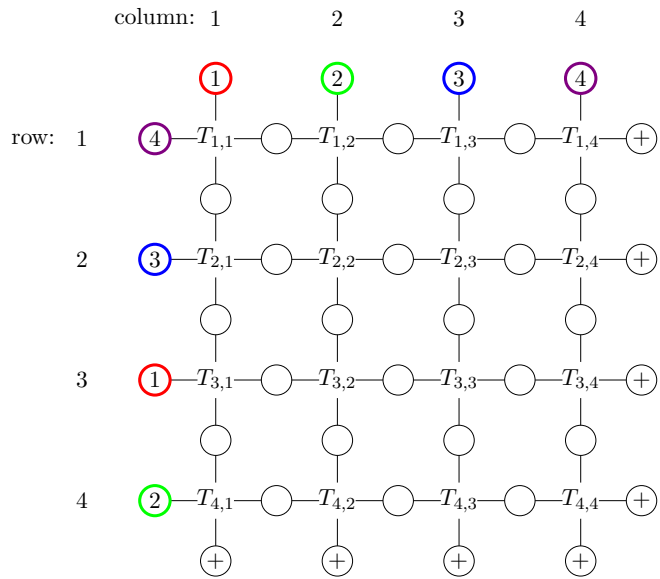

We assign a \emph{Boltzmann weight} to each vertex in the model. The weights are allowed to depend on the labels of the four edges adjacent to the vertex, the parameters $x_i, y_j$ assigned to the vertex, and the parameter $\beta$. In particular, adjacent edge labels determine whether a Boltzmann weight will be non-zero, and those vertices whose adjacent edge labeling has non-zero weight will be called \emph{admissible vertices}. In all our tables of Boltzmann weights for vertices, we report only on the admissible vertices; all other possible labelings around a vertex are understood to have weight 0. A \emph{system} $\mathfrak{S}$ is a fixed boundary condition for the grid, together with Boltzmann weights for the vertices. We often adorn $\mathfrak{S}$ with additional subscripts and superscripts to indicate a particular choice of boundary and set of Boltzmann weights. %we will often adorn the labels $T_{i,j}$ with the corresponding superscripts, but will omit these when the choice of Boltzmann weights is clear.

An \emph{admissible state} of a system will be a labeled $n \times n$ grid in which each vertex is admissible. As we will see in the subsequent sections, the Boltzmann weights on the vertices have been chosen so that admissible states consist of colored paths or ``strands'' along the edges of the lattice which begin and end along particular boundaries. The Boltzmann weight $B(\mathfrak{s})$ of an (admissible) state $\mathfrak{s}$ is defined as the product of the Boltzmann weights of its vertices. Finally, given a system $\mathfrak{S}$, the \emph{partition function} $Z(\mathfrak{S})$ is defined to be the sum of Boltzmann weights of the (admissible) states in the system, i.e.
\[ Z(\mathfrak{S})= \sum_{\mathfrak{s}\in\mathfrak{S}} B(\mathfrak{s}).\]

\subsection{The Six-Vertex Chromatic Model}

Our model arises from a Drinfeld twist of the $U_q(\widehat{\mathfrak{sl}}_{n+1})$ $R$-matrix. When we send $q\mapsto 0$, this leads to Demazure-like $\beta$-deformed divided difference operators in (\ref{isobaricdd}) that define Grothendieck polynomials. The Boltzmann weights for (admissible) vertices are presented in Figure \ref{pipeweights1}.

Rows are indexed from $1$ to $n$ increasing from top to bottom and columns are indexed from $1$ to $n$ increasing from left to right, as in Figure~\ref{iceboundary1}. It remains to describe boundary conditions for this model depending on two permutations $v,w \in S_n$.

\begin{itemize}
\item All edges on the right and bottom boundaries are labeled with $+$.
\item Edges on the top boundary are labeled, from left to right, by colors $v(1), \ldots, v(n)$.
\item Edges on the left boundary are labeled, from top to bottom, by colors $w(1), \ldots, w(n)$. 
\end{itemize}

The resulting system will be denoted $\mathfrak{S}_{v,w}$. Owing to the labeling defined above, it is often convenient to write our permutations in one-line notation. For example, the boundary conditions along the left in Figure \ref{iceboundary1} correspond to $w = 4312 = (1423)$ in cycle notation.

\begin{figure}
\centering
\scalebox{.95}{
$
\begin{array}{c@{\hspace{10pt}}c@{\hspace{30pt}}c@{\hspace{10pt}}c@{\hspace{10pt}}c@{\hspace{25pt}}}
\toprule
\tt{a}&\tt{b}_1&\tt{b_2}&\tt{c}_1&\tt{c}_2\\
\midrule
\begin{tikzpicture}
\coordinate (a) at (-.75, 0);
\coordinate (b) at (0, .75);
\coordinate (c) at (.75, 0);
\coordinate (d) at (0, -.75);
\coordinate (aa) at (-.75,.5);
\coordinate (cc) at (.75,.5);
\draw[line width=0.5mm, violet] (a)--(0,0);
\draw[line width=0.6mm, violet] (b)--(0,0);
\draw[line width=0.5mm, violet] (c)--(0,0);
\draw[line width=0.6mm, violet] (d)--(0,0);
\draw[line width=0.5mm, violet,fill=white] (a) circle (.25);
\draw[line width=0.5mm, violet,fill=white] (b) circle (.25);
\draw[line width=0.5mm, violet, fill=white] (c) circle (.25);
\draw[line width=0.5mm, violet, fill=white] (d) circle (.25);
\node at (0,1) { };
\node at (a) {$c$};
\node at (b) {$c$};
\node at (c) {$c$};
\node at (d) {$c$};
\end{tikzpicture}
&
\begin{tikzpicture}
\coordinate (a) at (-.75, 0);
\coordinate (b) at (0, .75);
\coordinate (c) at (.75, 0);
\coordinate (d) at (0, -.75);
\coordinate (aa) at (-.75,.5);
\coordinate (cc) at (.75,.5);
\draw[line width=0.5mm, blue] (a)--(c);
\draw[line width=0.6mm, red] (b)--(d);
\draw[line width=0.5mm,blue,fill=white] (a) circle (.25);
\draw[line width=0.5mm,blue,fill=white] (c) circle (.25);
\draw[line width=0.5mm,red,fill=white] (b) circle (.25);
\draw[line width=0.5mm,red,fill=white] (d) circle (.25);
\node at (0,1) { };
\node at (a) {$b$};
\node at (b) {$a$};
\node at (c) {$b$};
\node at (d) {$a$};
\end{tikzpicture}
&
\begin{tikzpicture}
\coordinate (a) at (-.75, 0);
\coordinate (b) at (0, .75);
\coordinate (c) at (.75, 0);
\coordinate (d) at (0, -.75);
\coordinate (aa) at (-.75,.5);
\coordinate (cc) at (.75,.5);
\draw[line width=0.5mm, red] (a)--(c);
\draw[line width=0.6mm, blue] (b)--(d);
\draw[line width=0.5mm,red,fill=white] (a) circle (.25);
\draw[line width=0.5mm,red,fill=white] (c) circle (.25);
\draw[line width=0.5mm,blue,fill=white] (b) circle (.25);
\draw[line width=0.5mm,blue,fill=white] (d) circle (.25);
\node at (0,1) { };
\node at (a) {$a$};
\node at (b) {$b$};
\node at (c) {$a$};
\node at (d) {$b$};
\end{tikzpicture}
%%%%%%%
& \begin{tikzpicture}
\coordinate (a) at (-.75, 0);
\coordinate (b) at (0, .75);
\coordinate (c) at (.75, 0);
\coordinate (d) at (0, -.75);
\coordinate (aa) at (-.75,.5);
\coordinate (cc) at (.75,.5);
\draw[line width=0.5mm, blue](a)--(0,0)--(b);
\draw[line width=0.5mm,blue,fill=white] (b) circle (.25);
\draw[line width=0.5mm,blue,fill=white] (a) circle (.25);
\draw[line width=0.5mm, red](d)--(0,0)--(c);
\draw[line width=0.5mm,red,fill=white] (c) circle (.25);
\draw[line width=0.5mm,red,fill=white] (d) circle (.25);
\node at (0,1) { };
\node at (a) {$b$};
\node at (b) {$b$};
\node at (c) {$a$};
\node at (d) {$a$};
\end{tikzpicture}
%%%%%%%
& \begin{tikzpicture}
\coordinate (a) at (-.75, 0);
\coordinate (b) at (0, .75);
\coordinate (c) at (.75, 0);
\coordinate (d) at (0, -.75);
\coordinate (aa) at (-.75,.5);
\coordinate (cc) at (.75,.5);
\draw[line width=0.5mm, red](a)--(0,0)--(b);
\draw[line width=0.5mm,red,fill=white] (b) circle (.25);
\draw[line width=0.5mm,red,fill=white] (a) circle (.25);
\draw[line width=0.5mm, blue](d)--(0,0)--(c);
\draw[line width=0.5mm,blue,fill=white] (c) circle (.25);
\draw[line width=0.5mm,blue,fill=white] (d) circle (.25);
\node at (0,1) { };
\node at (a) {$a$};
\node at (b) {$a$};
\node at (c) {$b$};
\node at (d) {$b$};
\end{tikzpicture}
%%%%%%%%%
\\
   \midrule
 1-q^2(1+\beta(x_i\oplus y_j)) & x_i\oplus y_j & \beta^2q^2(x_i\oplus y_j) & (1-q^2)(1+\beta (x_i\oplus y_j)) & 1-q^2  \\
   \bottomrule
\end{array}
$}
\caption{The Boltzmann weights for the chromatic model at a vertex in row $i$ and column $j$, where $x_i \oplus y_j$ denotes the formal group law $x_i+y_j+\beta x_iy_j$, $\textcolor{red}{a}<\textcolor{blue}{b}$, and ${\color{violet} c}$ is any color. We consider the $+$ label to be larger than any color, and the same weights hold when one or more labels are $+$.}
\label{pipeweights1}
\end{figure}

\begin{example}
Let $w = 4321$ and $v = 1234$. The system $\mathfrak{S}_{v,w}$ has only one admissible state, which is shown in Figure \ref{icestate1}. The weight of the first row is $(x_1 \oplus y_1)(x_1\oplus y_2)(x_1\oplus y_3)(1-q^2)$, since each of the first three vertices is of type $\tt{b}_1$ and the fourth is of type $\tt{c}_2$. Continuing similarly, we get that the full weight of this state, and thus this system, is $(1-q^2)^4\prod_{i+j < 5}(x_i \oplus y_j)\prod_{i+j > 5} (1-q^2(1+\beta(x_i\oplus y_j))$.\end{example}

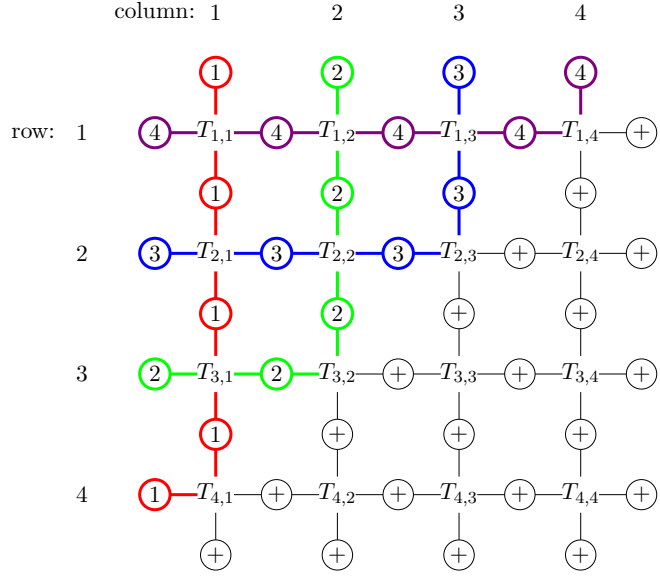
\begin{figure}[h]
\begin{center}
\scalebox{0.8}{
\begin{tikzpicture}
  \coordinate (ab) at (1,0);
  \coordinate (ad) at (3,0);
  \coordinate (af) at (5,0);
  \coordinate (ah) at (7,0);
  \coordinate (ba) at (0,1);
  \coordinate (bc) at (2,1);
  \coordinate (be) at (4,1);
  \coordinate (bg) at (6,1);
  \coordinate (bi) at (8,1);
  \coordinate (cb) at (1,2);
  \coordinate (cd) at (3,2);
  \coordinate (cf) at (5,2);
  \coordinate (ch) at (7,2);
  \coordinate (da) at (0,3);
  \coordinate (dc) at (2,3);
  \coordinate (de) at (4,3);
  \coordinate (dg) at (6,3);
  \coordinate (di) at (8,3);
  \coordinate (eb) at (1,4);
  \coordinate (ed) at (3,4);
  \coordinate (ef) at (5,4);
  \coordinate (eh) at (7,4);
  \coordinate (fa) at (0,5);
  \coordinate (fc) at (2,5);
  \coordinate (fe) at (4,5);
  \coordinate (fg) at (6,5);
  \coordinate (fi) at (8,5);
  \coordinate (gb) at (1,6);
  \coordinate (gd) at (3,6);
  \coordinate (gf) at (5,6);
  \coordinate (gh) at (7,6);
  \coordinate (ha) at (0,7);
  \coordinate (hc) at (2,7);
  \coordinate (he) at (4,7);
  \coordinate (hg) at (6,7);
  \coordinate (hi) at (8,7);
  \coordinate (ib) at (1,8);
  \coordinate (id) at (3,8);
  \coordinate (if) at (5,8);
  \coordinate (ih) at (7,8);
  \coordinate (bb) at (1,1);
  \coordinate (bd) at (3,1);
  \coordinate (bf) at (5,1);
  \coordinate (bh) at (7,1);
  \coordinate (db) at (1,3);
  \coordinate (dd) at (3,3);
  \coordinate (df) at (5,3);
  \coordinate (dh) at (7,3);
  \coordinate (fb) at (1,5);
  \coordinate (fd) at (3,5);
  \coordinate (ff) at (5,5);
  \coordinate (fh) at (7,5);
  \coordinate (hb) at (1,7);
  \coordinate (hd) at (3,7);
  \coordinate (hf) at (5,7);
  \coordinate (hh) at (7,7);
  \coordinate (bax) at (0,1.5);
  \coordinate (bcx) at (2,1.5);
  \coordinate (bex) at (4,1.5);
  \coordinate (bgx) at (6,1.5);
  \coordinate (bix) at (8,1.5);
  \coordinate (dax) at (0,3.5);
  \coordinate (dcx) at (2,3.5);
  \coordinate (dex) at (4,3.5);
  \coordinate (dgx) at (6,3.5);
  \coordinate (dix) at (8,3.5);
  \coordinate (fax) at (0,5.5);
  \coordinate (fcx) at (2,5.5);
  \coordinate (fex) at (4,5.5);
  \coordinate (fgx) at (6,5.5);
  \coordinate (fix) at (8,5.5);
  \coordinate (hax) at (0,7.5);
  \coordinate (hcx) at (2,7.5);
  \coordinate (hex) at (4,7.5);
  \coordinate (hgx) at (6,7.5);
  \coordinate (hix) at (8,7.5);
  \draw (ab)--(ib);
  \draw (ad)--(id);
  \draw (af)--(if);
  \draw (ah)--(ih);
  \draw (ba)--(bi);
  \draw (da)--(di);
  \draw (fa)--(fi);
  \draw (ha)--(hi);
  \draw[line width=0.5mm,red] (ba)--(bb)--(ib);
  \draw[line width=0.5mm,green] (da)--(dd)--(id);
  \draw[line width=0.5mm,blue] (fa)--(ff)--(if);
  \draw[line width=0.5mm,violet] (ha)--(hh)--(ih);
  \draw[fill=white] (ab) circle (.25);
  \draw[fill=white] (ad) circle (.25);
  \draw[fill=white] (af) circle (.25);
  \draw[fill=white] (ah) circle (.25);
  \draw[line width=0.5mm,red,fill=white] (ba) circle (.25);
  \draw[fill=white] (bc) circle (.25);
  \draw[fill=white] (be) circle (.25);
  \draw[fill=white] (bg) circle (.25);
  \draw[fill=white] (bi) circle (.25);
  \draw[line width=0.5mm,red,fill=white] (cb) circle (.25);
  \draw[fill=white] (cd) circle (.25);
  \draw[fill=white] (cf) circle (.25);
  \draw[fill=white] (ch) circle (.25);
  \draw[line width=0.5mm,green,fill=white] (da) circle (.25);
  \draw[line width=0.5mm,green,fill=white] (dc) circle (.25);
  \draw[fill=white] (de) circle (.25);
  \draw[fill=white] (dg) circle (.25);
  \draw[fill=white] (di) circle (.25);
  \draw[line width=0.5mm,red,fill=white] (eb) circle (.25);
  \draw[line width=0.5mm,green,fill=white] (ed) circle (.25);
  \draw[fill=white] (ef) circle (.25);
  \draw[fill=white] (eh) circle (.25);
  \draw[line width=0.5mm,blue,fill=white] (fa) circle (.25);
  \draw[line width=0.5mm,blue,fill=white] (fc) circle (.25);
  \draw[line width=0.5mm,blue,fill=white] (fe) circle (.25);
  \draw[fill=white] (fg) circle (.25);
  \draw[fill=white] (fi) circle (.25);
  \draw[line width=0.5mm,red,fill=white] (gb) circle (.25);
  \draw[line width=0.5mm,green,fill=white] (gd) circle (.25);
  \draw[line width=0.5mm,blue,fill=white] (gf) circle (.25);
  \draw[fill=white] (gh) circle (.25);
  \draw[line width=0.5mm,violet,fill=white] (ha) circle (.25);
  \draw[line width=0.5mm,violet,fill=white] (hc) circle (.25);
  \draw[line width=0.5mm,violet,fill=white] (he) circle (.25);
  \draw[line width=0.5mm,violet,fill=white] (hg) circle (.25);
  \draw[fill=white] (hi) circle (.25);
  \draw[line width=0.5mm,red,fill=white] (ib) circle (.25);
  \draw[line width=0.5mm,green,fill=white] (id) circle (.25);
  \draw[line width=0.5mm,blue,fill=white] (if) circle (.25);
  \draw[line width=0.5mm,violet,fill=white] (ih) circle (.25);
 \path[fill=white] (bb) circle (.3);
  \path[fill=white] (bd) circle (.3);
  \path[fill=white] (bf) circle (.3);
  \path[fill=white] (bh) circle (.3);
  \path[fill=white] (db) circle (.3);
  \path[fill=white] (dd) circle (.3);
  \path[fill=white] (df) circle (.3);
  \path[fill=white] (dh) circle (.3);
  \path[fill=white] (fb) circle (.3);
  \path[fill=white] (fd) circle (.3);
  \path[fill=white] (ff) circle (.3);
  \path[fill=white] (fh) circle (.3);
   \path[fill=white] (hb) circle (.3);
  \path[fill=white] (hd) circle (.3);
  \path[fill=white] (hf) circle (.3);
  \path[fill=white] (hh) circle (.3);
  \node at (bb) {$T_{4,1}$};
  \node at (bd) {$T_{4,2}$};
  \node at (bf) {$T_{4,3}$};
  \node at (bh) {$T_{4,4}$};
  \node at (db) {$T_{3,1}$};
  \node at (dd) {$T_{3,2}$};
  \node at (df) {$T_{3,3}$};
  \node at (dh) {$T_{3,4}$};
  \node at (fb) {$T_{2,1}$};
  \node at (fd) {$T_{2,2}$};
  \node at (ff) {$T_{2,3}$};
  \node at (fh) {$T_{2,4}$};
  \node at (hb) {$T_{1,1}$};
  \node at (hd) {$T_{1,2}$};
  \node at (hf) {$T_{1,3}$};
  \node at (hh) {$T_{1,4}$};
  \node at (-2,7) {row:};
  \node at (-1.2,7) {1};
  \node at (-1.2,5) {2};
  \node at (-1.2,3) {3};
  \node at (-1.2,1) {4};
  \node at (ib) {$1$};
  \node at (id) {$2$};
  \node at (if) {$3$};
  \node at (ih) {$4$};
  \node at (ha) {$4$};
  \node at (hc) {$4$};
  \node at (he) {$4$};
  \node at (hg) {$4$};
  \node at (hi) {$+$};
  \node at (gb) {$1$};
  \node at (gd) {$2$};
  \node at (gf) {$3$};
  \node at (gh) {$+$};
  \node at (fa) {$3$};
  \node at (fc) {$3$};
  \node at (fe) {$3$};
  \node at (fg) {$+$};
  \node at (fi) {$+$};
  \node at (eb) {$1$};
  \node at (ed) {$2$};
  \node at (ef) {$+$};
  \node at (eh) {$+$};
  \node at (da) {$2$};
  \node at (dc) {$2$};
  \node at (de) {$+$};
  \node at (dg) {$+$};
  \node at (di) {$+$};
  \node at (cb) {$1$};
  \node at (cd) {$+$};
  \node at (cf) {$+$};
  \node at (ch) {$+$};
  \node at (ba) {$1$};
  \node at (bc) {$+$};
  \node at (be) {$+$};
  \node at (bg) {$+$};
  \node at (bi) {$+$};
  \node at (ab) {$+$};
  \node at (ad) {$+$};
  \node at (af) {$+$};
  \node at (ah) {$+$};
  \node at (7,9) {$4$};
  \node at (5,9) {$3$};
  \node at (3,9) {$2$};
  \node at (1,9) {$1$};
  \node at (0,9.04) {column:};
\end{tikzpicture}}
\end{center}
\caption{The sole admissible state of system $\mathfrak{S}_{1234,4321}.$ Recall that we abbreviate $T_{i,j} := x_i,y_j$.}
\label{icestate1}
\end{figure}

In order to relate our weights to an $R$-matrix of a quantum group, we make use of the notation \[T_{a,b}^{c,d} \longleftrightarrow \raisebox{-1 em}{\scalebox{0.7}{\gammaice{d}{a}{b}{c}{}{}}}.\] Since we are defining rectangular weights, we use $T$ instead of $R$. Note that this use of $T_{a,b}^{c,d}$ refers to the matrix $T$ at the end of the proof of Proposition \ref{rectangularweightquantumgroupprop}, and should not be confused with the shorthand $T_{i,j}$ in figures used to denote the rectangular vertex in row $i$ and column $j$.

\begin{proposition} \label{rectangularweightquantumgroupprop}
The weights of Figure \ref{pipeweights1} arise from a Drinfeld twist of the $R$-matrix for an evaluation module for $U_q(\widehat{\mathfrak{sl}}_{n+1})$.
\end{proposition}

\begin{proof}
We use the $R$-matrix for the standard evaluation module for $U_q(\widehat{\mathfrak{sl}}_{n+1})$, which we will call $\twid{R}$. This can be found in many places, such as \cite{PerkSchultz-Uq(sln)Rmatrix,Jimbo}; we will use the version from Definition 2.1 of \cite{Kojima}. For $U_q(\widehat{\mathfrak{sl}}_{n+1})$, we consider two spins $a<b$; multiplying all weights from \cite{Kojima} by $q^2z-1$ to clear denominators, we have
\begin{align*}
    \twid{R}_{a,a}^{a,a} = \twid{R}_{b,b}^{b,b} &= 1-q^2z\\
    \twid{R}_{b,a}^{b,a} = \twid{R}_{a,b}^{a,b} &= (z-1)q\\
    \twid{R}_{a,b}^{b,a} &= 1-q^2\\
    \twid{R}_{b,a}^{a,b} &= (1-q^2)z
\end{align*} We will write this as a matrix using the basis $(a,a), (a,b), (b,a), (b,b)$: \begin{align*} 
    \twid{R} = \twid{R}^{(a,b)} := \begin{bmatrix}
    1-q^2z&0&0&0\\0&(z-1)q&1-q^2&0\\ 0&(1-q^2)z&(z-1)q&0\\ 0&0&0&1-q^2z
    \end{bmatrix}.
\end{align*} Note that this is just the 2-color matrix; for the full matrix, one would vary $a$ and $b$ over all colors (including $+$) such that $a<b$.

Next, we perform a Drinfeld twist by the matrix
\begin{align*}
    F := \begin{bmatrix}1&0&0&0\\0&(q\beta)^{1/2}&0&0\\0&0&(q\beta)^{-1/2}&0\\0&0&0&1
    \end{bmatrix}.
\end{align*} When $q \neq 0$ and $\beta \neq 0$, $F$ is invertible, and one may check that it satisfies the necessary conditions for a Drinfeld twist \cite[(1)]{Reshetikhin-twisted-hopf-algebras}. To apply the Drinfeld twist, notice that $F_{21} = F^{-1}$. Then the twisted $T$-matrix is \begin{align} \label{twistedRmatrix}
    \twid{R}^F := F_{21}\twid{R}F^{-1} = \begin{bmatrix}
    1-q^2z&0&0&0\\0&(z-1)q^2\beta&1-q^2&0\\ 0&(1-q^2)z&(z-1)\beta^{-1}&0\\ 0&0&0&1-q^2z
    \end{bmatrix}.
\end{align}

Finally, we substitute \[z\mapsto 1+\beta(x_i\oplus y_j) = (1+\beta x_i)(1+\beta y_j)\] to arrive at the following weights \begin{align*}
    T := \begin{bmatrix}
    1-q^2(1+\beta(x_i\oplus y_j))&0&0&0\\0&\beta^2q^2(x_i\oplus y_j)&1-q^2&0\\ 0&(1-q^2)(1+\beta(x_i\oplus y_j))&x_i\oplus y_j&0\\ 0&0&0&1-q^2(1+\beta(x_i\oplus y_j))
    \end{bmatrix},
\end{align*}
which precisely match the weights in Figure \ref{pipeweights1}.
\end{proof}

Our model in many cases is closely related to lattice models studied by Zinn-Justin \cite{Zinn-Justin-LRSchur}, Wheeler and Zinn-Justin \cite{ZJ-Wheeler}, Gorbounov and Korff \cite{GorbunovKorff}, and Knutson and Zinn-Justin \cite{KnutsonZJ-schubert-puzzles, KnutsonZJ-motivic}. We will now describe these relationships.

For permutations $w$ of only one descent (i.e. for which there is only one $k$ such that $w(k) > w(k+1)$), a specialization of our model recovers the model for Grassmannian Grothendieck polynomials used in \cite{ZJ-Wheeler}. First, specialize $q=0,\beta=-1$. Each such polynomial is indexed by a Weyl group element $w$, which corresponds to a partition $\lambda$ according to the following prescription. Let $k$ be the descent of $w$: then, for each $j = 1,...,k$, let $\lambda_{k+1-j} = w(j) - j$ (see Sottile \cite[\S 2]{Sottile} for discussion of this correspondence). Their partition functions are slightly different, owing to a slightly different convention for $\oplus$ in the Grothendieck polynomial, but the underlying diagrams possess pipes of the same shape. To match them precisely, set all strands below the descent $k$ to be labeled as $+$ strands, set the strands above the descent to all be a single color, then remove the rows with only $+$ strands (i.e. rows $k+1$ through $n$).

The five-vertex model obtained by this specialization also matches the \textit{osculating walker model} given by the ``L$'$ weights" in Figure 3.1 of \cite{GorbunovKorff}. On the other hand, the Grothendieck model of Buciumas and Scrimshaw \cite{BuciumasScrimshaw-bumpless-model} is reminiscent of, but does not specialize to, the \textit{vicious walker model} given by the ``L weights" of~\cite{GorbunovKorff}. It would be interesting to determine whether a modification of their model exists which degenerates exactly to the vicious walker model. Indeed, several recent works have made use of lattice models with these two different flavors of weights; these include the so-called ``Gamma'' and ``Delta'' weights used in \cite{BBFschur, BBCG, Ivanov, BBB}.

Similarly, for permutations $w$ of 3 or fewer descents, our model recovers the model denoted $\boldsymbol{S}^\lambda$ by Knutson and Zinn-Justin for Grothendieck polynomials ($q=0,\beta = -1$) in \cite{KnutsonZJ-schubert-puzzles}.  To see this for a given permutation $w$ of $d \leq 3$ descents, start with $\mathfrak{S}_{w,1}$ and relabel colored strand labels via the following method. Let $\omega(w)$ be the weakly increasing string of integers that increases at each descent of $w$. For example, if $w = 25143$ then $\omega(w) = 11223$. Now replace the left-hand boundary edges of $\mathfrak{S}_{w,1}$ with the labels given by $\omega(w)$. The resulting model is $\boldsymbol{S}^\lambda$, where we may read off the corresponding $\lambda$ used by Knutson and Zinn-Justin from the labels on the top edge. To match their definition of the Grothendieck polynomial, replace our formal group law $\oplus$ with  $x_i \oplus y_j = 1 - x_i/y_j$ and set our $w$ to be the inverse of their $\sigma$. For example, take $w = 25143$, so $w^{-1} = 31542$. Considering $\mathfrak{S}_{w,1}$, we then set strands 1 and 2 to be color $d_1$ (shown below in red), strands 3 and 4 to be color $d_2$ (green), and strand 5 to be color $d_3$ (blue). Reading just the indices of the reduced colors, we have $\omega := \omega(w) = 11223$ down the left boundary and $\lambda = 21321$ across the top boundary (matching the running example of Section 3.7 of \cite{KnutsonZJ-schubert-puzzles}), as shown in the figure below.
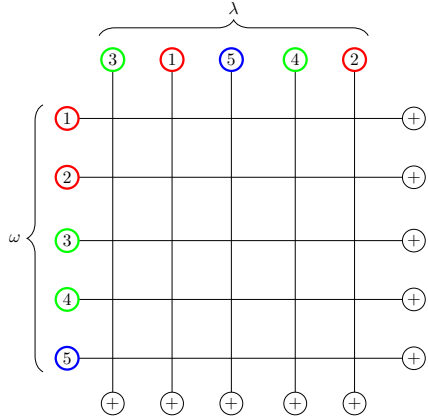
\begin{center}
\scalebox{0.6}{
\begin{tikzpicture}
  \coordinate (ab) at (1,0);
  \coordinate (ad) at (2.3,0);
  \coordinate (af) at (3.6,0);
  \coordinate (ah) at (5,0);
  \coordinate (aj) at (6.3,0);
  \coordinate (ba) at (0,1);
  \coordinate (bj) at (7.6,1);
  \coordinate (da) at (0,2.3);
  \coordinate (dj) at (7.6,2.3);
  \coordinate (fa) at (0,3.6);
  \coordinate (fj) at (7.6,3.6);
  \coordinate (ha) at (0,5);
  \coordinate (hj) at (7.6,5);
  \coordinate (ja) at (0,6.3);
  \coordinate (jj) at (7.6,6.3);
  \coordinate (kb) at (1,7.6);
  \coordinate (kd) at (2.3,7.6);
  \coordinate (kf) at (3.6,7.6);
  \coordinate (kh) at (5,7.6);
  \coordinate (kj) at (6.3,7.6);
  \draw (ab)--(kb);
  \draw (ad)--(kd);
  \draw (af)--(kf);
  \draw (ah)--(kh);
  \draw (aj)--(kj);
  \draw (ba)--(bj);
  \draw (da)--(dj);
  \draw (fa)--(fj);
  \draw (ha)--(hj);
  \draw (ja)--(jj);
  \draw[fill=white] (ab) circle (.25);
  \draw[fill=white] (ad) circle (.25);
  \draw[fill=white] (af) circle (.25);
  \draw[fill=white] (ah) circle (.25);
  \draw[fill=white] (aj) circle (.25);
  \draw[line width=0.5mm,blue,fill=white] (ba) circle (.25);
  \draw[fill=white] (bj) circle (.25);
  \draw[line width=0.5mm,green,fill=white] (da) circle (.25);
  \draw[fill=white] (dj) circle (.25);
  \draw[line width=0.5mm,green,fill=white] (fa) circle (.25);
  \draw[fill=white] (fj) circle (.25);
  \draw[line width=0.5mm,red,fill=white] (ha) circle (.25);
  \draw[fill=white] (hj) circle (.25);
  \draw[line width=0.5mm,red,fill=white] (ja) circle (.25);
  \draw[fill=white] (jj) circle (.25);
  \draw[line width=0.5mm,green,fill=white] (kb) circle (.25);
  \draw[line width=0.5mm,red,fill=white] (kd) circle (.25);
  \draw[line width=0.5mm,blue,fill=white] (kf) circle (.25);
  \draw[line width=0.5mm,green,fill=white] (kh) circle (.25);
  \draw[line width=0.5mm,red,fill=white] (kj) circle (.25);
  \draw [decorate,decoration={brace,amplitude=10pt},xshift=-15pt,yshift=0pt] (0,0.7) -- (0,6.6) node [black,midway,left, xshift=-10pt]{$\omega$};
  \draw [decorate,decoration={brace,amplitude=10pt},yshift=15pt,xshift=0pt] (0.7,7.6) -- (6.6,7.6) node [black,midway,above, yshift=10pt]{$\lambda$};
%  \node at (-2,7) {row:};
%  \node at (-1.2,7) {1};
%  \node at (-1.2,5) {2};
%  \node at (-1.2,3) {3};
%  \node at (-1.2,1) {4};
  \node at (kb) {$3$};
  \node at (kd) {$1$};
  \node at (kf) {$5$};
  \node at (kh) {$4$};
  \node at (kj) {$2$};
  \node at (ja) {$1$};
  \node at (jj) {$+$};
  \node at (ha) {$2$};
  \node at (hj) {$+$};
  \node at (fa) {$3$};
  \node at (fj) {$+$};
  \node at (da) {$4$};
  \node at (dj) {$+$};
  \node at (ba) {$5$};
  \node at (bj) {$+$};
  \node at (ab) {$+$};
  \node at (ad) {$+$};
  \node at (af) {$+$};
  \node at (ah) {$+$};
  \node at (aj) {$+$};
%  \node at (7,9) {$4$};
%  \node at (5,9) {$3$};
%  \node at (3,9) {$2$};
%  \node at (1,9) {$1$};
%  \node at (0,9.04) {column:};
\end{tikzpicture}}
\end{center}
Note that setting labels within the same descent to be the same color \emph{does not} change the partition function, since these strands could only have interacted in type $\tt{c}_2$ vertices before, which now become type $\tt{a}$ vertices, and both of these vertices have the same weight.

Finally, in the case where $q$ is arbitrary and $\beta=-1$ for any number of descents our model is related to the model used by Knutson and Zinn-Justin in \cite[(1)]{KnutsonZJ-schubert-puzzles}. Their model uses vertex weights directly from \cite{Jimbo}, (without a Drinfeld twist), and twist only when specializing $q=0$. Their polynomials have the advantage of being related to the Schubert calculus on the cotangent bundle of the flag variety, while ours have the advantage of specializing directly to the (non-deformed) $\beta$-Grothendieck polynomial. In addition, our boundary conditions are more general than theirs, and there doesn't appear to be an analogue of our ``biaxial'' polynomials in the literature.

\section{Solvability of the Lattice Models}\label{solvability}

We now show that the model in Section \ref{Pipes} is {\it solvable}, meaning that is satisfies a family of (quantum) Yang-Baxter equations (YBEs) for every pair of adjacent rows or columns. We say that a set of Boltzmann weights $T$ has a {\it row Yang-Baxter equation} in rows $i$ and $j$ if there exists a set of vertex weights $R_{i,j}$ such that, for any choice of boundary labels $\alpha,\beta, \gamma, \delta, \epsilon$ and $\eta$, equality holds for the partition functions of the following systems:
\begin{align}\label{Rvertexdefinition}
\begin{array}{c}
\scalebox{.85}{\begin{tikzpicture}
    %lines in YBE
  \draw (0,0)--(2,0);
  \draw (0,2)--(2,2);
  \draw (1,-1)--(1,3);
  \coordinate (a1) at (-2,0);
  \coordinate (c1) at (0,2);
  \coordinate (a2) at (-2,2);
  \coordinate (c2) at (0,0);
  % R vertex
  \draw (a1) to [out=0,in=180] (c1);
  \draw (a2) to [out=0,in=180] (c2);
  % boundary circles
  \draw[fill=white] (-2,0) circle (.3);
  \draw[fill=white] (-2,2) circle (.3);
  \draw[fill=white](1,3) circle(.3);
  \draw[fill=white] (2,2) circle (.3);
  \draw[fill=white] (2,0) circle (.3);
  \draw[fill=white](1,-1) circle (.3);
  %inner circles
  \path[fill=white] (-1,1) circle (.3);%R_i,j
  \path[fill=white] (1,2) circle (.3);%z_i
  \path[fill=white] (1,0) circle (.3);%z_j
  %outer labels
  \node at (-2,0){$\alpha$};
  \node at (-2,2){$\beta$};
  \node at (1,3){$\gamma$};
  \node at (2,2) {$\delta$};
  \node at (2,0) {$\epsilon$};
  \node at (1,-1){$\eta$};
  %inner labels
  \node at (1,2) {$T_{i,k}$};
  \node at (1,0) {$T_{j,k}$};
  \node at (-1,1){$R_{i,j}$};
  %equality
  \node at (3,1) {{\Large $=$}};
  \end{tikzpicture}}
\hspace{1em}
\scalebox{.85}{\begin{tikzpicture}
    %lines in YBE
  \draw (0,0)--(-2,0); 
  \draw (0,2)--(-2,2);
  \draw (-1,-1)--(-1,3);
  \coordinate (a1) at (2,0);
  \coordinate (c1) at (0,2);
  \coordinate (a2) at (2,2);
  \coordinate (c2) at (0,0);
  %R vertex
  \draw (a1) to [out=180,in=0] (c1);
  \draw (a2) to [out=180,in=0] (c2);
  %boundary circles
  \draw[fill=white] (-2,0) circle (.3);
  \draw[fill=white] (-2,2) circle (.3);
  \draw[fill=white](-1,3) circle(.3);
  \draw[fill=white] (2,2) circle (.3);
  \draw[fill=white] (2,0) circle (.3);
  \draw[fill=white](-1,-1) circle (.3);
  %inner vertex circles
  \path[fill=white] (1,1) circle (.3);%R_i,j
  \path[fill=white] (-1,2) circle (.3);%z_j
  \path[fill=white] (-1,0) circle (.3);%z_i
  %boundary labels
  \node at (-2,0) {$\alpha$};
  \node at (-2,2) {$\beta$};
  \node at (-1,3) {$\gamma$};
  \node at (2,2) {$\delta$};
  \node at (2,0) {$\epsilon$};
  \node at (-1,-1) {$\eta$};
  %inner vertex labels
  \node at (-1,2) {$T_{j,k}$};
  \node at (-1,0) {$T_{i,k}$};
  \node at (1,1) {$R_{i,j}$};
  \end{tikzpicture}}\end{array}.
\end{align}

Analogously, we say such set of Boltzmann weights for vertices has a {\it column Yang-Baxter equation} if instead there exists a set of vertex weights $R_{i,j}$ such that for any choice of boundary conditions $\alpha,\beta, \gamma, \delta, \epsilon$ and $\eta$ equality holds for the partition functions of the following:

\begin{align}\label{columnYBE}
\begin{array}{c}
\scalebox{.85}{\begin{tikzpicture}
    %YBE lines
  \draw (0,0)--(0,-2);
  \draw (2,0)--(2,-2);
  \draw (-1,-1)--(3,-1);
  \coordinate (a1) at (0,2);
  \coordinate (c1) at (2,0);
  \coordinate (a2) at (2,2);
  \coordinate (c2) at (0,0);
  %R vertex
  \draw (a1) to [out=-90,in=90] (c1);
  \draw (a2) to [out=-90,in=90] (c2);
  %boundary circles
  \draw[fill=white] (0,2) circle (.3);
  \draw[fill=white] (2,2) circle (.3);
  \draw[fill=white] (3,-1) circle (.3);
  \draw[fill=white] (2,-2) circle (.3);
  \draw[fill=white] (0,-2) circle (.3);
  \draw[fill=white] (-1,-1) circle (.3);
  %inner circles
  \path[fill=white] (0,-1) circle (.3);
  \path[fill=white] (2,-1) circle (.3);
  \path[fill=white] (1,1) circle (.4);
  \node at (0,-1) {$T_{k,i}$};
  \node at (2,-1) {$T_{k,j}$};
  \node at (1,1){$R_{i,j}$};
  \node at (0,2) {$\alpha$};
  \node at (2,2){$\beta$};
  \node at (3,-1){$\gamma$};
  \node at (2,-2){$\delta$};
  \node at (0,-2){$\epsilon$};
  \node at (-1,-1) {$\eta$};
  %equals
  \node at (4,0) {{\Large $=$}};
  \end{tikzpicture}}
\hspace{1em}
\scalebox{.85}{\begin{tikzpicture}
  % Ybe lines
  \draw (0,0)--(0,2);
  \draw (2,0)--(2,2);
  \draw (-1,1)--(3,1);
  \coordinate (a1) at (0,0);
  \coordinate (c1) at (2,-2);
  \coordinate (a2) at (2,0);
  \coordinate (c2) at (0,-2);
  %R vertex
  \draw (a1) to [out=-90,in=90] (c1);
  \draw (a2) to [out=-90,in=90] (c2);
  %boundary circles
  \draw[fill=white] (0,2) circle (.3);
  \draw[fill=white] (2,2) circle (.3);
  \draw[fill=white](3,1) circle(.3);
  \draw[fill=white] (2,-2) circle (.3);
  \draw[fill=white] (0,-2) circle (.3);
  \draw[fill=white](-1,1) circle (.3);
  %inner circles
  \path[fill=white] (0,1) circle (.3);
  \path[fill=white] (2,1) circle (.3);
  \path[fill=white] (1,-1) circle (.4);
  %inner vertex labels
  \node at (0,1) {$T_{k,j}$};
  \node at (2,1) {$T_{k,i}$};
  \node at (1,-1){$R_{i,j}$};
  %outer vertex labels
  \node at (0,2){$\alpha$};
  \node at (2,2){$\beta$};
  \node at (3,1) {$\gamma$};
  \node at (2,-2) {$\delta$};
  \node at (0,-2){$\epsilon$};
  \node at (-1,1){$\eta$};
  \end{tikzpicture}}\end{array}.
\end{align}
Here we have drawn the solution weights $R_{i,j}$ to the Yang-Baxter equation by rotating the earlier vertex diagrams by 45 degrees, thereby indicating that it is a different kind of vertex with a different set of Boltzmann weights. Colloquially, we refer to this rotated vertex as an $R$-vertex and readers familiar with the algebraic interpretation of the above diagrams will note that the $R$-vertex Boltzmann weights are entries of the $R$-matrix solving the Yang-Baxter equation as an endomorphism of a triple tensor product of vector spaces. We sometimes refer to our earlier vertices that comprise the square grid as \emph{rectangular vertices} in contrast with the $R$-vertices. 

\begin{theorem}\label{DemazureYBE}
The Boltzmann weights from Figure \ref{pipeweights1} satisfy a row Yang-Baxter equation with $R$-vertex weights given by the first row of Figure \ref{RweightsDemazure} and a column Yang-Baxter equation with $R$-vertex weights in the second row of Figure \ref{RweightsDemazure}. Thus, our model is solvable in both row and column variables. 
\end{theorem}

\begin{proof}
To obtain the row $R$-vertices in Figure \ref{RweightsDemazure}, we take the matrix in (\ref{twistedRmatrix}), and make the substitution $z\mapsto x_j\ominus x_i = \frac{x_j-x_i}{1+\beta x_i}$. Noting that $1 + \beta(x_j\ominus x_i) = \frac{1+\beta x_j}{1+\beta x_i}$, we scale the weights by $1+\beta x_i$ to obtain \begin{align*}
   R =  \begin{bmatrix}
    1+\beta x_i -q^2(1+\beta x_j) &0&0&0\\0& x_j-x_i&(1-q^2)(1+\beta x_i)&0\\0&(1-q^2)(1+\beta x_j)&q^2\beta^2(x_j-x_i)&0 \\ 0&0&0&1+\beta x_i-q^2(1+\beta x_j)
    \end{bmatrix},
\end{align*}
which matches the row weights in Figure \ref{RweightsDemazure} under the correspondence \[R_{a,b}^{c,d} \longleftrightarrow \raisebox{-1 em}{\scalebox{0.7}{
\begin{tikzpicture}[scale=0.77, every node/.style={scale=0.8}]
\coordinate (a) at (-.55, -.55);
\coordinate (b) at (-.55, .55);
\coordinate (c) at (.55, .55);
\coordinate (d) at (.55, -.55);
\draw (a)--(c);
\draw (b)--(d);
\draw[fill=white] (a) circle (.25);
\draw[fill=white] (b) circle (.25);
\draw[fill=white] (c) circle (.25);
\draw[fill=white] (d) circle (.25);
\node at (0,1) { };
\node at (a) {$d$};
\node at (b) {$c$};
\node at (c) {$b$};
\node at (d) {$a$};
\end{tikzpicture}}}.\]

By the $U_q(\widehat{\mathfrak{sl}}_{n+1})$ Yang-Baxter equation \cite[(2.10)]{Kojima}, we have \begin{align*} \twid{R}_{12}(u/v)\twid{R}_{13}(u)\twid{R}_{23}(v) = \twid{R}_{23}(v)\twid{R}_{13}(u)\twid{R}_{12}(u/v),\end{align*} for any parameters $u$ and $v$. Drinfeld twists don't affect this relation \cite[Theorem~1]{Reshetikhin-twisted-hopf-algebras}, so the relation holds when we twist by $F$: \begin{align*} \twid{R}^F_{12}(u/v)\twid{R}^F_{13}(u)\twid{R}^F_{23}(v) = \twid{R}^F_{23}(v)\twid{R}^F_{13}(u)\twid{R}^F_{12}(u/v).\end{align*}
To obtain the rectangular weights in Figure \ref{pipeweights1}, we substitute $u\mapsto 1+\beta(x_j\oplus y) = (1+\beta x_j)(1+\beta y)$, $v\mapsto 1+\beta(x_i\oplus y) = (1+\beta x_i)(1+\beta y)$. Then, $\frac{u}{v} = \frac{1+\beta x_j}{1+\beta x_i} = 1 + \beta(x_j\ominus x_i)$.

Then, $\twid{R}^F(1 + \beta(x_j\ominus x_i)) = R(x_i,x_j)$ and $\twid{R}^F(1 + \beta(x\oplus y)) = T(x,y)$. Therefore, the row solvability of our model is given by the modified Yang-Baxter equation \begin{align*} \twid{R}^F_{12}\left(1 + \beta(x_j\ominus x_i)\right) & \twid{R}^F_{13}(1+\beta(x_j\oplus y))\twid{R}^F_{23}(1+\beta(x_j\oplus y)) \\&= \twid{R}^F_{23}(1+\beta(x_j\oplus y))\twid{R}^F_{13}(1+\beta(x_j\oplus y))\twid{R}^F_{12}\left(1 + \beta(x_j\ominus x_i)\right).\end{align*}

Column solvability follows from the same argument with the modified Yang-Baxter equation \begin{align*}\twid{R}^F_{12}\left(1 + \beta(y_j\ominus y_i)\right) & \twid{R}^F_{13}(1+\beta(x\oplus y_j))\twid{R}^F_{23}(1+\beta(x\oplus y_i)) \\&= \twid{R}^F_{23}(1+\beta(x\oplus y_i))\twid{R}^F_{13}(1+\beta(x\oplus y_j))\twid{R}^F_{12}\left(1 + \beta(y_j\ominus y_i)\right). \qedhere \end{align*}
\end{proof}

\begin{figure}[h]
\centering
\scalebox{.95}{$\begin{array}{c@{\hspace{10pt}}c@{\hspace{30pt}}c@{\hspace{25pt}}c@{\hspace{10pt}}c@{\hspace{25pt}}}
\toprule
\multicolumn{5}{l}{\text{{\bf Chromatic Model $R$-vertex weights}:}}\\
\toprule
\multicolumn{5}{l}{\text{Row Yang-Baxter equation $R$-vertex weights:}}\\
\toprule
\tt{a} & \tt{b_1} & \tt{b_2} & \tt{c_1} & \tt{c_2} \\
\midrule
%  ++++ : 1+beta x_i
\begin{tikzpicture}[scale=0.7]
\draw[line width = .5mm, violet] (0,0) to [out = 0, in = 180] (2,2);
\draw[line width = .5mm, violet] (0,2) to [out = 0, in = 180] (2,0);
\draw[line width=0.5mm, violet, fill=white] (0,0) circle (.35);
\draw[line width=0.5mm, violet, fill=white] (0,2) circle (.35);
\draw[line width=0.5mm, violet, fill=white] (2,2) circle (.35);
\draw[line width=0.5mm, violet, fill=white] (2,0) circle (.35);
\node at (0,0) {$c$};
\node at (0,2) {$c$};
\node at (2,2) {$c$};
\node at (2,0) {$c$};
\end{tikzpicture}
&
\begin{tikzpicture}[scale=0.7]
\draw[line width = .5mm, red] (0,0) to [out = 0, in = 180] (2,2);
\draw[line width = .5mm, blue] (0,2) to [out = 0, in = 180] (2,0);
\draw[line width=0.5mm, red, fill=white] (0,0) circle (.35);
\draw[line width=0.5mm, blue, fill=white] (0,2) circle (.35);
\draw[line width=0.5mm, red, fill=white] (2,2) circle (.35);
\draw[line width=0.5mm, blue, fill=white] (2,0) circle (.35);
\node at (0,0) {$a$};
\node at (0,2) {$b$};
\node at (2,2) {$a$};
\node at (2,0) {$b$};
\end{tikzpicture}
&
\begin{tikzpicture}[scale=0.7]
\draw[line width = .5mm, blue] (0,0) to [out = 0, in = 180] (2,2);
\draw[line width = .5mm, red] (0,2) to [out = 0, in = 180] (2,0);
\draw[line width=0.5mm, blue, fill=white] (0,0) circle (.35);
\draw[line width=0.5mm, red, fill=white] (0,2) circle (.35);
\draw[line width=0.5mm, blue, fill=white] (2,2) circle (.35);
\draw[line width=0.5mm, red, fill=white] (2,0) circle (.35);
\node at (0,0) {$b$};
\node at (0,2) {$a$};
\node at (2,2) {$b$};
\node at (2,0) {$a$};
\end{tikzpicture}
&
\begin{tikzpicture}[scale=0.7]
\draw[line width = .5mm,blue] (0,2) to [out = 0, in = 120] (1,1);
\draw[line width = .5mm, blue] (2,2) to [out = 180, in=60] (1,1);
\draw[line width = .5mm, red] (0,0) to [out = 0, in = -120] (1,1);
\draw[line width = .5mm, red] (2,0) to [out = 180, in = -60] (1,1);
\draw[line width=0.5mm, red, fill=white] (0,0) circle (.35);
\draw[line width=0.5mm, blue, fill=white] (0,2) circle (.35);
\draw[line width=0.5mm, blue, fill=white] (2,2) circle (.35);
\draw[line width=0.5mm, red, fill=white] (2,0) circle (.35);
\node at (0,0) {$a$};
\node at (0,2) {$b$};
\node at (2,2) {$b$};
\node at (2,0) {$a$};
\end{tikzpicture}
&
\begin{tikzpicture}[scale=0.7]
\draw[line width = .5mm,red] (0,2) to [out = 0, in = 120] (1,1);
\draw[line width = .5mm, red] (2,2) to [out = 180, in=60] (1,1);
\draw[line width = .5mm, blue] (0,0) to [out = 0, in = -120] (1,1);
\draw[line width = .5mm, blue] (2,0) to [out = 180, in = -60] (1,1);
\draw[line width=0.5mm, blue, fill=white] (0,0) circle (.35);
\draw[line width=0.5mm, red, fill=white] (0,2) circle (.35);
\draw[line width=0.5mm, red, fill=white] (2,2) circle (.35);
\draw[line width=0.5mm, blue, fill=white] (2,0) circle (.35);
\node at (0,0) {$b$};
\node at (0,2) {$a$};
\node at (2,2) {$a$};
\node at (2,0) {$b$};
\end{tikzpicture}
\\
   \midrule
 1+\beta x_i - q^2(1+\beta x_j) & x_j-x_i & \beta^2q^2(x_j-x_i) & (1-q^2)(1+\beta x_j) & (1-q^2)(1+\beta x_i) \\
\toprule
\toprule
\multicolumn{5}{l}{\text{Column Yang-Baxter equation $R$-vertex weights:}}\\
\toprule
\tt{a} & \tt{b_1} & \tt{b_2} & \tt{c_1} & \tt{c_2} \\
\midrule
%  ++++ : 1+beta x_i
\begin{tikzpicture}[scale=0.7]
\draw[line width = .5mm, violet] (0,0) to [out = 0, in = 180] (2,2);
\draw[line width = .5mm, violet] (0,2) to [out = 0, in = 180] (2,0);
\draw[line width=0.5mm, violet, fill=white] (0,0) circle (.35);
\draw[line width=0.5mm, violet, fill=white] (0,2) circle (.35);
\draw[line width=0.5mm, violet, fill=white] (2,2) circle (.35);
\draw[line width=0.5mm, violet, fill=white] (2,0) circle (.35);
\node at (0,0) {$c$};
\node at (0,2) {$c$};
\node at (2,2) {$c$};
\node at (2,0) {$c$};
\end{tikzpicture}
&
\begin{tikzpicture}[scale=0.7]
\draw[line width = .5mm, red] (0,0) to [out = 0, in = 180] (2,2);
\draw[line width = .5mm, blue] (0,2) to [out = 0, in = 180] (2,0);
\draw[line width=0.5mm, red, fill=white] (0,0) circle (.35);
\draw[line width=0.5mm, blue, fill=white] (0,2) circle (.35);
\draw[line width=0.5mm, red, fill=white] (2,2) circle (.35);
\draw[line width=0.5mm, blue, fill=white] (2,0) circle (.35);
\node at (0,0) {$a$};
\node at (0,2) {$b$};
\node at (2,2) {$a$};
\node at (2,0) {$b$};
\end{tikzpicture}
&
\begin{tikzpicture}[scale=0.7]
\draw[line width = .5mm, blue] (0,0) to [out = 0, in = 180] (2,2);
\draw[line width = .5mm, red] (0,2) to [out = 0, in = 180] (2,0);
\draw[line width=0.5mm, blue, fill=white] (0,0) circle (.35);
\draw[line width=0.5mm, red, fill=white] (0,2) circle (.35);
\draw[line width=0.5mm, blue, fill=white] (2,2) circle (.35);
\draw[line width=0.5mm, red, fill=white] (2,0) circle (.35);
\node at (0,0) {$b$};
\node at (0,2) {$a$};
\node at (2,2) {$b$};
\node at (2,0) {$a$};
\end{tikzpicture}
&
\begin{tikzpicture}[scale=0.7]
\draw[line width = .5mm,blue] (0,2) to [out = 0, in = 120] (1,1);
\draw[line width = .5mm, red] (2,2) to [out = 180, in=60] (1,1);
\draw[line width = .5mm, blue] (0,0) to [out = 0, in = -120] (1,1);
\draw[line width = .5mm, red] (2,0) to [out = 180, in = -60] (1,1);
\draw[line width=0.5mm, blue, fill=white] (0,0) circle (.35);
\draw[line width=0.5mm, blue, fill=white] (0,2) circle (.35);
\draw[line width=0.5mm, red, fill=white] (2,2) circle (.35);
\draw[line width=0.5mm, red, fill=white] (2,0) circle (.35);
\node at (0,0) {$b$};
\node at (0,2) {$b$};
\node at (2,2) {$a$};
\node at (2,0) {$a$};
\end{tikzpicture}&
\begin{tikzpicture}[scale=0.7]
\draw[line width = .5mm,red] (0,2) to [out = 0, in = 120] (1,1);
\draw[line width = .5mm, blue] (2,2) to [out = 180, in=60] (1,1);
\draw[line width = .5mm, red] (0,0) to [out = 0, in = -120] (1,1);
\draw[line width = .5mm, blue] (2,0) to [out = 180, in = -60] (1,1);
\draw[line width=0.5mm, red, fill=white] (0,0) circle (.35);
\draw[line width=0.5mm, red, fill=white] (0,2) circle (.35);
\draw[line width=0.5mm, blue, fill=white] (2,2) circle (.35);
\draw[line width=0.5mm, blue, fill=white] (2,0) circle (.35);
\node at (0,0) {$a$};
\node at (0,2) {$a$};
\node at (2,2) {$b$};
\node at (2,0) {$b$};
\end{tikzpicture}
\\
   \midrule
 1+\beta y_i - q^2(1+\beta y_j) &  \beta^2q^2(y_j-y_i) & y_j-y_i & (1-q^2)(1+\beta y_j) & (1-q^2)(1+\beta y_i) \\
   \bottomrule
\end{array}$}
    \caption{The $R$-vertex weights that swap strands $i$ and $j$, where ${\color{red}a}<{\color{blue}b}$ and {\color{violet} $c$} is any color. The first set of weights satisfies the row Yang-Baxter equation, and the second satisfies the column Yang-Baxter equation.}
    \label{RweightsDemazure}
\end{figure}

\begin{example}\label{YBEexample} With boundary conditions as in Figure \ref{YBEexamplefigure}, there are two admissible states of the first system in the column Yang-Baxter  equation (\ref{columnYBE}) with $R$-vertex weights as in Figure \ref{RweightsDemazure} and rectangular weights as in Figure \ref{pipeweights1}. The partition function is $(1-q^2)^2(1+\beta y_i)(x_k\oplus y_i) + (1-q^2)^2(y_j-y_i)(1+\beta(x_k\oplus y_i))$. On the other hand, there is one admissible state of the second system with weight $(1-q^2)^2(1+\beta y_i)(x_k\oplus y_j)$, and indeed this equals the first partition function.
\end{example}

\begin{figure}[h]
\begin{align*}
\scalebox{.85}{\begin{tikzpicture}
    %Ybe lines
  \draw (0,0)--(0,-2);
  \draw (2,0)--(2,-2);
  \draw (-1,-1)--(3,-1);
  \coordinate (a1) at (0,2);
  \coordinate (c1) at (2,0);
  \coordinate (a2) at (2,2);
  \coordinate (c2) at (0,0);
  % R vertex
  \draw[line width=0.5mm, red] (a1) to [out=-90,in=150] (1,1);
  \draw[line width=0.5mm, red] (c2) to [out=90,in=-150] (1,1);
  \draw[line width=.5mm,green] (a2) to [out=-90,in=30] (1,1);
  \draw[line width=.5mm,green] (c1) to [out=90,in=-30] (1,1);
  %colored lines
  \draw[line width=.5mm,green] (c1)--(2,-1)--(-1,-1);
  \draw[line width=.5mm,red] (c2)--(0,-2);
  \draw[line width=.5mm,blue] (3,-1)--(2,-1)--(2,-2);
  %edge circles
  \draw[line width=0.5mm, blue,fill=white] (2,-2) circle (.3);
  \draw[line width=0.5mm, red,fill=white] (0,-2) circle (.3);
  \draw[line width=0.5mm, green,fill=white] (-1,-1) circle (.3);
  \draw[line width=0.5mm, blue,fill=white] (3,-1) circle (.3);
  \draw[line width=0.5mm, red,fill=white] (0,2) circle (.3);
  \draw[line width=0.5mm, green,fill=white] (2,2) circle (.3);
  \draw[line width=0.5mm, green,fill=white] (c1) circle (.3);
  \draw[line width=0.5mm, red,fill=white] (c2) circle (.3);
  \draw[line width=0.5mm, green,fill=white] (1,-1) circle (.3);
  %vertex circles
  \path[fill=white] (0,-1) circle (.3);
  \path[fill=white] (2,-1) circle (.3);
  \path[fill=white] (1,1) circle (.4);
  %vertex labels
  \node at (0,-1) {$T_{k,i}$};
  \node at (2,-1) {$T_{k,j}$};
  \node at (1,1){$R_{i,j}$};
  %edge labels
  \node at (0,-2) {$1$};
  \node at (-1,-1) {$2$};
  \node at (0,2) {$1$};
  \node at (2,2) {$2$};
  \node at (3,-1) {$3$};
  \node at (2,-2) {$3$};
  \node at (c1) {$2$};
  \node at (c2) {$1$};
  \node at (1,-1) {$2$};
  \node at (4,0) {{\Large $+$}};
  \end{tikzpicture}}
\hspace{1em}
\scalebox{.85}{\begin{tikzpicture}
  %YBE lines
  \draw (0,0)--(0,-2);
  \draw (2,0)--(2,-2);
  \draw (-1,-1)--(3,-1);
  \coordinate (a1) at (0,2);
  \coordinate (c1) at (2,0);
  \coordinate (a2) at (2,2);
  \coordinate (c2) at (0,0);
  %strands
  \draw[line width=0.5mm, red] (a1) to [out=-90,in=90] (c1)--(2,-1)--(0,-1)--(0,-2);
  \draw[line width=0.5mm, green] (a2) to [out=-90,in=90] (c2)--(0,-1)--(-1,-1);
  \draw[line width=0.5mm, blue] (3,-1)--(2,-1)--(2,-2);
  %edge circles
  \draw[line width=0.5mm, blue,fill=white] (2,-2) circle (.3);
  \draw[line width=0.5mm, red,fill=white] (0,-2) circle (.3);
  \draw[line width=0.5mm, green,fill=white] (-1,-1) circle (.3);
  \draw[line width=0.5mm, blue,fill=white] (3,-1) circle (.3);
  \draw[line width=0.5mm, red,fill=white] (0,2) circle (.3);
  \draw[line width=0.5mm, green,fill=white] (2,2) circle (.3);
  \draw[line width=0.5mm, red,fill=white] (c1) circle (.3);
  \draw[line width=0.5mm, green,fill=white] (c2) circle (.3);
  \draw[line width=0.5mm, red,fill=white] (1,-1) circle (.3);
  %vertex circles
  \path[fill=white] (0,-1) circle (.3);
  \path[fill=white] (2,-1) circle (.3);
  \path[fill=white] (1,1) circle (.4);
  %vertex labels
  \node at (0,-1) {$T_{k,i}$};
  \node at (2,-1) {$T_{k,j}$};
  \node at (1,1) {$R_{i,j}$};
  %edge labels
  \node at (0,-2) {$1$};
  \node at (-1,-1) {$2$};
  \node at (0,2) {$1$};
  \node at (2,2) {$2$};
  \node at (3,-1) {$3$};
  \node at (2,-2) {$3$};
  \node at (c1) {$1$};
  \node at (c2) {$2$};
  \node at (1,-1) {$1$};
  \node at (4,0) {{\Large $=$}};
  \end{tikzpicture}}
  \hspace{1em}
  \scalebox{.85}{\begin{tikzpicture}
  %YBE lines
  \draw (0,0)--(0,2);
  \draw (2,0)--(2,2);
  \draw (-1,1)--(3,1);
  \coordinate (a1) at (0,0);
  \coordinate (c1) at (2,-2);
  \coordinate (a2) at (2,0);
  \coordinate (c2) at (0,-2);
  %strands
  \draw[line width=0.5mm, red] (0,2)--(a1) to [out=-90,in=150] (1,-1);
  \draw[line width=0.5mm, red] (c2) to [out=90,in=-150] (1,-1);
  \draw[line width=.5mm,blue] (3,1)--(2,1)--(a2) to [out=-90,in=30] (1,-1);
  \draw[line width=.5mm,blue] (c1) to [out=90,in=-30] (1,-1);
  \draw[line width=.5mm,green] (-1,1)--(2,1)--(2,2);
  %edge circles
  \draw[line width=0.5mm, red,fill=white] (0,-2) circle (.3);
  \draw[line width=0.5mm, blue,fill=white] (2,-2) circle (.3);
  \draw[line width=0.5mm, green,fill=white](-1,1) circle (.3);
  \draw[line width=0.5mm, blue,fill=white](3,1) circle(.3);
  \draw[line width=0.5mm, red,fill=white] (0,2) circle (.3);
  \draw[line width=0.5mm, green,fill=white] (2,2) circle (.3);
  \draw[line width=0.5mm, red,fill=white] (a1) circle (.3);
  \draw[line width=0.5mm, blue,fill=white] (a2) circle (.3);
  \draw[line width=0.5mm, green,fill=white] (1,1) circle (.3);
  %vertex circles
  \path[fill=white] (0,1) circle (.3);
  \path[fill=white] (2,1) circle (.3);
  \path[fill=white] (1,-1) circle (.4);
  %vertex labels
  \node at (0,1) {$T_{k,j}$};
  \node at (2,1) {$T_{k,i}$};
  \node at (1,-1) {$R_{i,j}$};
  %edge labels
  \node at (0,2) {$1$};
  \node at (2,2) {$2$};
  \node at (3,1) {$3$};
  \node at (2,-2) {$3$};
  \node at (0,-2) {$1$};
  \node at (-1,1) {$2$};
  \node at (a2) {$3$};
  \node at (a1) {$1$};
  \node at (1,1) {$2$};
  \end{tikzpicture}}
\end{align*}
   \caption{The three admissible states that appear in the column Yang-Baxter equation (see (\ref{columnYBE})) with the chosen boundary conditions, and rectangular Boltzmann weights as in Figure \ref{pipeweights1}.}
   \label{YBEexamplefigure}
\end{figure}

\begin{figure}
\centering
\scalebox{.95}{
$
\begin{array}{c@{\hspace{10pt}}c@{\hspace{10pt}}c@{\hspace{5pt}}c@{\hspace{5pt}}c@{\hspace{5pt}}}
\toprule
\multicolumn{5}{l}{\text{{\bf Weights $S$}:}}\\
\toprule
\tt{a}&\tt{b}_1&\tt{b_2}&\tt{c}_1&\tt{c}_2\\
\midrule
\begin{tikzpicture}
\coordinate (a) at (-.75, 0);
\coordinate (b) at (0, .75);
\coordinate (c) at (.75, 0);
\coordinate (d) at (0, -.75);
\coordinate (aa) at (-.75,.5);
\coordinate (cc) at (.75,.5);
\draw[line width=0.5mm, violet] (a)--(0,0);
\draw[line width=0.6mm, violet] (b)--(0,0);
\draw[line width=0.5mm, violet] (c)--(0,0);
\draw[line width=0.6mm, violet] (d)--(0,0);
\draw[line width=0.5mm, violet,fill=white] (a) circle (.25);
\draw[line width=0.5mm, violet,fill=white] (b) circle (.25);
\draw[line width=0.5mm, violet, fill=white] (c) circle (.25);
\draw[line width=0.5mm, violet, fill=white] (d) circle (.25);
\node at (0,1) { };
\node at (a) {$c$};
\node at (b) {$c$};
\node at (c) {$c$};
\node at (d) {$c$};
\end{tikzpicture}
&
\begin{tikzpicture}
\coordinate (a) at (-.75, 0);
\coordinate (b) at (0, .75);
\coordinate (c) at (.75, 0);
\coordinate (d) at (0, -.75);
\coordinate (aa) at (-.75,.5);
\coordinate (cc) at (.75,.5);
\draw[line width=0.5mm, blue] (a)--(c);
\draw[line width=0.6mm, red] (b)--(d);
\draw[line width=0.5mm,blue,fill=white] (a) circle (.25);
\draw[line width=0.5mm,blue,fill=white] (c) circle (.25);
\draw[line width=0.5mm,red,fill=white] (b) circle (.25);
\draw[line width=0.5mm,red,fill=white] (d) circle (.25);
\node at (0,1) { };
\node at (a) {$b$};
\node at (b) {$a$};
\node at (c) {$b$};
\node at (d) {$a$};
\end{tikzpicture}
&
\begin{tikzpicture}
\coordinate (a) at (-.75, 0);
\coordinate (b) at (0, .75);
\coordinate (c) at (.75, 0);
\coordinate (d) at (0, -.75);
\coordinate (aa) at (-.75,.5);
\coordinate (cc) at (.75,.5);
\draw[line width=0.5mm, red] (a)--(c);
\draw[line width=0.6mm, blue] (b)--(d);
\draw[line width=0.5mm,red,fill=white] (a) circle (.25);
\draw[line width=0.5mm,red,fill=white] (c) circle (.25);
\draw[line width=0.5mm,blue,fill=white] (b) circle (.25);
\draw[line width=0.5mm,blue,fill=white] (d) circle (.25);
\node at (0,1) { };
\node at (a) {$a$};
\node at (b) {$b$};
\node at (c) {$a$};
\node at (d) {$b$};
\end{tikzpicture}
%%%%%%%
& \begin{tikzpicture}
\coordinate (a) at (-.75, 0);
\coordinate (b) at (0, .75);
\coordinate (c) at (.75, 0);
\coordinate (d) at (0, -.75);
\coordinate (aa) at (-.75,.5);
\coordinate (cc) at (.75,.5);
\draw[line width=0.5mm, blue](a)--(0,0)--(b);
\draw[line width=0.5mm,blue,fill=white] (b) circle (.25);
\draw[line width=0.5mm,blue,fill=white] (a) circle (.25);
\draw[line width=0.5mm, red](d)--(0,0)--(c);
\draw[line width=0.5mm,red,fill=white] (c) circle (.25);
\draw[line width=0.5mm,red,fill=white] (d) circle (.25);
\node at (0,1) { };
\node at (a) {$b$};
\node at (b) {$b$};
\node at (c) {$a$};
\node at (d) {$a$};
\end{tikzpicture}
%%%%%%%
& \begin{tikzpicture}
\coordinate (a) at (-.75, 0);
\coordinate (b) at (0, .75);
\coordinate (c) at (.75, 0);
\coordinate (d) at (0, -.75);
\coordinate (aa) at (-.75,.5);
\coordinate (cc) at (.75,.5);
\draw[line width=0.5mm, red](a)--(0,0)--(b);
\draw[line width=0.5mm,red,fill=white] (b) circle (.25);
\draw[line width=0.5mm,red,fill=white] (a) circle (.25);
\draw[line width=0.5mm, blue](d)--(0,0)--(c);
\draw[line width=0.5mm,blue,fill=white] (c) circle (.25);
\draw[line width=0.5mm,blue,fill=white] (d) circle (.25);
\node at (0,1) { };
\node at (a) {$a$};
\node at (b) {$a$};
\node at (c) {$b$};
\node at (d) {$b$};
\end{tikzpicture}
%%%%%%%%%
\\
   \midrule
 1-q^2(1+\beta(x_i\oplus y_j)) & \beta^2q^2(x_i\oplus y_j) &  x_i\oplus y_j & (1-q^2)(1+\beta (x_i\oplus y_j)) & 1-q^2  \\
\toprule
\toprule
\multicolumn{5}{l}{\text{{\bf Weights $S^*$}:}}\\
\toprule
\tt{a}&\tt{b}_1&\tt{b_2}&\tt{c}_1&\tt{c}_2\\
\midrule
\begin{tikzpicture}
\coordinate (a) at (-.75, 0);
\coordinate (b) at (0, .75);
\coordinate (c) at (.75, 0);
\coordinate (d) at (0, -.75);
\coordinate (aa) at (-.75,.5);
\coordinate (cc) at (.75,.5);
\draw[line width=0.5mm, violet] (a)--(0,0);
\draw[line width=0.6mm, violet] (b)--(0,0);
\draw[line width=0.5mm, violet] (c)--(0,0);
\draw[line width=0.6mm, violet] (d)--(0,0);
\draw[line width=0.5mm, violet,fill=white] (a) circle (.25);
\draw[line width=0.5mm, violet,fill=white] (b) circle (.25);
\draw[line width=0.5mm, violet, fill=white] (c) circle (.25);
\draw[line width=0.5mm, violet, fill=white] (d) circle (.25);
\node at (0,1) { };
\node at (a) {$c$};
\node at (b) {$c$};
\node at (c) {$c$};
\node at (d) {$c$};
\end{tikzpicture}
&
\begin{tikzpicture}
\coordinate (a) at (-.75, 0);
\coordinate (b) at (0, .75);
\coordinate (c) at (.75, 0);
\coordinate (d) at (0, -.75);
\coordinate (aa) at (-.75,.5);
\coordinate (cc) at (.75,.5);
\draw[line width=0.5mm, blue] (a)--(c);
\draw[line width=0.6mm, red] (b)--(d);
\draw[line width=0.5mm,blue,fill=white] (a) circle (.25);
\draw[line width=0.5mm,blue,fill=white] (c) circle (.25);
\draw[line width=0.5mm,red,fill=white] (b) circle (.25);
\draw[line width=0.5mm,red,fill=white] (d) circle (.25);
\node at (0,1) { };
\node at (a) {$b$};
\node at (b) {$a$};
\node at (c) {$b$};
\node at (d) {$a$};
\end{tikzpicture}
&
\begin{tikzpicture}
\coordinate (a) at (-.75, 0);
\coordinate (b) at (0, .75);
\coordinate (c) at (.75, 0);
\coordinate (d) at (0, -.75);
\coordinate (aa) at (-.75,.5);
\coordinate (cc) at (.75,.5);
\draw[line width=0.5mm, red] (a)--(c);
\draw[line width=0.6mm, blue] (b)--(d);
\draw[line width=0.5mm,red,fill=white] (a) circle (.25);
\draw[line width=0.5mm,red,fill=white] (c) circle (.25);
\draw[line width=0.5mm,blue,fill=white] (b) circle (.25);
\draw[line width=0.5mm,blue,fill=white] (d) circle (.25);
\node at (0,1) { };
\node at (a) {$a$};
\node at (b) {$b$};
\node at (c) {$a$};
\node at (d) {$b$};
\end{tikzpicture}
%%%%%%%
& \begin{tikzpicture}
\coordinate (a) at (-.75, 0);
\coordinate (b) at (0, .75);
\coordinate (c) at (.75, 0);
\coordinate (d) at (0, -.75);
\coordinate (aa) at (-.75,.5);
\coordinate (cc) at (.75,.5);
\draw[line width=0.5mm, blue](a)--(0,0)--(d);
\draw[line width=0.5mm,blue,fill=white] (d) circle (.25);
\draw[line width=0.5mm,blue,fill=white] (a) circle (.25);
\draw[line width=0.5mm, red](b)--(0,0)--(c);
\draw[line width=0.5mm,red,fill=white] (c) circle (.25);
\draw[line width=0.5mm,red,fill=white] (b) circle (.25);
\node at (0,1) { };
\node at (a) {$b$};
\node at (b) {$a$};
\node at (c) {$a$};
\node at (d) {$b$};
\end{tikzpicture}
%%%%%%%
& \begin{tikzpicture}
\coordinate (a) at (-.75, 0);
\coordinate (b) at (0, .75);
\coordinate (c) at (.75, 0);
\coordinate (d) at (0, -.75);
\coordinate (aa) at (-.75,.5);
\coordinate (cc) at (.75,.5);
\draw[line width=0.5mm, red](a)--(0,0)--(d);
\draw[line width=0.5mm,red,fill=white] (d) circle (.25);
\draw[line width=0.5mm,red,fill=white] (a) circle (.25);
\draw[line width=0.5mm, blue](b)--(0,0)--(c);
\draw[line width=0.5mm,blue,fill=white] (c) circle (.25);
\draw[line width=0.5mm,blue,fill=white] (b) circle (.25);
\node at (0,1) { };
\node at (a) {$a$};
\node at (b) {$b$};
\node at (c) {$b$};
\node at (d) {$a$};
\end{tikzpicture}
%%%%%%%%%
\\
   \midrule
 1-q^2(1+\beta(x_i\oplus (\ominus y_j))) &  x_i\oplus (\ominus y_j) &  \beta^2q^2(x_i\oplus (\ominus y_j))& (1-q^2)(1+\beta (x_i\oplus (\ominus y_j))) & 1-q^2  \\
   \bottomrule
\end{array}
$}
\caption{The Boltzmann weights $S_{i,j}$ and $S_{i,j}^*$ at a vertex in row $i$ and column $j$ where $\textcolor{red}{a}<\textcolor{blue}{b}$, and ${\color{violet} c}$ is any color.}
\label{cauchyweights}
\end{figure}

We will eventually have the use for a Yang-Baxter equation that interchanges rows of two different transformations of our weights from Figure \ref{pipeweights1}. Let $S_{i,k}$ denote the set of weights, such that the weight of a vertex in row $i$ and column $k$ is equal to the its weight from Figure 3 after reflecting the vertex along the line $y=-x$. Let $S_{i,k}^*$ denote a second set of weights where we instead reflect across the $x$-axis, followed by making the substitution $y_k \mapsto \ominus y_k$. The notation for the second set of weights is owed to the fact that eventually, we will show that the weights $S^*$ are connected to the dual of the polynomial given by the weights $S$. The weights $S$ and $S^*$ can be found in Figure \ref{cauchyweights}. Below, we show exists a set of $R$-vertices designed to swap a row with weights from $S$ with one having weights according to $S^*$. We will use this Yang-Baxter equation to great effect in Section \ref{cauchy}. We call these the \emph{rhombus $R$-vertices} since the swapping of strands with different weights is analogous to operations encoded by the rhomboid tiles of Knutson-Tao puzzles (see, e.g. \cite{ZJ-Wheeler}).

\begin{theorem}\label{SwapYBE}
For any choice of boundary conditions $\alpha, \beta,\gamma, \delta, \epsilon,$ and $\eta$ the partition functions of the following two states are equal:
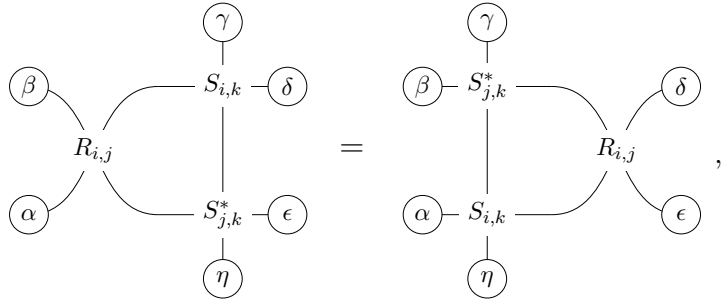
\begin{align*}
\begin{array}{c}
\scalebox{.85}{\begin{tikzpicture}
  %YBE lines
  \draw (0,0)--(2,0);
  \draw (0,2)--(2,2);
  \draw (1,-1)--(1,3);
  \coordinate (a1) at (-2,0);
  \coordinate (c1) at (0,2);
  \coordinate (a2) at (-2,2);
  \coordinate (c2) at (0,0);
  %R vertex
  \draw (a1) to [out=0,in=180] (c1);
  \draw (a2) to [out=0,in=180] (c2);
  %boundary circles
  \draw[fill=white] (-2,2) circle (.3);
  \draw[fill=white] (-2,0) circle (.3);
  \draw[fill=white](1,-1) circle (.3);
  \draw[fill=white](1,3) circle(.3);
  \draw[fill=white] (2,2) circle (.3);
  \draw[fill=white] (2,0) circle (.3);
  %vertex circles
  %\path[fill=white] (-1,1) circle (.4);
  %\path[fill=white] (1,2) circle (.55);
 % \path[fill=white] (1,0) circle (.55);
  %vertex labels
  \node[rectangle, fill=white] at (1,2) { $S_{i,k}$ };
  \node[rectangle, fill=white] at (1,0) {$S_{j,k}^*$};
  \node[rectangle, fill=white] at (-1,1){$R_{i,j}$};
  %boundary labels
  \node at (1,3){$\gamma$};
  \node at (2,2) {$\delta$};
  \node at (2,0) {$\epsilon$};
  \node at (-2,2){$\beta$};
  \node at (-2,0){$\alpha$};
  \node at (1,-1){$\eta$};
  %equal
  \node at (3,1) {{\Large $=$}};
  \end{tikzpicture}}
\hspace{1em}
\scalebox{.85}{\begin{tikzpicture}
  %YBE lines
  \draw (0,0)--(-2,0);
  \draw (0,2)--(-2,2);
  \draw (-1,-1)--(-1,3);
  \coordinate (a1) at (2,0);
  \coordinate (c1) at (0,2);
  \coordinate (a2) at (2,2);
  \coordinate (c2) at (0,0);
  %R vertex
  \draw (a1) to [out=180,in=0] (c1);
  \draw (a2) to [out=180,in=0] (c2);
  %boundary circles
  \draw[fill=white] (2,2) circle (.3);
  \draw[fill=white] (2,0) circle (.3);
  \draw[fill=white](-1,-1) circle (.3);
  \draw[fill=white](-1,3) circle(.3);
  \draw[fill=white] (-2,2) circle (.3);
  \draw[fill=white] (-2,0) circle (.3);
  %vertex circles
  %\path[fill=white] (1,1) circle (.4);
  %\path[fill=white] (-1,2) circle (.55);
  %\path[fill=white] (-1,0) circle (.4);
  %vertex labels
  \node[rectangle, fill=white] at (-1,2) {$S_{j,k}^*$};
  \node[rectangle, fill=white] at (-1,0) {$S_{i,k}$};
  \node[rectangle, fill=white] at (1,1){$R_{i,j}$};
  %boundary labels
  \node at (-1,3){$\gamma$};
  \node at (-2,2) {$\beta$};
  \node at (-2,0) {$\alpha$};
  \node at (2,2){$\delta$};
  \node at (2,0){$\epsilon$};
  \node at (-1,-1){$\eta$};
  \end{tikzpicture}}
  \end{array},
\end{align*}
  where $S_{i,k}$ and $S_{j,k}^*$ are the Boltzmann weights defined in the paragraph above, and $R_{i,j}$ the $R$-vertex weights from Figure \ref{RweightsCauchy}.
  \end{theorem}
\begin{proof}
Because $n$ (the size of our grid, of the permutation, and of the set of colored labels) is arbitrary, then in principle the number of colors appearing in the diagrams associated to the Yang-Baxter equation is unbounded. However, the Boltzmann weights of the rectangular vertices depend only on the relative ordering of the labels, and at most three colors (including $+$) can be present. Thus it suffices to check for a solution in three colors. This is easily accomplished by a computer algebra system or a very lengthy hand calculation. In our case, the $R$-vertex weights were calculated with SageMath.
\end{proof}

\begin{figure}[h]
\centering
\scalebox{.95}{$\begin{array}{c@{\hspace{10pt}}c@{\hspace{25pt}}c@{\hspace{15pt}}c@{\hspace{10pt}}c@{\hspace{15pt}}}
\toprule
\multicolumn{5}{l}{\text{Rhombus Yang-Baxter equation $R$-vertex weights:}}\\
\toprule
\tt{a} & \tt{b_1} & \tt{b_2} & \tt{c_1} & \tt{c_2} \\
\midrule
%  ++++ : 1+beta x_i
\begin{tikzpicture}[scale=0.7]
\draw[line width = .5mm, violet] (0,0) to [out = 0, in = 180] (2,2);
\draw[line width = .5mm, violet] (0,2) to [out = 0, in = 180] (2,0);
\draw[line width=0.5mm, violet, fill=white] (0,0) circle (.35);
\draw[line width=0.5mm, violet, fill=white] (0,2) circle (.35);
\draw[line width=0.5mm, violet, fill=white] (2,2) circle (.35);
\draw[line width=0.5mm, violet, fill=white] (2,0) circle (.35);
\node at (0,0) {$c$};
\node at (0,2) {$c$};
\node at (2,2) {$c$};
\node at (2,0) {$c$};
\end{tikzpicture}
&
\begin{tikzpicture}[scale=0.7]
\draw[line width = .5mm, red] (0,0) to [out = 0, in = 180] (2,2);
\draw[line width = .5mm, blue] (0,2) to [out = 0, in = 180] (2,0);
\draw[line width=0.5mm, red, fill=white] (0,0) circle (.35);
\draw[line width=0.5mm, blue, fill=white] (0,2) circle (.35);
\draw[line width=0.5mm, red, fill=white] (2,2) circle (.35);
\draw[line width=0.5mm, blue, fill=white] (2,0) circle (.35);
\node at (0,0) {$a$};
\node at (0,2) {$b$};
\node at (2,2) {$a$};
\node at (2,0) {$b$};
\end{tikzpicture}
&
\begin{tikzpicture}[scale=0.7]
\draw[line width = .5mm, blue] (0,0) to [out = 0, in = 180] (2,2);
\draw[line width = .5mm, red] (0,2) to [out = 0, in = 180] (2,0);
\draw[line width=0.5mm, blue, fill=white] (0,0) circle (.35);
\draw[line width=0.5mm, red, fill=white] (0,2) circle (.35);
\draw[line width=0.5mm, blue, fill=white] (2,2) circle (.35);
\draw[line width=0.5mm, red, fill=white] (2,0) circle (.35);
\node at (0,0) {$b$};
\node at (0,2) {$a$};
\node at (2,2) {$b$};
\node at (2,0) {$a$};
\end{tikzpicture}
&
\begin{tikzpicture}[scale=0.7]
\draw[line width = .5mm,blue] (0,2) to [out = 0, in = 120] (1,1);
\draw[line width = .5mm, red] (2,2) to [out = 180, in=60] (1,1);
\draw[line width = .5mm, blue] (0,0) to [out = 0, in = -120] (1,1);
\draw[line width = .5mm, red] (2,0) to [out = 180, in = -60] (1,1);
\draw[line width=0.5mm, blue, fill=white] (0,0) circle (.35);
\draw[line width=0.5mm, blue, fill=white] (0,2) circle (.35);
\draw[line width=0.5mm, red, fill=white] (2,2) circle (.35);
\draw[line width=0.5mm, red, fill=white] (2,0) circle (.35);
\node at (0,0) {$b$};
\node at (0,2) {$b$};
\node at (2,2) {$a$};
\node at (2,0) {$a$};
\end{tikzpicture}&
\begin{tikzpicture}[scale=0.7]
\draw[line width = .5mm,red] (0,2) to [out = 0, in = 120] (1,1);
\draw[line width = .5mm, blue] (2,2) to [out = 180, in=60] (1,1);
\draw[line width = .5mm, red] (0,0) to [out = 0, in = -120] (1,1);
\draw[line width = .5mm, blue] (2,0) to [out = 180, in = -60] (1,1);
\draw[line width=0.5mm, red, fill=white] (0,0) circle (.35);
\draw[line width=0.5mm, red, fill=white] (0,2) circle (.35);
\draw[line width=0.5mm, blue, fill=white] (2,2) circle (.35);
\draw[line width=0.5mm, blue, fill=white] (2,0) circle (.35);
\node at (0,0) {$a$};
\node at (0,2) {$a$};
\node at (2,2) {$b$};
\node at (2,0) {$b$};
\end{tikzpicture}
\\
   \midrule
1-q^2(1+\beta (x_i\oplus x_j)) &  x_j+x_i(1+ \beta x_j) & \beta^2q^2(x_i\oplus x_j)& (1-q^2)(1+\beta (x_i\oplus x_j)) & 1-q^2 \\
   \bottomrule
\end{array}$}
   \caption{The rhombus $R$-vertex weights that swap a strand $i$ attached to weights $S$ with strand $j$ attached weights given by $S^*$. Here, ${\color{red}a}<{\color{blue}b}$ and ${\color{violet} c}$ is any color.}
   \label{RweightsCauchy}
\end{figure}

\section{Evaluating Partition Functions of the Models}\label{modelproof}

\subsection{The Partition Function of the Chromatic Model}

In order to calculate the partition function of the chromatic model, we begin with a base case set of boundary conditions and then show that the Yang-Baxter equations of Section \ref{solvability} provide operators that shift between different permutations on both labeled boundaries, acting as simple reflections on the right that can both increase and decrease length. Since we may travel in either direction regarding length, there are many possible base cases one could choose; we pick a particularly nice one with only one state.

\begin{proposition}\label{latticebasecase}
Let $n$ be any positive integer and let $w_0$ be the longest word in $S_n$. With Boltzmann weights as defined in Figure \ref{pipeweights1}, 
\[ Z\big(\mathfrak{S}_{1,w_0}(\boldsymbol{x},\boldsymbol{y})\big) =  (1-q^2)^n \prod_{i+j< n+1}(x_i \oplus y_j) \prod_{i+j>n+1} (1-q^2(1+\beta x_i\oplus y_j)).\]
In particular, this partition function matches $\mathcal{G}_{w_0}^{(\beta,q)}(x,y)$.
\end{proposition}

\begin{proof}
We induct on $n$, noting that for any $n$ the system $\mathfrak{S}_{1,w_0}$ always has a single admissible state. Note that we will add a superscript $n$ to the system when we need to emphasize in which symmetric group we are working. For the base case $n=1$, the sole admissible state is a single vertex of type $\tt{c}_2$, so $Z(\mathfrak{S}_{1,w_0}^1) = 1-q^2$ and the statement holds. Supposing the statement holds for all $k \leq n-1$, we consider the single admissible state for $w_0 \in S_n$. It is perhaps illustrative to consult Figure \ref{icestate1}, which displays this state for $w_0 \in S_4$. Since color $n$ appears in row 1 on the left boundary, it must travel straight across the row 1 before exiting out the top boundary in column $n$. Therefore, row 1 consists of $n-1$ vertices of type $\tt{b}_1$ as color $n$ crosses each of colors $1,\dots,n-1$, in columns $1,\dots,n-1$ respectively, and one vertex of type $\tt{c}_2$ in column $n$, since $+>n$. Furthermore, the remaining vertices in column $n$ must all be of type $\tt{a}$ in the label $+$. If we remove row 1 and column $n$ from our state, the remaining $n-1\times n-1$ grid  has precisely the boundary conditions for the long word $w_0 \in S_{n-1}$, so we have
\begin{align*}Z(\mathfrak{S}_{w_0}^n;x_1,\dots,x_n;y_1,\dots,y_n) =  &(1-q^2) \prod_{j=1}^{n-1} (x_1 \oplus y_j) \prod_{i = 2}^{n-1} (1 - q^2(1+ \beta(x_i\oplus y_n)))\\ &\hspace{1cm}\cdot Z(\mathfrak{S}_{w_0}^{n-1};x_2,...,x_n;y_1,\dots,y_{n-1}).
\end{align*}
Applying the inductive hypothesis then gives us precisely the desired formula.
\end{proof}

\begin{lemma}\label{permutationdecrease}
For any $w \in S_n$ and any simple reflection $s_i$ such that $\ell(ws_i) = \ell(w) - 1$,
\[ Z(\mathfrak{S}_{v,ws_i}(\boldsymbol{x},\boldsymbol{y})) = \pi_{i}^{(\beta,q)} Z(\mathfrak{S}_{v,w}(\boldsymbol{x},\boldsymbol{y})).\] 
For any $v\in S_n$ and any simple reflection $s_i$ such that $\ell(vs_i) = \ell(v) - 1$, \[Z(\mathfrak{S}_{vs_i,w}(\boldsymbol{x},\boldsymbol{y})) = \twid{\pi}_{i}^{(\beta,q)} Z(\mathfrak{S}_{v,w}(\boldsymbol{x},\boldsymbol{y})).\]
\end{lemma}

\begin{proof}
Write $w$ in one line notation as $w = c_1c_2\cdots c_n$. Since $\ell(ws_i) = \ell(w)-1$, we must have $c_{i} > c_{i+1}$. We begin by evaluating the partition function of the following system:
 \[\begin{array}{c}
 \scalebox{.85}
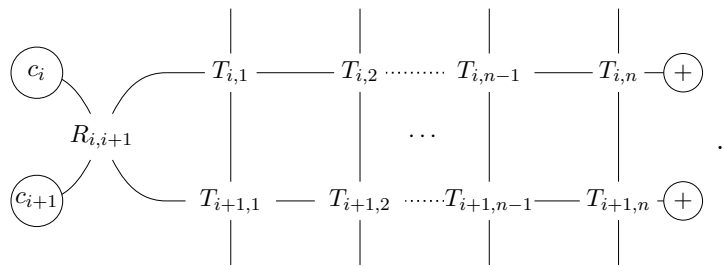
{\begin{tikzpicture}
  \draw (0,0)--(3,0);
  \draw (0,2)--(3,2);
  \draw [thick, dotted] (3,0)-- (5,0);
  \draw [thick, dotted] (3,2)-- (5,2);
  \draw (5,0) -- (8,0);
  \draw (5,2) -- (8,2);
  \draw (1,-1)--(1,3);
  \draw (3,-1)--(3,3);
  \draw (5,-1)--(5,3);
  \draw (7,-1)--(7,3);
  \coordinate (a1) at (-2,0);
  \coordinate (c1) at (0,2);
  \coordinate (a2) at (-2,2);
  \coordinate (c2) at (0,0);
  \draw (a1) to [out=0,in=180] (c1);
  \draw (a2) to [out=0,in=180] (c2);
  \path[fill=white] (-1,1) circle (.4);
  \draw[fill=white] (-2,2) circle (.4);
  \draw[fill=white] (-2,0) circle (.4);
  \draw[fill=white] (8,2) circle (.3);
  \draw[fill=white] (8,0) circle (.3);
  \path[fill=white] (1,2) ellipse (.4cm and .3cm);
  \path[fill=white] (1,0) ellipse (.7cm and .3cm);
  \path[fill=white] (3,2) ellipse (.4cm and .3cm);
  \path[fill=white] (3,0) ellipse (.7cm and .3cm);
  \path[fill=white] (5,2) ellipse (.7cm and .3cm);
  \path[fill=white] (5,0) ellipse (.7cm and .3cm);
  \path[fill=white] (7,2) ellipse (.4cm and .3cm);
  \path[fill=white] (7,0) ellipse (.6cm and .3cm);
  \node at (-1,1) {$R_{i,i+1}$};
  \node at (1,2) {$T_{i,1}$};
  \node at (1,0) {$T_{i+1,1}$};
  \node at (3,2) {$T_{i,2}$};
  \node at (3,0) {$T_{i+1,2}$};
  \node at (5,2) {$T_{i,n-1}$};
  \node at (5,0) {$T_{i+1,n-1}$};
  \node at (7,2) {$T_{i,n}$};
  \node at (7,0) {$T_{i+1,n}$};
  \node at (-2,2){$c_{i}$};
  \node at (-2,0){$c_{i+1}$};
  \node at (8,0){$+$};
  \node at (8,2){$+$};
  \node at (4, 1) {$\cdots$};
  \end{tikzpicture}}
  \end{array}.
  \]
Consulting the table of $R$-vertex weights for the row Yang-Baxter Equation (Figure \ref{RweightsDemazure}), we see that we have two options for the $R$-vertex on the left, namely type $\tt{c}_1$ and type $\tt{b}_1$. States with an $R$-vertex of type $\tt{c}_1$ don't flip the boundary conditions $c_i$ and $c_{i+1}$ on the other side of the $R$-vertex, so they will have boundary conditions $\mathfrak{S}_{v,w}$ on the remaining grid and will thus jointly contribute $(1-q^2)(1 + \beta x_{i+1}) Z(\mathfrak{S}_{v,w};\boldsymbol{x};\boldsymbol{y})$ to the partition function. States with an $R$-vertex of type $\tt{b}_1$ will flip the boundary conditions $c_i$ and $c_{i+1}$ on the other side of the $R$-vertex, so they will have boundary conditions $\mathfrak{S}_{v,ws_i}$ on the remaining grid and will contribute in total $(x_{i+1}-x_i)Z(\mathfrak{S}_{v,ws_i};\boldsymbol{x};\boldsymbol{y})$. 

By Theorem \ref{DemazureYBE}, we may repeatedly apply the Yang-Baxter equation to move the $R$-vertex to the right, column by column, according to the familiar train argument, to obtain the following lattice whose partition function matches that of the above system:
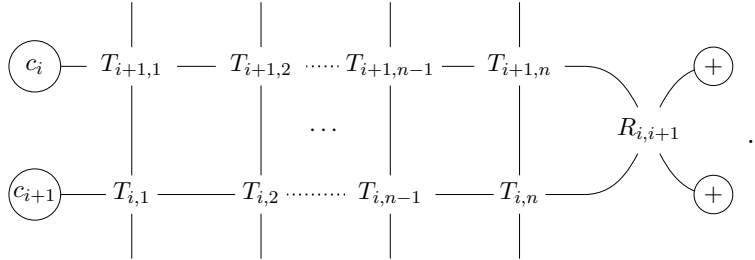
\[
\begin{array}{c}
\scalebox{.85}{\begin{tikzpicture}
  \draw (-0.5,0)--(3,0);
  \draw (-0.5,2)--(3,2);
  \draw [thick, dotted] (3,0)-- (5,0);
  \draw [thick, dotted] (3,2)-- (5,2);
  \draw (5,0) -- (8,0);
  \draw (5,2) -- (8,2);
  \draw (1,-1)--(1,3);
  \draw (3,-1)--(3,3);
  \draw (5,-1)--(5,3);
  \draw (7,-1)--(7,3);
  \coordinate (a1) at (8,0);
  \coordinate (c1) at (10,2);
  \coordinate (a2) at (8,2);
  \coordinate (c2) at (10,0);
  \draw (a1) to [out=0,in=180] (c1);
  \draw (a2) to [out=0,in=180] (c2);
  \path[fill=white] (9,1) circle (.4);
  \draw[fill=white] (-0.5,2) circle (.4);
  \draw[fill=white] (-0.5,0) circle (.4);
  \draw[fill=white] (10,2) circle (.3);
  \draw[fill=white] (10,0) circle (.3);
  \path[fill=white] (1,2) ellipse (.7cm and .3cm);
  \path[fill=white] (1,0) ellipse (.4cm and .3cm);
  \path[fill=white] (3,2) ellipse (.7cm and .3cm);
  \path[fill=white] (3,0) ellipse (.4cm and .3cm);
  \path[fill=white] (5,2) ellipse (.8cm and .3cm);
  \path[fill=white] (5,0) ellipse (.7cm and .3cm);
  \path[fill=white] (7,2) ellipse (.7cm and .3cm);
  \path[fill=white] (7,0) ellipse (.4cm and .3cm);
  \node at (9,1) {$R_{i,i+1}$};
  \node at (1,0) {$T_{i,1}$};
  \node at (1,2) {$T_{i+1,1}$};
  \node at (3,0) {$T_{i, 2}$};
  \node at (3,2) {$T_{i+1,2}$};
  \node at (5,0) {$T_{i,n-1}$};
  \node at (5,2) {$T_{i+1,n-1}$};
  \node at (7,0) {$T_{i,n}$};
  \node at (7,2) {$T_{i+1,n}$};
  \node at (-0.5,0){$c_{i+1}$};
  \node at (-0.5,2){$c_{i}$};
  \node at (10,0){$+$};
  \node at (10,2){$+$};
  \node at (4,1) {$\cdots$};
  \end{tikzpicture}}
  \end{array}.
  \]
Referring again to Figure \ref{RweightsDemazure}, there is only one possibility for the $R$-vertex on the right, namely type $\tt{a}$ with label $+$. Therefore, the $R$-vertex always has weight $(1+\beta x_i - q^2(1+ \beta x_{i+1}))$ and the rest of the system has boundary conditions for $\mathfrak{S}_{v,w}$, but with parameters $x_i$ and $x_{i+1}$ flipped. Equating these two expressions for the left and right hand side of the train argument and solving for $Z(\mathfrak{S}_{v,ws_i};\boldsymbol{x};\boldsymbol{y})$, we obtain the desired operator description for the first equation.

The second equation follows from applying analogous train arguments along the columns.
\end{proof}

Parallel results hold in the case where the simple reflection increases the length of $w$ or $v$, in which case the train argument results in the application of the operators $\pi_i^{-1}$ and $\twid{\pi}_i^{-1}$, respectively.

We can then combine these results to conclude that the lattice model gives the biaxial $(\beta,q)$-Grothendieck polynomials, as well as thereby specializing to our double and dual double $(\beta,q)$-Grothendieck polynomials.

\begin{theorem}\label{modelgivesgrotpolys}
 For any $v,w\in S_n$,
 \[ Z(\left(\mathfrak{S}_{v,w}(\boldsymbol{x},\boldsymbol{y})\right)) =  \mathcal{G}^{(\beta,q)}_{v,w}(\boldsymbol{x};\boldsymbol{y}).\]
\end{theorem}

\begin{proof}
We first show that for any $w \in S_n$, 
\[Z(\left(\mathfrak{S}_{1,w}(\boldsymbol{x},\boldsymbol{y})\right)) =  \mathcal{G}^{(\beta,q)}_w(\boldsymbol{x};\boldsymbol{y}).\]
This follows from applying the length-decreasing operators of Lemma \ref{permutationdecrease} to the base case for $w_0$ in Proposition \ref{latticebasecase}. Then the result follows by applying $\twid{\pi}_i^{-1}$ to the identity above.
\end{proof}

\begin{proposition}\label{subprop:modelgivesgrotpolys}
For any $w\in S_n$,
\[Z(\left(\mathfrak{S}_{w_0w^{-1},w}(\boldsymbol{x},\boldsymbol{y})\right)) =  \mathcal{H}^{(\beta,q)}_w(\boldsymbol{x};\boldsymbol{y}),\]
\end{proposition}

\begin{proof}
Note that the base case for $w_0$ is the same as in Proposition \ref{latticebasecase}. Then, the length-increasing operator acting on the columns on the model is  $\twid{\pi}_i^{-1}$, which gives us \[Z\left(\mathfrak{S}_{w_0w^{-1}s_i,w_0}(\boldsymbol{x},\boldsymbol{y})\right) = Z\left(\mathfrak{S}_{w_0(s_iw)^{-1},w_0}(\boldsymbol{x},\boldsymbol{y})\right).\]
The resulting action on the permutation $w$ is length-decreasing and acts on the left, which matches the action of $\twid{\pi}_i^{-1}$ on the dual polynomials from Definition \ref{dualdef}.
\end{proof}

\subsection{Interesting Specializations when $q=0$}

Now we are ready to show that our six-vertex model simultaneously $q$-deforms both the $\beta$-Grothendieck and dual $\beta$-Grothendieck polynomials. Along with this, we will investigate left and right actions of Demazure operators on our model. In the case of the $\mathcal{G}^{(\beta)}$, these actions will follow from the specialization result, while for the $\mathcal{H}^{(\beta)}$ the reverse is true. The reason for this discrepancy was just a convention choice. Our Drinfeld twist was chosen to give the $\beta$-Grothendieck polynomials immediately, while the inverse twist would have given the dual $\beta$-Grothendieck polynomials more directly.

\begin{proposition}
\label{hudson}
Upon setting $q = 0$, the partition function of the chromatic model specializes to the double $\beta$-Grothendieck polynomials in two ways, namely 
\[Z(\left(\mathfrak{S}_{1,w}(\boldsymbol{x},\boldsymbol{y})\right))|_{q=0}=  \mathcal{G}^{(\beta)}_w(\boldsymbol{x};\boldsymbol{y}) = Z(\left(\mathfrak{S}_{1,w^{-1}}(\boldsymbol{y},\boldsymbol{x})\right))|_{q=0} .\] 
\end{proposition}

\begin{proof}
The first equality follows from Theorem \ref{modelgivesgrotpolys} and the definition of $\mathcal{G}_w^{(\beta,q)}$ as a $q$-deformation given in Definition \ref{grotpoly}. The second equality follows from first applying the following identity for $\beta$-Grothendieck polynomials (see for example the Appendix of \cite{Hudson})
\begin{align}\label{hudsoneq}
    \mathcal{G}_w^{(\beta)}(\boldsymbol{x}; \boldsymbol{y}) = \mathcal{G}_{w^{-1}}^{(\beta)}(\boldsymbol{y}; \boldsymbol{x})
\end{align}
then again applying Theorem \ref{modelgivesgrotpolys} under the specialization $q=0$.
\end{proof}

Combining Lemma \ref{permutationdecrease} with Proposition \ref{hudson}, we obtain an alternate recursive definition for the $\beta$-Grothendieck polynomials via left actions.

\begin{corollary}\label{canusetheYstoo!}
If we  consider $\pi_{i,y}^{(\beta)} := (\beta^2q^2)\twid{\pi}_i^{(\beta,q)}|_{q=0}$ to be the $\beta$-deformed divided difference operator acting on the $y$ variables (as opposed to the $x$-variable definition in Section \ref{grotpoly}), we have
\[\mathcal{G}^{(\beta)}_{s_iw}(\boldsymbol{x};\boldsymbol{y}) = \pi_{i,y}^{(\beta)} (\mathcal{G}^{(\beta)}_w(\boldsymbol{x};\boldsymbol{y}))\]
whenever $\ell(s_iw) = \ell(w) - 1$.
\end{corollary}

Next, we prove the corresponding results for the dual polynomials. In doing so, we justify the name choice: our dual polynomials specialize when $q=0$ to the dual $\beta$-Grothendieck polynomials.

\begin{theorem}
We have
\[ \mathcal{H}^{(\beta,q)}_w(\boldsymbol{x};\boldsymbol{y})|_{q=0} = \mathcal{H}^{(\beta)}_w(\boldsymbol{x};\boldsymbol{y}).
\]
Then, for any $w\in S_n$, there are two boundary system conditions on the chromatic model that give the dual double $\beta$-Grothendieck polynomials. That is,
\[Z(\left(\mathfrak{S}_{w_0w^{-1},w_0}(\boldsymbol{x},\boldsymbol{y})\right))  = \mathcal{H}^{(\beta)}_w(\boldsymbol{x};\boldsymbol{y}) = Z(\left(\mathfrak{S}_{w_0w,w_0}(\boldsymbol{y},\boldsymbol{x})\right)).\]
\label{modelgivesdual}
\end{theorem}

Proving this theorem requires two key lemmas. The first lemma is a slight generalization of an identity between Grothendieck polynomials ($\beta = -1$) and their duals proven by Lenart, Robinson, and Sottile \cite[Corollary~6.26]{LenartRobinsonSottile}, inspired by a similar identity proven by Lascoux \cite{Lascoux-anneau}. The second lemma is an analogue of (\ref{hudsoneq}) for the dual $(\beta,q)$-Grothendieck polynomials under $q=0$, which provides us with a second recursive formula for these polynomials.
Recall that our formal group law notation implies that $\ominus x_j:=\frac{-x_j}{1+\beta x_j}$.

\begin{lemma}\label{dualtoregular}
Specializing $q=0$, the $(\beta,q)$-Grothendieck polynomials satisfy the following relation with the $\beta$-Grothendieck polynomials:
\begin{align}\mathcal{H}^{(\beta,q)}_w(\boldsymbol{x};\boldsymbol{y})|_{q=0} = (-1)^{\ell(w)}\prod_{i+j\leq n} (1+\beta(x_i\oplus y_j))\cdot \mathcal{G}^{(\beta)}_w(\ominus \boldsymbol{x};\ominus \boldsymbol{y}). \label{relate-G-H-equation}\end{align}
\end{lemma}

\begin{proof}
First, note that there is a correspondence of states between the systems $\mathfrak{S}_{w_0w^{-1},w_0}(\boldsymbol{x};\boldsymbol{y})$ and $\mathfrak{S}_{1,w^{-1}}(\boldsymbol{y};\boldsymbol{x}).$ Given a state for the system $\mathfrak{S}_{1,w^{-1}}(\boldsymbol{y};\boldsymbol{x}),$ flip over the diagonal to obtain a state for $\mathfrak{S}_{w^{-1},1}(\boldsymbol{x};\boldsymbol{y}).$ Then, reverse the order of the colors, which corresponds to multiplying the boundary conditions on the left by $w_0$, and relabel internal paths accordingly, which obtains a state for $\mathfrak{S}_{w_0w^{-1},w_0}(\boldsymbol{x};\boldsymbol{y}).$

To maintain the same partition function, we must apply analogous transformations to the set of weights: flipping over the diagonal swaps the weights for $\tt{b}_1$ and $\tt{b}_2$, but then reversing the order of the colors swaps them back, in addition to swapping $\tt{c}_1$ and $\tt{c}_2$. To account for the swap of $\tt{c}_1$ and $\tt{c}_2$, we substitute in $x\mapsto \ominus x, y\mapsto \ominus y$. Note that $(\ominus x) \oplus (\ominus y) = -\frac{x\oplus y}{1 + \beta (x \oplus y)}$, so up to denominator, this substitution switches the type $\tt{c}$ vertices back, while only affecting other vertices by a sign in $\tt{b}_1$ (See Figure \ref{pipeweights1mod}).

\begin{figure}[h]
\centering
\scalebox{.95}{
$
\begin{array}{c@{\hspace{10pt}}c@{\hspace{10pt}}c@{\hspace{10pt}}c@{\hspace{10pt}}c@{\hspace{25pt}}}
\toprule
\tt{a}&\tt{b}_1&\tt{b_2}&\tt{c}_1&\tt{c}_2\\
\midrule
\begin{tikzpicture}
\coordinate (a) at (-.75, 0);
\coordinate (b) at (0, .75);
\coordinate (c) at (.75, 0);
\coordinate (d) at (0, -.75);
\coordinate (aa) at (-.75,.5);
\coordinate (cc) at (.75,.5);
\draw[line width=0.5mm, violet] (a)--(0,0);
\draw[line width=0.6mm, violet] (b)--(0,0);
\draw[line width=0.5mm, violet] (c)--(0,0);
\draw[line width=0.6mm, violet] (d)--(0,0);
\draw[line width=0.5mm, violet,fill=white] (a) circle (.25);
\draw[line width=0.5mm, violet,fill=white] (b) circle (.25);
\draw[line width=0.5mm, violet, fill=white] (c) circle (.25);
\draw[line width=0.5mm, violet, fill=white] (d) circle (.25);
\node at (0,1) { };
\node at (a) {$c$};
\node at (b) {$c$};
\node at (c) {$c$};
\node at (d) {$c$};
\end{tikzpicture}
&
\begin{tikzpicture}
\coordinate (a) at (-.75, 0);
\coordinate (b) at (0, .75);
\coordinate (c) at (.75, 0);
\coordinate (d) at (0, -.75);
\coordinate (aa) at (-.75,.5);
\coordinate (cc) at (.75,.5);
\draw[line width=0.5mm, blue] (a)--(c);
\draw[line width=0.6mm, red] (b)--(d);
\draw[line width=0.5mm,blue,fill=white] (a) circle (.25);
\draw[line width=0.5mm,blue,fill=white] (c) circle (.25);
\draw[line width=0.5mm,red,fill=white] (b) circle (.25);
\draw[line width=0.5mm,red,fill=white] (d) circle (.25);
\node at (0,1) { };
\node at (a) {$b$};
\node at (b) {$a$};
\node at (c) {$b$};
\node at (d) {$a$};
\end{tikzpicture}
&
\begin{tikzpicture}
\coordinate (a) at (-.75, 0);
\coordinate (b) at (0, .75);
\coordinate (c) at (.75, 0);
\coordinate (d) at (0, -.75);
\coordinate (aa) at (-.75,.5);
\coordinate (cc) at (.75,.5);
\draw[line width=0.5mm, red] (a)--(c);
\draw[line width=0.6mm, blue] (b)--(d);
\draw[line width=0.5mm,red,fill=white] (a) circle (.25);
\draw[line width=0.5mm,red,fill=white] (c) circle (.25);
\draw[line width=0.5mm,blue,fill=white] (b) circle (.25);
\draw[line width=0.5mm,blue,fill=white] (d) circle (.25);
\node at (0,1) { };
\node at (a) {$a$};
\node at (b) {$b$};
\node at (c) {$a$};
\node at (d) {$b$};
\end{tikzpicture}
%%%%%%%
& \begin{tikzpicture}
\coordinate (a) at (-.75, 0);
\coordinate (b) at (0, .75);
\coordinate (c) at (.75, 0);
\coordinate (d) at (0, -.75);
\coordinate (aa) at (-.75,.5);
\coordinate (cc) at (.75,.5);
\draw[line width=0.5mm, blue](a)--(0,0)--(b);
\draw[line width=0.5mm,blue,fill=white] (b) circle (.25);
\draw[line width=0.5mm,blue,fill=white] (a) circle (.25);
\draw[line width=0.5mm, red](d)--(0,0)--(c);
\draw[line width=0.5mm,red,fill=white] (c) circle (.25);
\draw[line width=0.5mm,red,fill=white] (d) circle (.25);
\node at (0,1) { };
\node at (a) {$b$};
\node at (b) {$b$};
\node at (c) {$a$};
\node at (d) {$a$};
\end{tikzpicture}
%%%%%%%
& \begin{tikzpicture}
\coordinate (a) at (-.75, 0);
\coordinate (b) at (0, .75);
\coordinate (c) at (.75, 0);
\coordinate (d) at (0, -.75);
\coordinate (aa) at (-.75,.5);
\coordinate (cc) at (.75,.5);
\draw[line width=0.5mm, red](a)--(0,0)--(b);
\draw[line width=0.5mm,red,fill=white] (b) circle (.25);
\draw[line width=0.5mm,red,fill=white] (a) circle (.25);
\draw[line width=0.5mm, blue](d)--(0,0)--(c);
\draw[line width=0.5mm,blue,fill=white] (c) circle (.25);
\draw[line width=0.5mm,blue,fill=white] (d) circle (.25);
\node at (0,1) { };
\node at (a) {$a$};
\node at (b) {$a$};
\node at (c) {$b$};
\node at (d) {$b$};
\end{tikzpicture}
%%%%%%%%%
\\
  \midrule
 1 & -\frac{x_i\oplus y_j}{1+ \beta (x_i\oplus y_j)} & 0 & \frac{1}{1+\beta (x_i\oplus y_j)} & 1 \\
  \bottomrule
\end{array}
$}
\caption{The Boltzmann weights from Figure \ref{pipeweights1} under the substitution $x\mapsto \ominus x, y\mapsto \ominus y$ and $q=0$.}
\label{pipeweights1mod}
\end{figure}

Comparing states of $\mathfrak{S}_{w_0w^{-1},w_0}(\boldsymbol{x};\boldsymbol{y})$ under this weight system to those under our usual weights of Figure \ref{pipeweights1}, we see that each $\tt{b}_1$ vertex will contribute to the length of $w$ and add a factor of $(-1)$. Furthermore, since we are in the $q=0$ case, every vertex above the antidiagonal is of type $\tt{b}_1$ or $\tt{c}_1$, and each of these weights differs between the systems by a factor of $(1+\beta(x_i\oplus y_j))$. Correcting for these factors, we then obtain the desired equation.
\end{proof}

\begin{lemma}\label{Hidentity}
We have an analogous relation to (\ref{hudsoneq}) for the dual polynomials: that is, $$\mathcal{H}^{(\beta,q)}_w(\boldsymbol{x};\boldsymbol{y})|_{q=0}=\mathcal{H}^{(\beta,q)}_{w^{-1}}(\boldsymbol{y};\boldsymbol{x})|_{q=0}.$$
As in Corollary \ref{canusetheYstoo!} for the $\beta$-Grothendieck polynomials, we obtain a second recursive definition for this specialization $\mathcal{H}^{(\beta,q)}_w(\boldsymbol{x};\boldsymbol{y})|_{q=0}$ via a right action. That is, if we let $\mu^{(\beta)}_{i,x}:= (\beta^2q^2) \pi_i^{-1}|_{q=0}$ be the dual divided difference operators acting on the $x$-variables, we obtain
\[ \mathcal{H}^{(\beta,q)}_{ws_i}(\boldsymbol{x};\boldsymbol{y})|_{q=0} = \mu^{(\beta)}_{i,x}( \mathcal{H}^{(\beta,q)}_w(\boldsymbol{x};\boldsymbol{y})|_{q=0})\]
in the case that $\ell(s_iw) = \ell(w) -1$.
\end{lemma}

\begin{proof}
Beginning with Lemma \ref{dualtoregular}, apply the identity (\ref{hudsoneq}):
\begin{align*}
\mathcal{H}^{(\beta,q)}_w(\boldsymbol{x};\boldsymbol{y})|_{q=0} &= (-1)^{\ell(w)} \prod_{i+j\leq n} (1+\beta(x_i\oplus y_j))\cdot \mathcal{G}^{(\beta)}_{w^{-1}}(\ominus \boldsymbol{y};\ominus \boldsymbol{x})
\intertext{Then, since the product term is symmetric in $x$ and $y$, and $\ell(w) = \ell(w^{-1})$,}
&= (-1)^{\ell(w^{-1})} \prod_{i+j\leq n} (1+\beta(y_i\oplus x_j))\cdot \mathcal{G}^{(\beta)}_{w^{-1}}(\ominus \boldsymbol{y};\ominus \boldsymbol{x})\\ &= \mathcal{H}^{(\beta,q)}_{w^{-1}}(\boldsymbol{y};\boldsymbol{x})|_{q=0}.
\end{align*}
The new recursion then follows from passing the divided difference operators of Definition \ref{dualdef} through this identity.
\end{proof}

\begin{proof}[Proof of Theorem \ref{modelgivesdual}]
The base case $\mathcal{H}^{(\beta,q)}_{w_0}(\boldsymbol{x};\boldsymbol{y})$ specializes to $\mathcal{H}^{(\beta)}_{w_0}(\boldsymbol{x};\boldsymbol{y})$. Furthermore, note that upon setting $q=0$, our operator $\mu_{i,x}^{(\beta)} = \mu_i$, the operator from Definition \ref{dual-polynomials-definition}, so the first statement follows. The second identity then follows from combining Proposition \ref{subprop:modelgivesgrotpolys} with the first identity and Lemma \ref{dualtoregular}.
\end{proof}

\section{Correspondence with Pipe Dreams} \label{pipedreamcorrespondencesection}
In Proposition~\ref{hudson} of the prior section, we proved that when $q=0$ two different specializations of the chromatic model give the double $\beta$-Grothendieck polynomials according to their definition in terms of divided difference operators. In this section, we explain the correspondence between these specializations and pipe dreams, thus providing another way of seeing that the pipe dream and divided difference definitions of Grothendieck polynomials are equivalent.  In particular, the admissible states of both specializations biject easily with reduced pipe dreams. However, the weights of corresponding states account for different sets of nonreduced pipe dreams, corresponding to the two different (yet equivalent) methods for reducing pipe dreams. See Section 6 of \cite{LenartRobinsonSottile} for an explanation of the equivalence of the two reductions. Since the Grothendieck polynomials are obtained from setting $q=0$, the entirety of this section we will assume $q=0$.

\begin{proposition}\label{reducedpipebijection}
For $w\in S_n$, there exists a bijection between admissible lattice model states in $\mathfrak{S}_{1,w^{-1}}(\boldsymbol{y},\boldsymbol{x})$ with $q=0$ and reduced pipe dreams for $w$.
\end{proposition}

\begin{proof}
Given a reduced pipe dream for a permutation $w$, overlay it on a square lattice grid. Then label pipes on the left boundary by $1,2,...,n$ from top to bottom. Assign a label of $i$ on an edge if the edge is along the strand of pipe $i$, and label $+$ if no pipe travels along that edge. Depending on the labels of the pipes involved, bent tiles \raisebox{-1 ex}{\begin{tikzpicture}[scale=.5]
    %grid
    \draw (0,0) -- (1,0);
    \draw (0,1) -- (1,1);
    \draw (0,0) -- (0,1);
    \draw (1,0) -- (1,1);
    %pipes
    \draw[rounded corners = 2mm, line width = .5mm] (0,.5) -- (.5,.5)-- (.5,1);
    \draw[rounded corners = 2mm, line width = .5mm] (.5,0) --(.5,.5)-- (1,.5);
\end{tikzpicture}}
will form vertices of types $\tt{c}_1$ and $\tt{c}_2$. Crossing tiles 
\raisebox{-1 ex}{\begin{tikzpicture}[scale=.5]
    %grid
    \draw (0,0) -- (1,0);
    \draw (0,1) -- (1,1);
    \draw (0,0) -- (0,1);
    \draw (1,0) -- (1,1);
    %pipes
    \draw[rounded corners = 2mm, line width = .5mm] (0,.5)--(1,.5);
    \draw[rounded corners = 2mm, line width = .5mm] (.5,0) --(.5,1);
\end{tikzpicture}} give vertices of type $\tt{b}_2$, since we have at most one crossing of any pair of pipes and according to our labeling, these cannot be of type $\tt{b}_1$. Blank tiles will give vertices of type $\tt{a}$. Furthermore, in the pipe dream definition, the permutation $w = w(1) \, w(2) \, w(3) \cdots w(n)$ corresponds to the $x$-coordinates of pipes on the top boundary; examining the newly-built lattice model, we see that color $i$ will appear in the $w(i)$-th place on the top boundary. Thus, from the reduced pipe dream, we have a state $\mathfrak{s}$ of the model $\mathfrak{S}_{w^{-1},1}(\boldsymbol{x},\boldsymbol{y})$, where crossing tiles in the pipe dream correspond to vertices of type $\tt{b}_2$, of weight zero (as $q=0$ in this section). However, the weight of a crossing tile of a pipe dream in row $i$ and column $j$ is $x_i \oplus y_j$. Thus, we reflect the state $\mathfrak{s}$ along the line $y=-x$, thereby swapping the label on the top and left boundaries, the row and column variables, and vertices of type $\tt{b}_2$ with ones of type $\tt{b}_1$. The result is a state of $\mathfrak{S}_{1,w^{-1}}(\boldsymbol{y},\boldsymbol{x})$.

Conversely, given an admissible state of $\mathfrak{S}_{1,w^{-1}}(\boldsymbol{y},\boldsymbol{x})$, again reflect the state across the diagonal, and the lack of a vertex of type $\tt{b}_2$ in admissible states of $\mathfrak{S}_{1,w^{-1}}(\boldsymbol{y},\boldsymbol{x})$ means the reflected state has no vertices of type $\tt{b}_1$. This forces the colored strands to only tile on or above above the diagonal and any pair of pipes to cross at most once, so we have a reduced pipe dream. A similar argument to the above tells us that this is a reduced pipe dream for $w$, so we obtain a bijection between admissible states of $\mathfrak{S}_{1,w^{-1}}(\boldsymbol{y},\boldsymbol{x})$ and reduced pipe dreams for $w$. 
\end{proof}

One of the advantages of the lattice model point of view is that we can use a single state to account for nonreduced pipe dreams as well as reduced ones by using the Boltzmann weights. In particular, one of the drawbacks of a pipe dream phrasing of the $\beta$-Grothendieck polynomials is that it is difficult to immediately read off the permutation corresponding to a nonreduced pipe dream and, relatedly, to obtain all nonreduced pipe dreams corresponding to a certain permutation. By incorporating the reduction process into the weights, that difficulty is removed. That is, our set of weights is constructed such that the weight of an admissible state of $\mathfrak{S}_{1,w^{-1}}(\boldsymbol{y},\boldsymbol{x})$ is equal to the sum of the weights of all pipe dreams with reduction equal to its corresponding reduced pipe dream. An example of this correspondence is seen in Figure \ref{pipedreamsmodel2}, while Figure \ref{reductionisweird} gives an example with a particularly subtle reduction.

\begin{proposition}[See \cite{KnutsonMiller-subwords,LenartRobinsonSottile}] \label{pipedreamproposition} For any $w \in S_n$,
\[\mathcal{G}^{(\beta)}_w(\boldsymbol{x};\boldsymbol{y}) = \sum_{P\in PD(w)} \emph{\text{wt}}(P) \beta^{\emph{\text{ex}}(P)}.\]
\end{proposition}

\begin{proof}
Consider the strands in a state $\mathfrak{s}$ of $\mathfrak{S}_{1,w^{-1}}(\boldsymbol{y},\boldsymbol{x})$: by Proposition \ref{reducedpipebijection}, these form a unique reduced pipe dream $P$ associated to $w$. In this bijection, a vertex of weight $x_i\oplus y_j$ in $\mathfrak{s}$ corresponds to a crossing tile in row $i$ and column $j$ of $P$ (we call this location $(i,j)$) of $P$. In addition, the vertices of weight $1+\beta (x_i\oplus y_j)$ correspond to tiles in $P$ where there is no crossing, but where the two pipes on the tile have already crossed further southwest. Swapping this tile out for a crossing tile produces a nonreduced pipe dream that reduces to $P$ with one extra crossing in location $(i,j)$, thus with excess 1 and weight $\text{wt}(P)\cdot (x_i\oplus y_j)$.

This swapping may be done independently at all the vertices with weights $1+\beta (x_i\oplus y_j)$; swapping any such tiles in locations $(i_1,j_1),\ldots (i_m,j_m)$ to crossing tiles corresponds to a pipe dream $P'$ that reduces to $P$ with excess $m$ and weight $\text{wt}(P') = \text{wt}(P)\cdot \prod_{k=1}^m (x_{i_k}\oplus y_{j_k})$. Conversely, every pipe dream that reduces to $P$ can be constructed in this way, by adding back in an extra crossing of a pair of pipes.

In other words, flipping one of these tiles at a vertex of weight $1+\beta (x_i\oplus y_j)$ corresponds, in the partition function, to choosing the $\beta (x_i\oplus y_j)$ part of the weight. On the other hand, not flipping the tile corresponds to choosing the 1 part of the weight. The binomial theorem tells us that \[\text{wt}(\mathfrak{s}) = \sum_{P'\to P} \text{wt}(P')\beta^{\text{ex}(P')},\] where the sum is over all pipe dreams $P'$ that reduce to $P$. Since every pipe dream has exactly one reduction, and since reduced pipe dreams are in bijection with states in our lattice model, we can sum over states and (combining with the result of Theorem~\ref{hudson}) arrive at our result: \[\mathcal{G}^{(\beta)}_w(\boldsymbol{x};\boldsymbol{y}) = Z(\mathfrak{S}_{1,w^{-1}}(\boldsymbol{y},\boldsymbol{x})) = \sum_{\mathfrak{s}\in \mathfrak{S}_{1,w^{-1}}(\boldsymbol{y},\boldsymbol{x})} \text{wt}(\mathfrak{s}) = \sum_{PD(w)} \text{wt}(P)\beta^{\text{ex}(P)}. \qedhere \]
\end{proof}

\begin{example}
We follow the same convention as \cite{PechenikSearles} for pipe dream reduction, in which we consider each pair of pipes and eliminate all crossings between them except for the most southwestern crossing. One problem that arises is the order in which we reduce pipes. Consider the nonreduced pipe dream from Figure \ref{pipedreamreduction}, reproduced on the far left in Figure \ref{reductionisweird}. It is not immediately clear which pair of pipes we should consider first, as there is are double crossings of pipes 2 and 4 as well as pipes 3 and 4 (following the naming convention of Proposition \ref{reducedpipebijection}). As the reader may verify, this choice drastically affects what reduced pipe dream we obtain and thus to what permutation the nonreduced pipe dream corresponds. The lattice model eliminates the need to define a convention for choice of pairs, since it starts from the reduced pipe dream. However, it ``sees" the eliminated crossing of pipes 2 and 4 in the corresponding reduced pipe dream in the vertex at row $2$ and column $1$, and the weight of that vertex $(1 + \beta (x_1 \oplus y_2))$ gives a term for the reduced pipe dream (coming from the summand $1$) and a term for the nonreduced pipe dream (coming from $\beta (x_1 \oplus y_2$)).
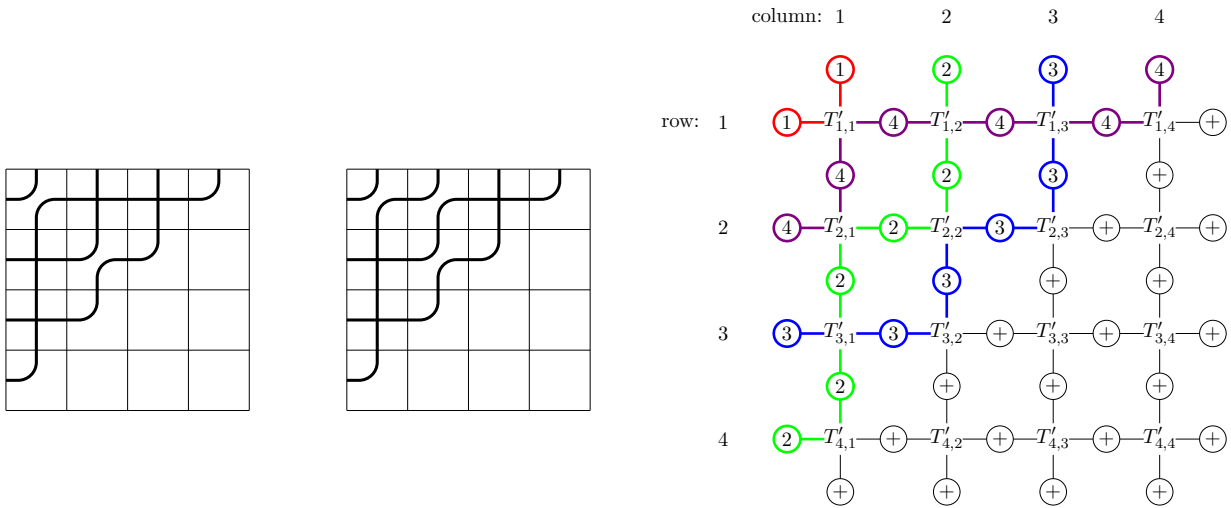
\begin{figure}[h]
    \centering
\raisebox{1.25 cm}{\scalebox{.8}{\begin{tikzpicture}
%grid
\draw (0,0) -- (4,0);
\draw (0,1) -- (4,1);
\draw (0,2) -- (4,2);
\draw (0,3) -- (4,3);
\draw (0,4) -- (4,4);
\draw (0,0) -- (0,4);
\draw (1,0) -- (1,4);
\draw (2,0) -- (2,4);
\draw (3,0) -- (3,4);
\draw (4,0) -- (4,4);
%pipes
\draw[rounded corners = 3mm, line width = .5mm] (0,3.5) -- (.5,3.5) -- (.5, 4);
\draw[rounded corners= 3mm, line width = .5mm] (0, 2.5) -- (1.5,2.5) -- (1.5,4);
\draw[rounded corners = 3mm, line width = .5mm] (0,1.5) -- (1.5,1.5) -- (1.5,2.5)  -- (2.5,2.5) -- (2.5,4);
\draw[rounded corners = 3mm, line width = .5mm] (0,.5) -- (.5,.5) -- (.5,3.5)  -- (3.5,3.5) -- (3.5,4);
\end{tikzpicture}}}
\hspace{1cm}
\raisebox{1.25 cm}{\scalebox{.8}{\begin{tikzpicture}
%grid
\draw (0,0) -- (4,0);
\draw (0,1) -- (4,1);
\draw (0,2) -- (4,2);
\draw (0,3) -- (4,3);
\draw (0,4) -- (4,4);
\draw (0,0) -- (0,4);
\draw (1,0) -- (1,4);
\draw (2,0) -- (2,4);
\draw (3,0) -- (3,4);
\draw (4,0) -- (4,4);
%pipes
\draw[rounded corners = 3mm, line width = .5mm] (0,3.5) -- (.5,3.5) -- (.5, 4);
\draw[rounded corners= 3mm, line width = .5mm] (0, 2.5) -- (1.5,2.5) -- (1.5,3.5)--(3.5,3.5)--(3.5,4);
\draw[rounded corners = 3mm, line width = .5mm] (0,1.5) -- (1.5,1.5) -- (1.5,2.5)  -- (2.5,2.5) -- (2.5,4);
\draw[rounded corners = 3mm, line width = .5mm] (0,.5) -- (.5,.5) -- (.5,3.5)  -- (1.5,3.5) -- (1.5,4);
\end{tikzpicture}}}
\hspace{.5cm}
\scalebox{.7}{
\begin{tikzpicture}
  \coordinate (ab) at (1,0);
  \coordinate (ad) at (3,0);
  \coordinate (af) at (5,0);
  \coordinate (ah) at (7,0);
  \coordinate (ba) at (0,1);
  \coordinate (bc) at (2,1);
  \coordinate (be) at (4,1);
  \coordinate (bg) at (6,1);
  \coordinate (bi) at (8,1);
  \coordinate (cb) at (1,2);
  \coordinate (cd) at (3,2);
  \coordinate (cf) at (5,2);
  \coordinate (ch) at (7,2);
  \coordinate (da) at (0,3);
  \coordinate (dc) at (2,3);
  \coordinate (de) at (4,3);
  \coordinate (dg) at (6,3);
  \coordinate (di) at (8,3);
  \coordinate (eb) at (1,4);
  \coordinate (ed) at (3,4);
  \coordinate (ef) at (5,4);
  \coordinate (eh) at (7,4);
  \coordinate (fa) at (0,5);
  \coordinate (fc) at (2,5);
  \coordinate (fe) at (4,5);
  \coordinate (fg) at (6,5);
  \coordinate (fi) at (8,5);
  \coordinate (gb) at (1,6);
  \coordinate (gd) at (3,6);
  \coordinate (gf) at (5,6);
  \coordinate (gh) at (7,6);
  \coordinate (ha) at (0,7);
  \coordinate (hc) at (2,7);
  \coordinate (he) at (4,7);
  \coordinate (hg) at (6,7);
  \coordinate (hi) at (8,7);
  \coordinate (ib) at (1,8);
  \coordinate (id) at (3,8);
  \coordinate (if) at (5,8);
  \coordinate (ih) at (7,8);
  \coordinate (bb) at (1,1);
  \coordinate (bd) at (3,1);
  \coordinate (bf) at (5,1);
  \coordinate (bh) at (7,1);
  \coordinate (db) at (1,3);
  \coordinate (dd) at (3,3);
  \coordinate (df) at (5,3);
  \coordinate (dh) at (7,3);
  \coordinate (fb) at (1,5);
  \coordinate (fd) at (3,5);
  \coordinate (ff) at (5,5);
  \coordinate (fh) at (7,5);
  \coordinate (hb) at (1,7);
  \coordinate (hd) at (3,7);
  \coordinate (hf) at (5,7);
  \coordinate (hh) at (7,7);
  \coordinate (bax) at (0,1.5);
  \coordinate (bcx) at (2,1.5);
  \coordinate (bex) at (4,1.5);
  \coordinate (bgx) at (6,1.5);
  \coordinate (bix) at (8,1.5);
  \coordinate (dax) at (0,3.5);
  \coordinate (dcx) at (2,3.5);
  \coordinate (dex) at (4,3.5);
  \coordinate (dgx) at (6,3.5);
  \coordinate (dix) at (8,3.5);
  \coordinate (fax) at (0,5.5);
  \coordinate (fcx) at (2,5.5);
  \coordinate (fex) at (4,5.5);
  \coordinate (fgx) at (6,5.5);
  \coordinate (fix) at (8,5.5);
  \coordinate (hax) at (0,7.5);
  \coordinate (hcx) at (2,7.5);
  \coordinate (hex) at (4,7.5);
  \coordinate (hgx) at (6,7.5);
  \coordinate (hix) at (8,7.5);
  \draw (ab)--(ib);
  \draw (ad)--(id);
  \draw (af)--(if);
  \draw (ah)--(ih);
  \draw (ba)--(bi);
  \draw (da)--(di);
  \draw (fa)--(fi);
  \draw (ha)--(hi);
  \draw[line width=0.5mm,green] (ba)--(bb)--(fb)--(fd)--(id);
  \draw[line width=0.5mm,blue] (da)--(dd)--(fd)--(ff)--(if);
  \draw[line width=0.5mm,violet] (fa)--(fb)--(hb)--(hh)--(ih);
  \draw[line width=0.5mm,red] (ha)--(hb)--(ib);
  \draw[fill=white] (ab) circle (.25);
  \draw[fill=white] (ad) circle (.25);
  \draw[fill=white] (af) circle (.25);
  \draw[fill=white] (ah) circle (.25);
  \draw[line width=0.5mm,green,fill=white] (ba) circle (.25);
  \draw[fill=white] (bc) circle (.25);
  \draw[fill=white] (be) circle (.25);
  \draw[fill=white] (bg) circle (.25);
  \draw[fill=white] (bi) circle (.25);
  \draw[line width=0.5mm,green,fill=white] (cb) circle (.25);
  \draw[fill=white] (cd) circle (.25);
  \draw[fill=white] (cf) circle (.25);
  \draw[fill=white] (ch) circle (.25);
  \draw[line width=0.5mm,blue,fill=white] (da) circle (.25);
  \draw[line width=0.5mm,blue,fill=white] (dc) circle (.25);
  \draw[fill=white] (de) circle (.25);
  \draw[fill=white] (dg) circle (.25);
  \draw[fill=white] (di) circle (.25);
  \draw[line width=0.5mm,green,fill=white] (eb) circle (.25);
  \draw[line width=0.5mm,blue,fill=white] (ed) circle (.25);
  \draw[fill=white] (ef) circle (.25);
  \draw[fill=white] (eh) circle (.25);
  \draw[line width=0.5mm,violet,fill=white] (fa) circle (.25);
  \draw[line width=0.5mm,green,fill=white] (fc) circle (.25);
  \draw[line width=0.5mm,blue,fill=white] (fe) circle (.25);
  \draw[fill=white] (fg) circle (.25);
  \draw[fill=white] (fi) circle (.25);
  \draw[line width=0.5mm,violet,fill=white] (gb) circle (.25);
  \draw[line width=0.5mm,green,fill=white] (gd) circle (.25);
  \draw[line width=0.5mm,blue,fill=white] (gf) circle (.25);
  \draw[fill=white] (gh) circle (.25);
  \draw[line width=0.5mm,red,fill=white] (ha) circle (.25);
  \draw[line width=0.5mm,violet,fill=white] (hc) circle (.25);
  \draw[line width=0.5mm,violet,fill=white] (he) circle (.25);
  \draw[line width=0.5mm,violet,fill=white] (hg) circle (.25);
  \draw[fill=white] (hi) circle (.25);
  \draw[line width=0.5mm,red,fill=white] (ib) circle (.25);
  \draw[line width=0.5mm,green,fill=white] (id) circle (.25);
  \draw[line width=0.5mm,blue,fill=white] (if) circle (.25);
  \draw[line width=0.5mm,violet,fill=white] (ih) circle (.25);
 \path[fill=white] (bb) circle (.3);
  \path[fill=white] (bd) circle (.3);
  \path[fill=white] (bf) circle (.3);
  \path[fill=white] (bh) circle (.3);
  \path[fill=white] (db) circle (.3);
  \path[fill=white] (dd) circle (.3);
  \path[fill=white] (df) circle (.3);
  \path[fill=white] (dh) circle (.3);
  \path[fill=white] (fb) circle (.3);
  \path[fill=white] (fd) circle (.3);
  \path[fill=white] (ff) circle (.3);
  \path[fill=white] (fh) circle (.3);
   \path[fill=white] (hb) circle (.3);
  \path[fill=white] (hd) circle (.3);
  \path[fill=white] (hf) circle (.3);
  \path[fill=white] (hh) circle (.3);
  \node at (bb) {$T_{4,1}'$};
  \node at (bd) {$T_{4,2}'$};
  \node at (bf) {$T_{4,3}'$};
  \node at (bh) {$T_{4,4}'$};
  \node at (db) {$T_{3,1}'$};
  \node at (dd) {$T_{3,2}'$};
  \node at (df) {$T_{3,3}'$};
  \node at (dh) {$T_{3,4}'$};
  \node at (fb) {$T_{2,1}'$};
  \node at (fd) {$T_{2,2}'$};
  \node at (ff) {$T_{2,3}'$};
  \node at (fh) {$T_{2,4}'$};
  \node at (hb) {$T_{1,1}'$};
  \node at (hd) {$T_{1,2}'$};
  \node at (hf) {$T_{1,3}'$};
  \node at (hh) {$T_{1,4}'$};
  \node at (-2,7) {row:};
  \node at (-1.2,7) {1};
  \node at (-1.2,5) {2};
  \node at (-1.2,3) {3};
  \node at (-1.2,1) {4};
  \node at (ib) {$1$};
  \node at (id) {$2$};
  \node at (if) {$3$};
  \node at (ih) {$4$};
  \node at (ha) {$1$};
  \node at (hc) {$4$};
  \node at (he) {$4$};
  \node at (hg) {$4$};
  \node at (hi) {$+$};
  \node at (gb) {$4$};
  \node at (gd) {$2$};
  \node at (gf) {$3$};
  \node at (gh) {$+$};
  \node at (fa) {$4$};
  \node at (fc) {$2$};
  \node at (fe) {$3$};
  \node at (fg) {$+$};
  \node at (fi) {$+$};
  \node at (eb) {$2$};
  \node at (ed) {$3$};
  \node at (ef) {$+$};
  \node at (eh) {$+$};
  \node at (da) {$3$};
  \node at (dc) {$3$};
  \node at (de) {$+$};
  \node at (dg) {$+$};
  \node at (di) {$+$};
  \node at (cb) {$2$};
  \node at (cd) {$+$};
  \node at (cf) {$+$};
  \node at (ch) {$+$};
  \node at (ba) {$2$};
  \node at (bc) {$+$};
  \node at (be) {$+$};
  \node at (bg) {$+$};
  \node at (bi) {$+$};
  \node at (ab) {$+$};
  \node at (ad) {$+$};
  \node at (af) {$+$};
  \node at (ah) {$+$};
  \node at (7,9) {$4$};
  \node at (5,9) {$3$};
  \node at (3,9) {$2$};
  \node at (1,9) {$1$};
  \node at (0,9.04) {column:};
\end{tikzpicture}}
    \caption{From left to right: a nonreduced pipe dream with subtle reduction, its reduced pipe dream, and the corresponding lattice model state that encapsulates both pipe dreams, for $w = 1432$.}
    \label{reductionisweird}
\end{figure}
\end{example}

\begin{figure}[h]
\centering
\scalebox{.8}{
\begin{tikzpicture}  \coordinate (ab) at (1,0);
  \coordinate (ad) at (3,0);
  \coordinate (af) at (5,0);
  \coordinate (ba) at (0,1);
  \coordinate (bc) at (2,1);
  \coordinate (be) at (4,1);
  \coordinate (bg) at (6,1);
  \coordinate (cb) at (1,2);
  \coordinate (cd) at (3,2);
  \coordinate (cf) at (5,2);
  \coordinate (da) at (0,3);
  \coordinate (dc) at (2,3);
  \coordinate (de) at (4,3);
  \coordinate (dg) at (6,3);
  \coordinate (eb) at (1,4);
  \coordinate (ed) at (3,4);
  \coordinate (ef) at (5,4);
  \coordinate (fa) at (0,5);
  \coordinate (fc) at (2,5);
  \coordinate (fe) at (4,5);
  \coordinate (fg) at (6,5);
  \coordinate (gb) at (1,6);
  \coordinate (gd) at (3,6);
  \coordinate (gf) at (5,6);
  \coordinate (bb) at (1,1);
  \coordinate (bd) at (3,1);
  \coordinate (bf) at (5,1);
  \coordinate (db) at (1,3);
  \coordinate (dd) at (3,3);
  \coordinate (df) at (5,3);
  \coordinate (fb) at (1,5);
  \coordinate (fd) at (3,5);
  \coordinate (ff) at (5,5);
%grid
  \draw (ba)--(bg);
  \draw (da)--(dg);
  \draw (fa)--(fg);
  \draw (ab)--(gb);
  \draw (ad)--(gd);
  \draw (af)--(gf);
%paths
  \draw[line width=0.5mm,red] (fa)--(fb)--(gb);
  \draw[line width=0.5mm,blue] (da)--(db)--(fb) -- (ff) -- (gf);
  \draw[line width=0.5mm,green] (ba)--(bb)--(db) --(dd)--(gd);
%circles
  \draw[fill=white] (ab) circle (.25);
  \draw[fill=white] (ad) circle (.25);
  \draw[fill=white] (af) circle (.25);
  \draw[line width=0.5mm,green,fill=white] (ba) circle (.25);
  \draw[fill=white] (bc) circle (.25);
  \draw[fill=white] (be) circle (.25);
  \draw[fill=white] (bg) circle (.25);
  \draw[line width=0.5mm,green,fill=white] (cb) circle (.25);
  \draw[fill=white] (cd) circle (.25);
  \draw[fill=white] (cf) circle (.25);
  \draw[line width=0.5mm,blue,fill=white] (da) circle (.25);
  \draw[line width=0.5mm,green,fill=white] (dc) circle (.25);
  \draw[fill=white] (de) circle (.25);
  \draw[fill=white] (dg) circle (.25);
  \draw[line width=0.5mm,blue,fill=white] (eb) circle (.25);
  \draw[line width=0.5mm,green,fill=white] (ed) circle (.25);
  \draw[fill=white] (ef) circle (.25);
  \draw[line width=0.5mm,red,fill=white] (fa) circle (.25);
  \draw[line width=0.5mm,blue,fill=white] (fc) circle (.25);
  \draw[line width=0.5mm,blue,fill=white] (fe) circle (.25);
  \draw[fill=white] (fg) circle (.25);
  \draw[line width=0.5mm,red,fill=white] (gb) circle (.25);
  \draw[line width=0.5mm,green,fill=white] (gd) circle (.25);
  \draw[line width=0.5mm,blue,fill=white] (gf) circle (.25);
  \path[fill=white] (bb) circle (.4);
  \path[fill=white] (bd) circle (.4);
  \path[fill=white] (bf) circle (.4);
  \path[fill=white] (db) circle (.4);
  \path[fill=white] (dd) circle (.4);
  \path[fill=white] (df) circle (.4);
  \path[fill=white] (fb) circle (.4);
  \path[fill=white] (fd) circle (.4);
  \path[fill=white] (ff) circle (.4);
%labels   
   \node at (bb) {$T_{3,1}'$};
   \node at (bd) {$T_{3,2}'$};
   \node at (bf) {$T_{3,3}'$};
   \node at (db) {$T_{2,1}'$};
   \node at (dd) {$T_{2,2}'$};
   \node at (df) {$T_{2,3}'$};
   \node at (fb) {$T_{1,1}'$};
   \node at (fd) {$T_{1,2}'$};
   \node at (ff) {$T_{1,3}'$};
   \node at (gb) {$1$};
   \node at (gd) {$2$};
   \node at (gf) {$3$};
   \node at (fa) {$1$};
   \node at (fc) {$3$};
   \node at (fe) {$3$};
   \node at (fg) {$+$};
   \node at (eb) {$3$};
   \node at (ed) {$2$};
   \node at (ef) {$+$}; 
   \node at (da) {$3$};
   \node at (dc) {$2$};
   \node at (de) {$+$};
   \node at (dg) {$+$};
   \node at (cb) {$2$};
   \node at (cd) {$+$};
   \node at (cf) {$+$};   
   \node at (ba) {$2$};
   \node at (bc) {$+$};
   \node at (be) {$+$};
   \node at (bg) {$+$}; 
   \node at (ab) {$+$};
   \node at (ad) {$+$};
   \node at (af) {$+$};
   \node at (-1,3) {$\mathfrak{s}_1 = $};
   \draw [decorate,decoration={brace,amplitude=10pt},xshift=0pt,yshift=4pt] (1,-1) -- (5,-1);
 
%pipe dream
%grid
  \draw (-1,-4) -- (2,-4);
  \draw (-1,-3) -- (2,-3);
  \draw (-1,-2) -- (2,-2);
  \draw (-1,-1) -- (2,-1);
  \draw (-1,-4) -- (-1,-1);
  \draw (0,-4) -- (0,-1);
  \draw (1,-4) -- (1,-1);
  \draw (2,-4) -- (2,-1);
%pipes
  \draw[rounded corners = 3mm, line width = .5mm] (-1,-1.5) -- (-.5,-1.5) -- (-.5, -1);
  \draw[rounded corners= 3mm, line width = .5mm] (-1, -2.5) -- (.5,-2.5) -- (.5,-1.5) -- (1.5,-1.5) -- (1.5,-1);
  \draw[rounded corners = 3mm, line width = .5mm] (-1,-3.5) -- (-.5,-3.5) -- (-.5,-1.5)  -- (.5,-1.5) -- (.5,-1);
  \node at (-2, -2.5) {$P_1 = $};
%pipe dream
%grid
  \draw (4,-4) -- (7,-4);
  \draw (4,-3) -- (7,-3);
  \draw (4,-2) -- (7,-2);
  \draw (4,-1) -- (7,-1);
  \draw (4,-4) -- (4,-1);
  \draw (5,-4) -- (5,-1);
  \draw (6,-4) -- (6,-1);
  \draw (7,-4) -- (7,-1);
%pipes
  \draw[rounded corners = 3mm, line width = .5mm] (4,-1.5) -- (4.5,-1.5) -- (4.5, -1);
  \draw[rounded corners= 3mm, line width = .5mm] (4, -2.5) -- (5.5,-2.5) -- (5.5,-1);
  \draw[rounded corners = 3mm, line width = .5mm] (4,-3.5) -- (4.5,-3.5) -- (4.5,-1.5)  -- (6.5,-1.5) -- (6.5,-1);
  \node at (3, -2.5) {$P_2 = $};
\end{tikzpicture}}
\ 
\scalebox{.8}{
\begin{tikzpicture}
 \coordinate (ab) at (1,0);
  \coordinate (ad) at (3,0);
  \coordinate (af) at (5,0);
  \coordinate (ba) at (0,1);
  \coordinate (bc) at (2,1);
  \coordinate (be) at (4,1);
  \coordinate (bg) at (6,1);
  \coordinate (cb) at (1,2);
  \coordinate (cd) at (3,2);
  \coordinate (cf) at (5,2);
  \coordinate (da) at (0,3);
  \coordinate (dc) at (2,3);
  \coordinate (de) at (4,3);
  \coordinate (dg) at (6,3);
  \coordinate (eb) at (1,4);
  \coordinate (ed) at (3,4);
  \coordinate (ef) at (5,4);
  \coordinate (fa) at (0,5);
  \coordinate (fc) at (2,5);
  \coordinate (fe) at (4,5);
  \coordinate (fg) at (6,5);
  \coordinate (gb) at (1,6);
  \coordinate (gd) at (3,6);
  \coordinate (gf) at (5,6);
  \coordinate (bb) at (1,1);
  \coordinate (bd) at (3,1);
  \coordinate (bf) at (5,1);
  \coordinate (db) at (1,3);
  \coordinate (dd) at (3,3);
  \coordinate (df) at (5,3);
  \coordinate (fb) at (1,5);
  \coordinate (fd) at (3,5);
  \coordinate (ff) at (5,5);
 %grid
  \draw (ba)--(bg);
  \draw (da)--(dg);
  \draw (fa)--(fg);
  \draw (ab)--(gb);
  \draw (ad)--(gd);
  \draw (af)--(gf);
 %paths
  \draw[line width=0.5mm,red] (fa)--(fb)--(gb);
  \draw[line width=0.5mm,blue] (da)--(dd)--(fd) -- (ff) -- (gf);
  \draw[line width=0.5mm,green] (ba)--(bb)--(fb) --(fd)--(gd);
 %circles
  \draw[fill=white] (ab) circle (.25);
  \draw[fill=white] (ad) circle (.25);
  \draw[fill=white] (af) circle (.25);
  \draw[line width=0.5mm,green,fill=white] (ba) circle (.25);
  \draw[fill=white] (bc) circle (.25);
  \draw[fill=white] (be) circle (.25);
  \draw[fill=white] (bg) circle (.25);
  \draw[line width=0.5mm,green,fill=white] (cb) circle (.25);
  \draw[fill=white] (cd) circle (.25);
  \draw[fill=white] (cf) circle (.25);
  \draw[line width=0.5mm,blue,fill=white] (da) circle (.25);
  \draw[line width=0.5mm,blue,fill=white] (dc) circle (.25);
  \draw[fill=white] (de) circle (.25);
  \draw[fill=white] (dg) circle (.25);
  \draw[line width=0.5mm,green,fill=white] (eb) circle (.25);
  \draw[line width=0.5mm,blue,fill=white] (ed) circle (.25);
  \draw[fill=white] (ef) circle (.25);
  \draw[line width=0.5mm,red,fill=white] (fa) circle (.25);
  \draw[line width=0.5mm,green,fill=white] (fc) circle (.25);
  \draw[line width=0.5mm,blue,fill=white] (fe) circle (.25);
  \draw[fill=white] (fg) circle (.25);
  \draw[line width=0.5mm,red,fill=white] (gb) circle (.25);
  \draw[line width=0.5mm,green,fill=white] (gd) circle (.25);
  \draw[line width=0.5mm,blue,fill=white] (gf) circle (.25);
  \path[fill=white] (bb) circle (.4);
  \path[fill=white] (bd) circle (.4);
  \path[fill=white] (bf) circle (.4);
  \path[fill=white] (db) circle (.4);
  \path[fill=white] (dd) circle (.4);
  \path[fill=white] (df) circle (.4);
  \path[fill=white] (fb) circle (.4);
  \path[fill=white] (fd) circle (.4);
  \path[fill=white] (ff) circle (.4);
%labels
  \node at (bb) {$T_{3,1}'$};
  \node at (bd) {$T_{3,2}'$};
  \node at (bf) {$T_{3,3}'$};
  \node at (db) {$T_{2,1}'$};
  \node at (dd) {$T_{2,2}'$};
  \node at (df) {$T_{2,3}'$};
  \node at (fb) {$T_{1,1}'$};
  \node at (fd) {$T_{1,2}'$};
  \node at (ff) {$T_{1,3}'$};
  \node at (gb) {$1$};
  \node at (gd) {$2$};
  \node at (gf) {$3$};
  \node at (fa) {$1$};
  \node at (fc) {$2$};
  \node at (fe) {$3$};
  \node at (fg) {$+$};
  \node at (eb) {$2$};
  \node at (ed) {$3$};
  \node at (ef) {$+$}; 
  \node at (da) {$3$};
  \node at (dc) {$3$};
  \node at (de) {$+$};
  \node at (dg) {$+$};
  \node at (cb) {$2$};
  \node at (cd) {$+$};
  \node at (cf) {$+$};   
  \node at (ba) {$2$};
  \node at (bc) {$+$};
  \node at (be) {$+$};
  \node at (bg) {$+$}; 
  \node at (ab) {$+$};
  \node at (ad) {$+$};
  \node at (af) {$+$};
  \node at (-1,3) {$\mathfrak{s}_2 = $};
  \draw [decorate,decoration={brace,amplitude=10pt},xshift=0pt,yshift=4pt] (1,-1) -- (5,-1);
%pipe dream   
%grid
  \draw (1.5,-4) -- (4.5,-4);
  \draw (1.5,-3) -- (4.5,-3);
  \draw (1.5,-2) -- (4.5,-2);
  \draw (1.5,-1) -- (4.5,-1);
  \draw (1.5,-4) -- (1.5,-1);
  \draw (2.5,-4) -- (2.5,-1);
  \draw (3.5,-4) -- (3.5,-1);
  \draw (4.5,-4) -- (4.5,-1);
%pipes
  \draw[rounded corners = 3mm, line width = .5mm] (1.5,-1.5) -- (2,-1.5) -- (2, -1);
  \draw[rounded corners= 3mm, line width = .5mm] (1.5, -2.5) -- (2,-2.5) -- (2,-1.5) -- (4,-1.5) -- (4,-1);
  \draw[rounded corners = 3mm, line width = .5mm] (1.5,-3.5) -- (2,-3.5) -- (2,-2.5)  -- (3,-2.5) -- (3,-1);
  \node at (.5, -2.5) {$P_3 = $};
\end{tikzpicture}}
\caption{The correspondence between pipe dreams and admissible states of $\mathfrak{S}_{1,w^{-1}}(\boldsymbol{y},\boldsymbol{x})$ when $w=132$, where $T'_{i,j}$ denotes the labels $y_i,x_j$. Here, the reduced pipe dream $P_1$ corresponds to state $\mathfrak{s}_1$, and the reduced pipe dream $P_3$ corresponds to $\mathfrak{s}_2$. Notice that $B(\mathfrak{s}_1) = \text{wt}(P_1) + \text{wt}(P_2)$ and $B(\mathfrak{s}_2) = \text{wt}(P_3).$ Thus, the weight of the single nonreduced pipe dream $P_2$ is accounted for by the lattice state corresponding to its reduction.}
\label{pipedreamsmodel2}
\end{figure}
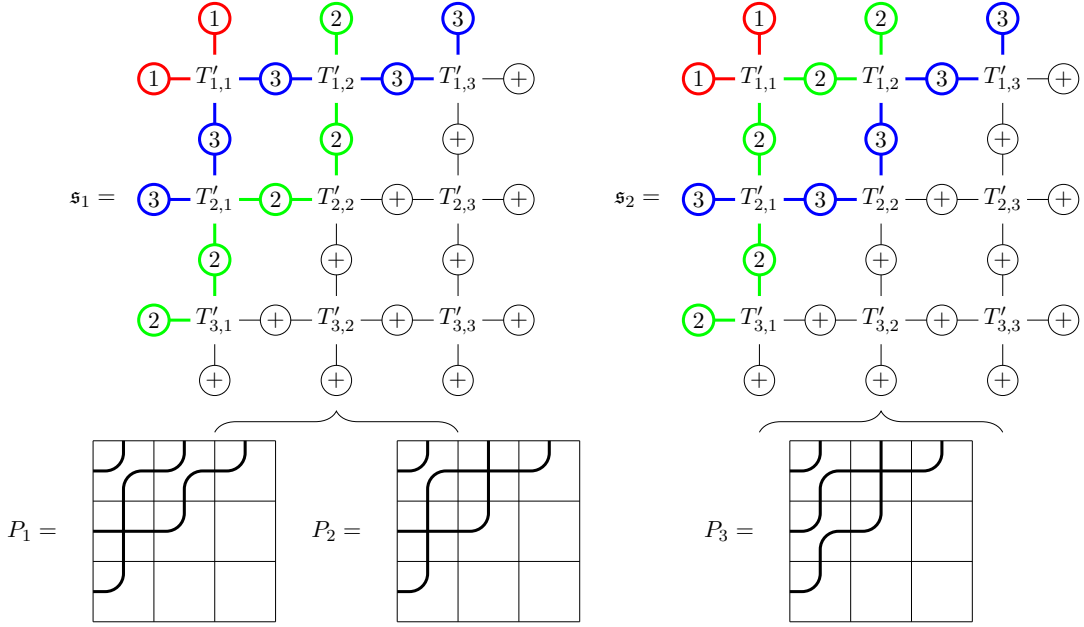

A similar bijection can be made between reduced pipe dreams for a permutation and admissible states of $\mathfrak{S}_{1,w}(\boldsymbol{x},\boldsymbol{y})$. This time, the edges of the system are labeled by $w(i)$ if the edge is along the strand of pipe $i$, and no reflection of the state. This correspondence is compatible with a different kind of reduction, where all but the northeastern-most crossing is removed. In this correspondence, the weight of an admissible state of $\mathfrak{S}_{1,w}(\boldsymbol{x},\boldsymbol{y})$ has weight equal to the sum of the weights of all pipe dreams whose reduction under this different method is the pipe dream corresponding to the admissible state. Figure \ref{pipedreamsmodel1} gives an example for the permutation $132.$

\begin{figure}[h]
\centering
\scalebox{.8}{
\begin{tikzpicture}
  \coordinate (ab) at (1,0);
  \coordinate (ad) at (3,0);
  \coordinate (af) at (5,0);
  \coordinate (ba) at (0,1);
  \coordinate (bc) at (2,1);
  \coordinate (be) at (4,1);
  \coordinate (bg) at (6,1);
  \coordinate (cb) at (1,2);
  \coordinate (cd) at (3,2);
  \coordinate (cf) at (5,2);
  \coordinate (da) at (0,3);
  \coordinate (dc) at (2,3);
  \coordinate (de) at (4,3);
  \coordinate (dg) at (6,3);
  \coordinate (eb) at (1,4);
  \coordinate (ed) at (3,4);
  \coordinate (ef) at (5,4);
  \coordinate (fa) at (0,5);
  \coordinate (fc) at (2,5);
  \coordinate (fe) at (4,5);
  \coordinate (fg) at (6,5);
  \coordinate (gb) at (1,6);
  \coordinate (gd) at (3,6);
  \coordinate (gf) at (5,6);
  \coordinate (bb) at (1,1);
  \coordinate (bd) at (3,1);
  \coordinate (bf) at (5,1);
  \coordinate (db) at (1,3);
  \coordinate (dd) at (3,3);
  \coordinate (df) at (5,3);
  \coordinate (fb) at (1,5);
  \coordinate (fd) at (3,5);
  \coordinate (ff) at (5,5);
 %grid
  \draw (ba)--(bg);
  \draw (da)--(dg);
  \draw (fa)--(fg);
  \draw (ab)--(gb);
  \draw (ad)--(gd);
  \draw (af)--(gf);
 %paths
  \draw[line width=0.5mm,red] (fa)--(fb)--(gb);
  \draw[line width=0.5mm,blue] (da)--(dd)--(fd) -- (ff) -- (gf);
  \draw[line width=0.5mm,green] (ba)--(bb)--(fb) --(fd)--(gd);
 %circles
  \draw[fill=white] (ab) circle (.25);
  \draw[fill=white] (ad) circle (.25);
  \draw[fill=white] (af) circle (.25);
  \draw[line width=0.5mm,green,fill=white] (ba) circle (.25);
  \draw[fill=white] (bc) circle (.25);
  \draw[fill=white] (be) circle (.25);
  \draw[fill=white] (bg) circle (.25);
  \draw[line width=0.5mm,green,fill=white] (cb) circle (.25);
  \draw[fill=white] (cd) circle (.25);
  \draw[fill=white] (cf) circle (.25);
  \draw[line width=0.5mm,blue,fill=white] (da) circle (.25);
  \draw[line width=0.5mm,blue,fill=white] (dc) circle (.25);
  \draw[fill=white] (de) circle (.25);
  \draw[fill=white] (dg) circle (.25);
  \draw[line width=0.5mm,green,fill=white] (eb) circle (.25);
  \draw[line width=0.5mm,blue,fill=white] (ed) circle (.25);
  \draw[fill=white] (ef) circle (.25);
  \draw[line width=0.5mm,red,fill=white] (fa) circle (.25);
  \draw[line width=0.5mm,green,fill=white] (fc) circle (.25);
  \draw[line width=0.5mm,blue,fill=white] (fe) circle (.25);
  \draw[fill=white] (fg) circle (.25);
  \draw[line width=0.5mm,red,fill=white] (gb) circle (.25);
  \draw[line width=0.5mm,green,fill=white] (gd) circle (.25);
  \draw[line width=0.5mm,blue,fill=white] (gf) circle (.25);
  \path[fill=white] (bb) circle (.4);
  \path[fill=white] (bd) circle (.4);
  \path[fill=white] (bf) circle (.4);
  \path[fill=white] (db) circle (.4);
  \path[fill=white] (dd) circle (.4);
  \path[fill=white] (df) circle (.4);
  \path[fill=white] (fb) circle (.4);
  \path[fill=white] (fd) circle (.4);
  \path[fill=white] (ff) circle (.4);
%labels
  \node at (bb) {$T_{3,1}$};
  \node at (bd) {$T_{3,2}$};
  \node at (bf) {$T_{3,3}$};
  \node at (db) {$T_{2,1}$};
  \node at (dd) {$T_{2,2}$};
  \node at (df) {$T_{2,3}$};
  \node at (fb) {$T_{1,1}$};
  \node at (fd) {$T_{1,2}$};
  \node at (ff) {$T_{1,3}$};
  \node at (gb) {$1$};
  \node at (gd) {$2$};
  \node at (gf) {$3$};
  \node at (fa) {$1$};
  \node at (fc) {$2$};
  \node at (fe) {$3$};
  \node at (fg) {$+$};
  \node at (eb) {$2$};
  \node at (ed) {$3$};
  \node at (ef) {$+$}; 
  \node at (da) {$3$};
  \node at (dc) {$3$};
  \node at (de) {$+$};
  \node at (dg) {$+$};
  \node at (cb) {$2$};
  \node at (cd) {$+$};
  \node at (cf) {$+$};   
  \node at (ba) {$2$};
  \node at (bc) {$+$};
  \node at (be) {$+$};
  \node at (bg) {$+$}; 
  \node at (ab) {$+$};
  \node at (ad) {$+$};
  \node at (af) {$+$};
  \node at (-1,3) {$\overline{\mathfrak{s}}_1 = $};
  \draw [decorate,decoration={brace,amplitude=10pt},xshift=0pt,yshift=4pt] (1,-1) -- (5,-1);
%pipe dream
%grid
  \draw (1.5,-4) -- (4.5,-4);
  \draw (1.5,-3) -- (4.5,-3);
  \draw (1.5,-2) -- (4.5,-2);
  \draw (1.5,-1) -- (4.5,-1);
  \draw (1.5,-4) -- (1.5,-1);
  \draw (2.5,-4) -- (2.5,-1);
  \draw (3.5,-4) -- (3.5,-1);
  \draw (4.5,-4) -- (4.5,-1);
%pipes
  \draw[rounded corners = 3mm, line width = .5mm] (1.5,-1.5) -- (2,-1.5) -- (2, -1);
  \draw[rounded corners= 3mm, line width = .5mm] (1.5, -2.5) -- (3,-2.5) -- (3,-1.5) -- (4,-1.5) -- (4,-1);
  \draw[rounded corners = 3mm, line width = .5mm] (1.5,-3.5) -- (2,-3.5) -- (2,-1.5)  -- (3,-1.5) -- (3,-1);
 \node at (.5, -2.5) {$P_1 = $};
\end{tikzpicture}}
 \
 \scalebox{.8}{
\begin{tikzpicture}
  \coordinate (ab) at (1,0);
  \coordinate (ad) at (3,0);
  \coordinate (af) at (5,0);
  \coordinate (ba) at (0,1);
  \coordinate (bc) at (2,1);
  \coordinate (be) at (4,1);
  \coordinate (bg) at (6,1);
  \coordinate (cb) at (1,2);
  \coordinate (cd) at (3,2);
  \coordinate (cf) at (5,2);
  \coordinate (da) at (0,3);
  \coordinate (dc) at (2,3);
  \coordinate (de) at (4,3);
  \coordinate (dg) at (6,3);
  \coordinate (eb) at (1,4);
  \coordinate (ed) at (3,4);
  \coordinate (ef) at (5,4);
  \coordinate (fa) at (0,5);
  \coordinate (fc) at (2,5);
  \coordinate (fe) at (4,5);
  \coordinate (fg) at (6,5);
  \coordinate (gb) at (1,6);
  \coordinate (gd) at (3,6);
  \coordinate (gf) at (5,6);
  \coordinate (bb) at (1,1);
  \coordinate (bd) at (3,1);
  \coordinate (bf) at (5,1);
  \coordinate (db) at (1,3);
  \coordinate (dd) at (3,3);
  \coordinate (df) at (5,3);
  \coordinate (fb) at (1,5);
  \coordinate (fd) at (3,5);
  \coordinate (ff) at (5,5);
%grid
  \draw (ba)--(bg);
  \draw (da)--(dg);
  \draw (fa)--(fg);
  \draw (ab)--(gb);
  \draw (ad)--(gd);
  \draw (af)--(gf);
%paths
  \draw[line width=0.5mm,red] (fa)--(fb)--(gb);
  \draw[line width=0.5mm,blue] (da)--(db)--(fb) -- (ff) -- (gf);
  \draw[line width=0.5mm,green] (ba)--(bb)--(db) --(dd)--(gd);
%circles
  \draw[fill=white] (ab) circle (.25);
  \draw[fill=white] (ad) circle (.25);
  \draw[fill=white] (af) circle (.25);
  \draw[line width=0.5mm,green,fill=white] (ba) circle (.25);
  \draw[fill=white] (bc) circle (.25);
  \draw[fill=white] (be) circle (.25);
  \draw[fill=white] (bg) circle (.25);
  \draw[line width=0.5mm,green,fill=white] (cb) circle (.25);
  \draw[fill=white] (cd) circle (.25);
  \draw[fill=white] (cf) circle (.25);
  \draw[line width=0.5mm,blue,fill=white] (da) circle (.25);
  \draw[line width=0.5mm,green,fill=white] (dc) circle (.25);
  \draw[fill=white] (de) circle (.25);
  \draw[fill=white] (dg) circle (.25);
  \draw[line width=0.5mm,blue,fill=white] (eb) circle (.25);
  \draw[line width=0.5mm,green,fill=white] (ed) circle (.25);
  \draw[fill=white] (ef) circle (.25);
  \draw[line width=0.5mm,red,fill=white] (fa) circle (.25);
  \draw[line width=0.5mm,blue,fill=white] (fc) circle (.25);
  \draw[line width=0.5mm,blue,fill=white] (fe) circle (.25);
  \draw[fill=white] (fg) circle (.25);
  \draw[line width=0.5mm,red,fill=white] (gb) circle (.25);
  \draw[line width=0.5mm,green,fill=white] (gd) circle (.25);
  \draw[line width=0.5mm,blue,fill=white] (gf) circle (.25);
  \path[fill=white] (bb) circle (.4);
  \path[fill=white] (bd) circle (.4);
  \path[fill=white] (bf) circle (.4);
  \path[fill=white] (db) circle (.4);
  \path[fill=white] (dd) circle (.4);
  \path[fill=white] (df) circle (.4);
  \path[fill=white] (fb) circle (.4);
  \path[fill=white] (fd) circle (.4);
  \path[fill=white] (ff) circle (.4);
%labels   
   \node at (bb) {$T_{3,1}$};
   \node at (bd) {$T_{3,2}$};
   \node at (bf) {$T_{3,3}$};
   \node at (db) {$T_{2,1}$};
   \node at (dd) {$T_{2,2}$};
   \node at (df) {$T_{2,3}$};
   \node at (fb) {$T_{1,1}$};
   \node at (fd) {$T_{1,2}$};
   \node at (ff) {$T_{1,3}$};
   \node at (gb) {$1$};
   \node at (gd) {$2$};
   \node at (gf) {$3$};
   \node at (fa) {$1$};
   \node at (fc) {$3$};
   \node at (fe) {$3$};
   \node at (fg) {$+$};
   \node at (eb) {$3$};
   \node at (ed) {$2$};
   \node at (ef) {$+$}; 
   \node at (da) {$3$};
   \node at (dc) {$2$};
   \node at (de) {$+$};
   \node at (dg) {$+$};
   \node at (cb) {$2$};
   \node at (cd) {$+$};
   \node at (cf) {$+$};   
   \node at (ba) {$2$};
   \node at (bc) {$+$};
   \node at (be) {$+$};
   \node at (bg) {$+$}; 
   \node at (ab) {$+$};
   \node at (ad) {$+$};
   \node at (af) {$+$};
   \node at (-1,3) {$\overline{\mathfrak{s}}_2 = $};
   \draw [decorate,decoration={brace,amplitude=10pt},xshift=0pt,yshift=4pt] (1,-1) -- (5,-1);
%pipe dreams
%grid
  \draw (-1,-4) -- (2,-4);
  \draw (-1,-3) -- (2,-3);
  \draw (-1,-2) -- (2,-2);
  \draw (-1,-1) -- (2,-1);
  \draw (-1,-4) -- (-1,-1);
  \draw (0,-4) -- (0,-1);
  \draw (1,-4) -- (1,-1);
  \draw (2,-4) -- (2,-1);
%pipes
  \draw[rounded corners = 3mm, line width = .5mm] (-1,-1.5) -- (-.5,-1.5) -- (-.5, -1);
  \draw[rounded corners= 3mm, line width = .5mm] (-1, -2.5) -- (.5,-2.5) -- (.5,-1);
  \draw[rounded corners = 3mm, line width = .5mm] (-1,-3.5) -- (-.5,-3.5) -- (-.5,-1.5)  -- (1.5,-1.5) -- (1.5,-1);
  \node at (-2, -2.5) {$P_2 = $};
%pipe dream  
%grid
  \draw (4,-4) -- (7,-4);
  \draw (4,-3) -- (7,-3);
  \draw (4,-2) -- (7,-2);
  \draw (4,-1) -- (7,-1);
  \draw (4,-4) -- (4,-1);
  \draw (5,-4) -- (5,-1);
  \draw (6,-4) -- (6,-1);
  \draw (7,-4) -- (7,-1);
%pipes
  \draw[rounded corners = 3mm, line width = .5mm] (4,-1.5) -- (4.5,-1.5) -- (4.5, -1);
  \draw[rounded corners= 3mm, line width = .5mm] (4, -2.5) -- (4.5,-2.5) -- (4.5,-1.5)--(6.5,-1.5)--(6.5,-1);
  \draw[rounded corners = 3mm, line width = .5mm] (4,-3.5) -- (4.5,-3.5) -- (4.5,-2.5)  -- (5.5,-2.5) -- (5.5,-1);
  \node at (3, -2.5) {$P_3 = $};
\end{tikzpicture}}
\caption{The correspondence for the states of $\mathfrak{S}_{1,w}(\boldsymbol{x},\boldsymbol{y})$ when $w = 132$.  This time, $B(\overline{\mathfrak{s}}_1) = \text{wt}(P_1)$ and $B(\overline{\mathfrak{s}}_2) = \text{wt}(P_2) + \text{wt}(P_3)$. In this case, the weight of the nonreduced pipe dream is accounted for by the state corresponding to removal of the southwest crossing, a reduction which is different from the convention defined in Section \ref{grotpoly}.}
\label{pipedreamsmodel1}
\end{figure}
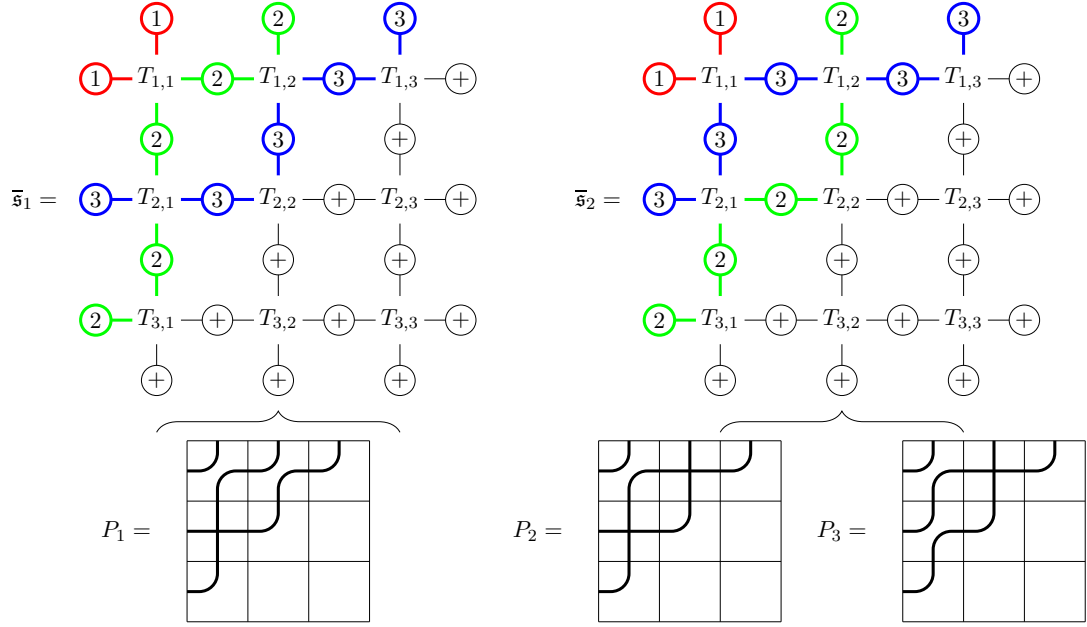

\section{Generalized Cauchy Identity}\label{cauchy}

Given an integer partition $\lambda$, let $s_\lambda$ denote the corresponding Schur function. The classical (dual) Cauchy identity for Schur polynomials states that \[\prod_{i,j} (1+x_iy_j) = \sum_\lambda s_\lambda(\boldsymbol{x})s_{\lambda'}(\boldsymbol{y}),\] where $\lambda'$ is the conjugate partition of $\lambda$.
It may be viewed as a result on tensor products of $\GL_n$ representations, a symmetric functions identity, or a by-product of the Boson-Fermion correspondence. The identity allows one to conclude that the Schur polynomials form a self-dual orthogonal basis of symmetric functions with respect to the Hall inner product. 

Generalized Cauchy formulas similarly allow one to obtain dual bases in $\lambda$-rings (see \cite{Lascoux-lambda-rings}). Such formulas have been obtained for many classes of polynomials, including Macdonald and LLT polynomials, factorial Schur functions, $k$-Schur functions, and various skew generalizations among them. See, for example,~\cite{Macdonald-schubert-polys} for a statement of the Cauchy formula for Schubert polynomials; this Cauchy identity describes structure in the cohomology ring of vector bundles on products of flag varieties \cite{Anderson}. Fomin and Kirillov proved a Cauchy identity for $\beta$-Grothendieck polynomials \cite{FominKirillov-gpybe} by proving relations between products in Yang-Baxter algebras.

In this section, we give a lattice model proof of Fomin and Kirillov's Cauchy identity that relies heavily on the solvability of the model. In fact, we'll prove a generalized Cauchy identity for the polynomials $\mathcal{G}_w^{(\beta, 0)}(\boldsymbol{x},\boldsymbol{y})$, and show that when $w=w_0$ this specializes to the earlier identity for $\beta$-Grothendieck polynomials. In \cite{BMN-factorial-schur}, Bump, McNamara, and Nakasuji gave a lattice model proof of the dual Cauchy identity for factorial Schur functions, and our proof is in much the same spirit. We finish the section by proving a Cauchy identity for the biaxial $\beta$-Grothendieck polynomials using similar methods. Unless otherwise specified, $q=0$ for this section, which applies to any vertex weights applied. Setting $q=0$ in weights $S$ and $S^*$ would make these 5-vertex models, which results in fewer admissible states. Furthermore, colored strands in any admissible state are confined to one side of a diagonal (see Figures \ref{pipedreamsmodel1} or \ref{pipedreamsmodel2} for reference). For ease of notation, we drop the superscript $q$ throughout this section, writing $\mathcal{G}_{v,w}^{(\beta)}(\boldsymbol{x},\boldsymbol{y}):= \mathcal{G}_{v,w}^{(\beta, 0)}(\boldsymbol{x},\boldsymbol{y})$.

\begin{theorem}[Generalized Cauchy Identity]\label{GenCauchy}
For any $w \in S_n$,
\[\mathcal{G}^{(\beta)}_w(\boldsymbol{x};\boldsymbol{y}) = \sum_{v\in S_n} \mathcal{G}^{(\beta)}_v(\boldsymbol{y}; \boldsymbol{z})\mathcal{G}^{(\beta)}_{v^{-1},w}(\boldsymbol{x}; \ominus \boldsymbol{z}).\]
\end{theorem}

We obtain this identity by introducing a new system of lattice models built from our earlier systems, then evaluating its partition function in two different ways. Given any permutation $w$, let $\mathfrak{R}_w(\boldsymbol{x};\boldsymbol{y};\boldsymbol{z})$ be the lattice model system with $2n$ rows and $n$ columns under the following set of conditions (see Figure~\ref{Cauchyrainbow}):
\begin{itemize}
\item Row parameters, from top to bottom, are $x_n,...,x_1,y_1,....,y_n$. (Because of repeated indices, we refer to rows as ``row $x_i$,'' etc.)
\item Column parameters are $z_1, \ldots, z_n$ from left to right.
\item  Boundary labeling conditions on the top, right, and bottom boundaries are all $+$.
\item For boundary conditions on the left boundary, the top $n$ rows are labeled by $w$ from bottom to top; the bottom $n$ rows are labeled with the identity permutation from top to bottom. 
\item The top $n$ rows take weights $S^*$ and the bottom $n$ rows take weights $S$ (see Figure \ref{cauchyweights}).
\end{itemize}

\begin{figure}[h]
\begin{center}
\scalebox{.8}{
\begin{tikzpicture}
%nodes
  \coordinate (ab) at (1,0);
  \coordinate (ad) at (3,0);
  \coordinate (af) at (5,0);
  \coordinate (ba) at (0,1);
  \coordinate (bc) at (2,1);
  \coordinate (be) at (4,1);
  \coordinate (bg) at (6,1);
  \coordinate (cb) at (1,2);
  \coordinate (cd) at (3,2);
  \coordinate (cf) at (5,2);
  \coordinate (da) at (0,3);
  \coordinate (dc) at (2,3);
  \coordinate (de) at (4,3);
  \coordinate (dg) at (6,3);
  \coordinate (eb) at (1,4);
  \coordinate (ed) at (3,4);
  \coordinate (ef) at (5,4);
  \coordinate (fa) at (0,5);
  \coordinate (fc) at (2,5);
  \coordinate (fe) at (4,5);
  \coordinate (fg) at (6,5);
  \coordinate (gb) at (1,6);
  \coordinate (gd) at (3,6);
  \coordinate (gf) at (5,6);
  \coordinate (ha) at (0,7);
  \coordinate (hc) at (2,7);
  \coordinate (he) at (4,7);
  \coordinate (hg) at (6,7);
  \coordinate (ib) at (1,8);
  \coordinate (id) at (3,8);
  \coordinate (if) at (5,8);
  \coordinate (ja) at (0,9);
  \coordinate (jc) at (2,9);
  \coordinate (je) at (4,9);
  \coordinate (jg) at (6,9);  
  \coordinate (kb) at (1,10);
  \coordinate (kd) at (3,10);
  \coordinate (kf) at (5,10);
  \coordinate (la) at (0,11);
  \coordinate (lc) at (2,11);
  \coordinate (le) at (4,11);
  \coordinate (lg) at (6,11);  
  \coordinate (mb) at (1,12);
  \coordinate (md) at (3,12);
  \coordinate (mf) at (5,12);
%grid
  \draw (ab)--(mb);
  \draw (ad)--(md);
  \draw (af)--(mf);
  \draw (ba)--(bg);
  \draw (da)--(dg);
  \draw (fa)--(fg);
  \draw (ha)--(hg);
  \draw (ja)--(jg);
  \draw (la)--(lg);
%circles
  \draw[fill=white] (ab) circle (.25);
  \draw[fill=white] (ad) circle (.25);
  \draw[fill=white] (af) circle (.25);
  \draw[line width=0.5mm,blue,fill=white] (ba) circle (.25);
  \draw[fill=white] (bc) circle (.25);
  \draw[fill=white] (be) circle (.25);
  \draw[fill=white] (bg) circle (.25);
  \draw[fill=white] (cb) circle (.25);
  \draw[fill=white] (cd) circle (.25);
  \draw[fill=white] (cf) circle (.25);
  \draw[line width=0.5mm,green,fill=white] (da) circle (.25);
  \draw[fill=white] (dc) circle (.25);
  \draw[fill=white] (de) circle (.25);
  \draw[fill=white] (dg) circle (.25);
  \draw[fill=white] (eb) circle (.25);
  \draw[fill=white] (ed) circle (.25);
  \draw[fill=white] (ef) circle (.25);
  \draw[line width=0.5mm,red,fill=white] (fa) circle (.25);
  \draw[fill=white] (fc) circle (.25);
  \draw[fill=white] (fe) circle (.25);
  \draw[fill=white] (fg) circle (.25);
  \draw[fill=white] (gb) circle (.25);
  \draw[fill=white] (gd) circle (.25);
  \draw[fill=white] (gf) circle (.25);
  \draw[fill=white] (ha) circle (.25);
  \draw[fill=white] (hc) circle (.25);
  \draw[fill=white] (he) circle (.25);
  \draw[fill=white] (hg) circle (.25);
  \draw[fill=white] (ib) circle (.25);
  \draw[fill=white] (id) circle (.25);
  \draw[fill=white] (if) circle (.25);
  \draw[fill=white] (ja) circle (.25);
  \draw[fill=white] (jc) circle (.25);
  \draw[fill=white] (je) circle (.25);
  \draw[fill=white] (jg) circle (.25);
  \draw[fill=white] (kb) circle (.25);
  \draw[fill=white] (kd) circle (.25);
  \draw[fill=white] (kf) circle (.25);
  \draw[fill=white] (la) circle (.25);
  \draw[fill=white] (lc) circle (.25);
  \draw[fill=white] (le) circle (.25);
  \draw[fill=white] (lg) circle (.25);
  \draw[fill=white] (mb) circle (.25);
  \draw[fill=white] (md) circle (.25);
  \draw[fill=white] (mf) circle (.25);
%spetral parameters
  \node at (-.5,11) {$x_3$};
  \node at (-.5,9) {$x_2$};
  \node at (-.5,7) {$x_1$};
  \node at (-.5,5) {$y_1$};
  \node at (-.5,3) {$y_2$};
  \node at (-.5,1) {$y_3$};
  \node at (1, 12.5) {$z_1$};
  \node at (3, 12.5) {$z_2$};
  \node at (5, 12.5) {$z_3$};
%bottom boundary
  \node at (ab) {$+$};
  \node at (ad) {$+$};
  \node at (af) {$+$};
%top bondary
  \node at (mb) {$+$};
  \node at (md) {$+$};
  \node at (mf) {$+$};
%left boundary
  \node at (ba) {$3$};
  \node at (da) {$2$};
  \node at (fa) {$1$};
  \node at (-1.5,9) {$w$};
%right boundary
  \node at (bg) {$+$};
  \node at (dg) {$+$};
  \node at (fg) {$+$};
  \node at (hg) {$+$};
  \node at (jg) {$+$};
  \node at (lg) {$+$};
%weight systems 
\draw [decorate,decoration={brace,amplitude=10pt,mirror,raise=4pt},yshift=0pt] (6.5,6) -- (6.5,12) node [black,midway,xshift=2.5cm] {weights $S^*$};
\draw [decorate,decoration={brace,amplitude=10pt,mirror,raise=4pt},yshift=0pt] (6.5,0) -- (6.5,6) node [black,midway,xshift=2.5cm] {weights $S$};
\draw [->,>=stealth] (-1.2,7)--(-1.2,11);
\draw[-, dotted, thick] (-2,6)--(8,6);
\end{tikzpicture}}
\end{center}
\caption{Boundary conditions for $\mathfrak{R}_w(\boldsymbol{x};\boldsymbol{y}; \boldsymbol{z})$ defined in Proposition \ref{rainbowsystem} when $n=3$.}\label{Cauchyrainbow}
\end{figure}

\begin{proposition}\label{rainbowsystem}
For $w\in S_n$, $Z(\mathfrak{R}_w(\boldsymbol{x};\boldsymbol{y};\boldsymbol{z})) = \mathcal{G}^{(\beta)}_w(\boldsymbol{x};\boldsymbol{y})$.
\end{proposition}

This proposition essentially follows from repeated use of the now familiar train argument. However, the base case $w=w_0$ in our recursive argument is surprisingly difficult. For this, we will need a pair of lemmas in increasing levels of generality. The proof of both lemmas makes use of a new statistic on colored edge labels called \emph{content}. To define it, we need an absolute numbering on rows and columns independent of the row and column parameter indices of each particular model. We number columns ascending from left to right and rows ascending from bottom to top. Note that each horizontal edge of the lattice occurs in a row $i$ and between columns $j$ and $j+1$ (where we can extend the numbering to include boundary edges), so we may assign coordinates to the horizontal edge of the form $(i, j+\frac{1}{2})$. Similarly, vertical edges lie between rows, so their coordinates are of the form $(i+\frac{1}{2}, j)$. For each colored edge in an admissible state, define the \emph{content} of that edge to be the integer given by the sum of its coordinates minus $\frac{1}{2}$. An example of the content of colored edges in a state of the top half of $\mathfrak{R}_w(\boldsymbol{x};\boldsymbol{y}; \boldsymbol{z})$ is given in Figure~\ref{contentdiagram}.

\begin{lemma}\label{nostupidcases} If we change the right boundary edge in the system $\mathfrak{R}_w(\boldsymbol{x};\boldsymbol{y}; \boldsymbol{z})$ by replacing a label $+$ in row $x_i$ with color $c$ and a label $+$ in row $y_j$ with color $c$ for any choice of $i,j \in [1,n]$, then the resulting system has no admissible states.
\end{lemma}

\begin{proof}
Consider only the top $n$ rows of this modified system, having Boltzmann weights $S^*$. We may view colored strands as entering this half from two sources: the bottom half (i.e., along the dotted line in Figure~\ref{Cauchyrainbow}) or from the newly colored right boundary edge with color $c$. Strands here travel up and to the left according to the admissible vertices. There is some ambiguity about how to interpret the movement of strands at an $\tt{a}$ vertex, since all adjacent edges are the same color; without loss of generality, we assume that the strand moving upward on the south edge travels left to the west edge at the vertex, while the strand on the east edge travels up to the north edge at the vertex. With these conventions, we may order the colored edges in a given strand, starting from the edge at which the strand enters the bottom and ending with the edge at which the strand exits out the left boundary. For each vertex that a given colored strand passes through, consider what happens to the content of the next edge in the strand as we move past the vertex. Traveling through a vertex either increases the content of the next colored label by 1 (in type $\tt{b_1}$ if $c$ is the smaller color); decreases the content of the next colored label by 1 (in type $\tt{b_1}$ if $c$ is the larger color); or leaves the content of the next colored label unchanged (noting our conventions on $\tt{a}$ vertices above).

Consider now the total content of all the colors as we move along all strands at once up and to the left. A vertex of type $\tt{b}_1$ with both $a<+$ and $b<+$ increases the content by one for the strand of smaller color, but decreases it by one for the strand of larger color. Thus the only vertex that changes the total content of all the strands is the type $\tt{b}_1$, with the larger color as $+$, which increases the total content by 1.
Therefore, the total content at the end of the strands along the left border must be greater than or equal to the total content at the beginning of the strands along the bottom. The total ending content of all colored edges on the left boundary of the top portion is $1+2+\ldots+n = \frac{n(n+1)}{2}$. To compute a lower bound on the total starting content along the bottom and right of the top half, note that colored strands move rightward and upward in the bottom half, according to the set of admissible vertices. The strand with color $c$ must then enter from the left boundary of the bottom half and exit at the lone edge labeled $c$ in the right boundary at row $y_j$. The remaining colored strands travel up to the top half of the lattice and their total content is at least $1+2+\cdots + n-1$. Finally we have an additional colored strand with color $c$ entering on the right boundary of the top half of the model, whose content is at least $n+1$. So the total starting content in the top half is at least $1+2+\ldots + n-1 + n+1 = \frac{n(n+1)}{2}+1$. This is more than the total ending content, giving a contradiction, so no admissible state is possible.
\end{proof}

\begin{figure}[h]
\begin{center}
\scalebox{.8}{
\begin{tikzpicture}
  \coordinate (ab) at (1,0);
  \coordinate (ad) at (3,0);
  \coordinate (af) at (5,0);
  \coordinate (ba) at (0,1);
  \coordinate (bc) at (2,1);
  \coordinate (be) at (4,1);
  \coordinate (bg) at (6,1);
  \coordinate (cb) at (1,2);
  \coordinate (cd) at (3,2);
  \coordinate (cf) at (5,2);
  \coordinate (da) at (0,3);
  \coordinate (dc) at (2,3);
  \coordinate (de) at (4,3);
  \coordinate (dg) at (6,3);
  \coordinate (eb) at (1,4);
  \coordinate (ed) at (3,4);
  \coordinate (ef) at (5,4);
  \coordinate (fa) at (0,5);
  \coordinate (fc) at (2,5);
  \coordinate (fe) at (4,5);
  \coordinate (fg) at (6,5);
  \coordinate (gb) at (1,6);
  \coordinate (gd) at (3,6);
  \coordinate (gf) at (5,6);
  \coordinate (bb) at (1,1);
  \coordinate (bd) at (3,1);
  \coordinate (bf) at (5,1);
  \coordinate (db) at (1,3);
  \coordinate (dd) at (3,3);
  \coordinate (df) at (5,3);
  \coordinate (fb) at (1,5);
  \coordinate (fd) at (3,5);
  \coordinate (ff) at (5,5);
%grid
  \draw (ba)--(bg);
  \draw (da)--(dg);
  \draw (fa)--(fg);
  \draw (ab)--(gb);
  \draw (ad)--(gd);
  \draw (af)--(gf);
%paths 
  \draw[line width=0.5mm,red] (fa)--(fb)--(ab);
  \draw[line width=0.5mm,green] (da)--(dd)--(bd) -- (bf) -- (af);
  \draw[line width=0.5mm,blue] (ba)--(bd)--(ad);
%circles
  \draw[line width=0.5mm,red,fill=white] (ab) circle (.25);
  \draw[line width=0.5mm,blue,fill=white] (ad) circle (.25);
  \draw[line width=0.5mm,green,fill=white] (af) circle (.25);
  \draw[line width=0.5mm,blue,fill=white] (ba) circle (.25);
  \draw[line width=0.5mm,blue,fill=white] (bc) circle (.25);
  \draw[line width=0.5mm,green,fill=white] (be) circle (.25);
  \draw[fill=white] (bg) circle (.25);
  \draw[line width=0.5mm,red,fill=white] (cb) circle (.25);
  \draw[line width=0.5mm,green,fill=white] (cd) circle (.25);
  \draw[fill=white] (cf) circle (.25);
  \draw[line width=0.5mm,green,fill=white] (da) circle (.25);
  \draw[line width=0.5mm,green,fill=white] (dc) circle (.25);
  \draw[fill=white] (de) circle (.25);
  \draw[fill=white] (dg) circle (.25);
  \draw[line width=0.5mm,red,fill=white] (eb) circle (.25);
  \draw[fill=white] (ed) circle (.25);
  \draw[fill=white] (ef) circle (.25);
  \draw[line width=0.5mm,red,fill=white] (fa) circle (.25);
  \draw[fill=white] (fc) circle (.25);
  \draw[fill=white] (fe) circle (.25);
  \draw[fill=white] (fg) circle (.25);
  \draw[fill=white] (gb) circle (.25);
  \draw[fill=white] (gd) circle (.25);
  \draw[fill=white] (gf) circle (.25);
  \path[fill=white] (bb) circle (.3);
  \path[fill=white] (bd) circle (.3);
  \path[fill=white] (bf) circle (.3);
  \path[fill=white] (db) circle (.3);
  \path[fill=white] (dd) circle (.3);
  \path[fill=white] (df) circle (.3);
  \path[fill=white] (fb) circle (.3);
  \path[fill=white] (fd) circle (.3);
  \path[fill=white] (ff) circle (.3);
%labels 
  \node at (bb) {\small $T_{1,1}$};
  \node at (bd) {\small $T_{1,2}$};
  \node at (bf) {\small $T_{1,3}$};
  \node at (db) {\small $T_{2,1}$};
  \node at (dd) {\small $T_{2,2}$};
  \node at (df) {\small $T_{2,3}$};
  \node at (fb) {\small $T_{3,1}$};
  \node at (fd) {\small $T_{3,2}$};
  \node at (ff) {\small $T_{3,3}$};
  \node at (gb) {$+$};
  \node at (gd) {$+$};
  \node at (gf) {$+$};
  \node at (fa) {$1$};
  \node at (fc) {$+$};
  \node at (fe) {$+$};
  \node at (fg) {$+$};
  \node at (eb) {$1$};
  \node at (ed) {$+$};
  \node at (ef) {$+$}; 
  \node at (da) {$2$};
  \node at (dc) {$2$};
  \node at (de) {$+$};
  \node at (dg) {$+$};
  \node at (cb) {$1$};
  \node at (cd) {$2$};
  \node at (cf) {$+$};   
  \node at (ba) {$3$};
  \node at (bc) {$3$};
  \node at (be) {$2$};
  \node at (bg) {$+$}; 
  \node at (ab) {$1$};
  \node at (ad) {$3$};
  \node at (af) {$2$};
  \node at (0,7) {column:};
  \node at (-2, 5) {row:};
  \node at (-1.2,5) {$3$};
  \node at (-1.2,3) {$2$};
  \node at (-1.2,1) {$1$};
  \node at (1,7) {$1$};
  \node at (3,7) {$2$};
  \node at (5,7) {$3$};
  \node at (0,5.5) {\small{3}};
  \node at (0,3.5) {\small{2}};
  \node at (0,1.5) {\small{1}};
  \node at (2, 1.5) {\small{2}};
  \node at (4, 1.5) {\small{3}}; 
  \node at (2,3.5) {\small{3}};
  \node at (.5,4) {\small{3}};
  \node at (.5,2) {\small{2}};
  \node at (.5,0) {\small{1}};
  \node at (2.5,2) {\small{3}};
  \node at (2.5,0) {\small{2}};
  \node at (4.5,0) {\small{3}};
\end{tikzpicture}}
\end{center}
\caption{The unique admissible state in the top half of $\mathfrak{R}_{w_0}(\boldsymbol{x};\boldsymbol{y}; \boldsymbol{z})$. The content (the sum of coordinates $- \frac{1}{2}$) is the number above or to the left of the color.}
\label{contentdiagram}
\end{figure}
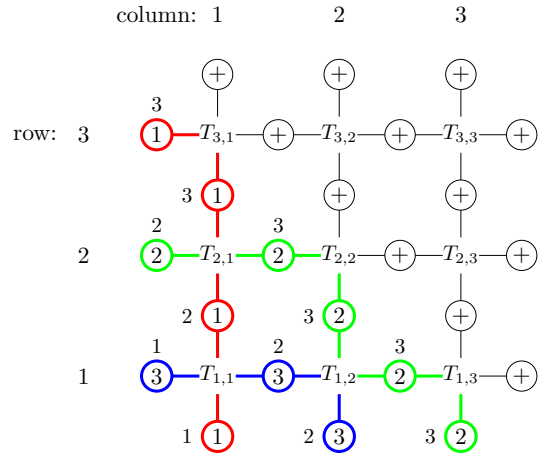

Since we will eventually want to swap all the $S$ rows past the $S^*$ rows in the system $\mathfrak{R}_w$, we need a similar result to hold for any system in which we have swapped some of these rows.

\begin{lemma} \label{nootherstupidcases}
Let $\mathfrak{R}'(\boldsymbol{x};\boldsymbol{y};\boldsymbol{z})$ be the system with $2n$ rows and $n$ columns:
\begin{itemize}
\item Row parameters, from top to bottom, are $x_n,...,x_1,y_1,....,y_n$.
\item Column parameters are $z_1, \ldots, z_n$ from left to right.
\item The top, right, and bottom boundary edges are all labeled $+$.
\item $n$ of the rows take weights $S^*$, while the other $n$ rows take weights $S$.
\item On the left boundary, the $n$ rows with weights $S^*$ are labeled with colors $1, \ldots, n$, in any order; the $n$ rows with weights $S$ are also labeled with colors $1, \ldots, n$, in any order.
\item For a given color $i$, the row with $S^*$ weights with label $i$ on its left boundary is higher than the row with weights  $S$ with label $i$ on its left boundary.
\end{itemize}
Let $\mathfrak{R}''(\boldsymbol{x};\boldsymbol{y};\boldsymbol{z})$ be the system $\mathfrak{R}'(\boldsymbol{x};\boldsymbol{y};\boldsymbol{z})$ with two $+$ labels on the right boundary replaced by a certain label $c$ between 1 and $n$ such that:
\begin{itemize}
\item Of the two rows, one takes weights $S^*$, while the other takes weights $S$.
\item The $S^*$ row is higher than the $S$ row.
\end{itemize}
Then $\mathfrak{R}''(\boldsymbol{x};\boldsymbol{y};\boldsymbol{z})$ has no admissible states.
\end{lemma}

\begin{proof}
This is a generalization of Lemma~\ref{nostupidcases}, and will again make use of the total content of all colored edges in the lattice (not including $+$). Consider first the system $\mathfrak{R}'(\boldsymbol{x};\boldsymbol{y};\boldsymbol{z})$. In any admissible state, our condition on the location of colors along the left boundary ensures that each colored strand begins on the left boundary at a $S$ row and travels upwards (moving left and right through the model) and exits on the left of a $S^*$ row. Let us consider the increase in total content from all colored strands as we move along the strand from entry to exit.

Let $r_S(i)$ (resp. $r_{S^*}(i))$ be the row in which color $i$ appears on the left boundary in $S$ (resp. $S^*$), counting from the top down. On one hand, the total content in $\mathfrak{R}'(\boldsymbol{x};\boldsymbol{y};\boldsymbol{z})$ must increase by precisely \[\sum_{i=1}^n (r_S(i)-r_{S^*}(i)) = n^2 - 2s,\] where $s$ is the number of pairs of a $S$ row and a $S^*$ row where the $S^*$ row is higher.

On the other hand, we will show that this content increase is the minimum that we could expect. Recall from Lemma \ref{nostupidcases} that no vertex taking weights $S^*$ reduces the total content. Now with weights $S$, vertices $\tt{b}_2, \tt{c_1}$, and $\tt{c_2}$ where the larger color is $+$ increase the total content by 2, while any vertex of these types having $a<+$ and $b<+$ increases the total content by 4, and a vertex of type $\tt{a}$ with $c<+$ also increases the total content by 4. On the other hand, vertices $\tt{c}_1$ and $\tt{c}_2$ move the color involved up half a row, and vertices $\tt{b}_2$ and $\tt{a}$ move a color up one row. In other words, the increase in content for every vertex is precisely twice the number of rows risen by strands through the vertex, except for a vertex of type $\tt{b}_2$ having larger color $+$; in the latter case, the content increases but there is no rise.

Since every strand begins in a $S$ row and moves upwards to a $S^*$ row, the strands altogether need to rise $n^2-2s$ rows, of which $\frac{n^2}{2}-s$ are $S$ rows (if a strand rises one row from a $S$ row to a $S^*$ row, say, we consider half that rise to take place in each row). Since the total content increase must be $n^2-2s$, we must have exactly the minimum content increases at each vertex: none in $S^*$ rows, and 2 per row crossed in $S$ rows.

Now we consider $\mathfrak{R}''(\boldsymbol{x};\boldsymbol{y};\boldsymbol{z})$. Note that the $c$ strand from the right boundary of a $S^*$ row travels leftward and upward through the top half of the lattice and exits on the left boundary. The $c$ strand from the right boundary of a $S$ row is the exit point of the $c$-colored strand moving upward and rightward from a lower spot on the left boundary. This means that the former strand ends on the left of a $S^*$ row, while the latter ends on the left of a $S$ row.

Consider now the total content of all the strands ending on the left side of $S^*$ rows. In other words, we consider only the higher of the two $c$ strands for each $c$. We will show that the total content must be higher in this scenario than in $\mathfrak{R}'(\boldsymbol{x};\boldsymbol{y};\boldsymbol{z})$; as in Lemma \ref{nostupidcases}, this will give a contradiction. Since the endpoints of all our (considered) strands are the same as in $\mathfrak{R}'(\boldsymbol{x};\boldsymbol{y};\boldsymbol{z})$, it suffices to consider the effect of the new starting point of the $c$ strand. Let $r'(c)$ be the row of $S^*$ that has $c$ on its right boundary. Then the content at the start of strand $c$ in $\mathfrak{R}''(\boldsymbol{x};\boldsymbol{y};\boldsymbol{z})$ is $C := n + r_S(c)-r'(c)$ higher than the content at the start of strand $c$ in $\mathfrak{R}'(\boldsymbol{x};\boldsymbol{y};\boldsymbol{z})$.

If $r'(c)<r_S(c)$, then $r'(c)$ is higher in the diagram than $r_S(c)$ (recall we label our rows from the top down), so the $c$ strand may cross as many as $\min\{r_S(c)-r'(c)-\frac{1}{2},n-\frac{1}{2}\}$ fewer $S$ rows (the $\frac{1}{2}$ arises because the $c$ strand starts in $S^*$ and ends in $S$). Since each strand crossing each $S$ row increases the content by at least 2, this can decrease the total content at the end of the strands by at most $D_1 := \min\{2r_S(c)-2r'(c)-1,2n-1\}$. However, \[C-D_1 = \max\{1 + (n - r_S(c)+r'(c)), 1-(n-r_S(c)+r'(c)\} \ge 1,\] so the total content at the end of the strands of $\mathfrak{R}''(\boldsymbol{x};\boldsymbol{y};\boldsymbol{z})$ must be higher than the total content at the end of the strands of $\mathfrak{R}'(\boldsymbol{x};\boldsymbol{y};\boldsymbol{z})$.

If $r'(c)>r_S(c)$, then $r'(c)$ is lower in the diagram than $r_S(c)$, and the $c$ strand must cross at least $\min\{r'(c)-r_S(c)-n+\frac{1}{2},\frac{1}{2}\}$ more $S$ rows, which will increase the total content at the end of the strands by at least $D_2 := \min(2r'(c)-2r_S(c)-2n+1,0)$. Therefore, \[C+D_2 = \min(1 + (r'(c)-r_S(c)-n), 1 - (r'(c)-r_S(c)-n)) \ge 1,\] which is again a contradiction, so $\mathfrak{R}''(\boldsymbol{x};\boldsymbol{y};\boldsymbol{z})$ has no admissible states.
\end{proof}

\begin{lemma} \label{rainbowlemma}
Proposition \ref{rainbowsystem} holds for $w = w_0$.
\end{lemma}

\begin{proof} We proceed by induction on $n$: for $n=1$, there is only one state for the system $\frak{R}_1(x_1;y_1;z_1)$, shown below, which has partition function 1:
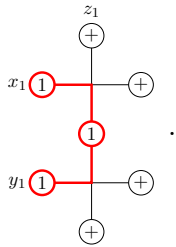
\[\begin{array}{c}
\scalebox{.65}{
\begin{tikzpicture}
%nodes
  \coordinate (ab) at (1,0);
  \coordinate (ba) at (0,1);
  \coordinate (bb) at (1,1);
  \coordinate (db) at (1,3);
  \coordinate (bc) at (2,1);
  \coordinate (cb) at (1,2);
  \coordinate (da) at (0,3);
  \coordinate (dc) at (2,3);
  \coordinate (eb) at (1,4);

%grid
  \draw (ab)--(eb);
  \draw (ba)--(bc);
  \draw (da)--(dc);

%paths
  \draw[line width=0.5mm,red] (ba)--(bb)--(db)--(da);

%circles
  \draw[fill=white] (ab) circle (.25);
  \draw[line width=0.5mm, red, fill=white] (ba) circle (.25);
  \draw[fill=white] (bc) circle (.25);
  \draw[line width=0.5mm, red, fill=white] (cb) circle (.25);
  \draw[line width=0.5mm, red, fill=white] (da) circle (.25);
  \draw[fill=white] (dc) circle (.25);
  \draw[fill=white] (eb) circle (.25);

%spectral parameters
  \node at (-.5,3) {$x_1$};
  \node at (-.5,1) {$y_1$};
  \node at (1,4.5) {$z_1$};
%bottom boundary
  \node at (ab) {$+$};

%top bondary
  \node at (eb) {$+$};

%left boundary
  \node at (ba) {$1$};
  \node at (da) {$1$};
 
  \node at (cb) {$1$};
	
%right boundary
  \node at (bc) {$+$};
  \node at (dc) {$+$};

\end{tikzpicture}} \end{array}.\]

Suppose that the claim holds for $w_0 \in S_{n-1}$. Consider the boundary conditions $\frak{R}_{w_0}(\boldsymbol{x};\boldsymbol{y};\boldsymbol{z})$ and attach a rhombus $R$-vertex to the right boundary in our now familiar train argument.

\begin{figure}[h]
\begin{center}
\scalebox{.75}{\begin{tikzpicture}
  \draw (11,0)--(17,0);
  \draw (11,2)--(17,2);
  \draw (12,-1)--(12,3);
  \draw (14,-1)--(14,3);
  \draw (16,-1)--(16,3);
  \coordinate (a1) at (9,0);
  \coordinate (c1) at (11,2);
  \coordinate (a2) at (9,2);
  \coordinate (c2) at (11,0);
  \draw[line width = .5mm, red] (a1) to [out=0,in=180] (c1);
  \draw[line width = .5mm, blue] (a2) to [out=0,in=180] (c2);
  \draw[line width=0.5mm,blue,fill=white] (9,2) circle (.3);
  \draw[line width=0.5mm,red,fill=white] (9,0) circle (.3);
  \draw[line width=0.5mm,red,fill=white] (11,2) circle (.3);
  \draw[line width=0.5mm,blue,fill=white] (11,0) circle (.3);
  \draw[fill=white] (17,2) circle (.3);
  \draw[fill=white] (17,0) circle (.3);
  \node at (17.7,0) {$x_1$};
  \node at (17.7,2) {$y_1$};
  \node at (9,0){$1$};
  \node at (9,2){$n$};
  \node at (11,0){$n$};
  \node at (11,2){$1$};
  \node at (17,0){$+$};
  \node at (17,2){$+$};
  \node[fill=white] at (14,2) {weights $S$};
  \path[fill=gray, opacity = 0.3] (11.5,1.5) rectangle (16.5, 2.5);
  \node[fill=white] at (14,0) {weights $S^*$};
  \path[fill=gray, opacity = 0.3] (11.5,0.5) rectangle (16.5, -0.5);
  
  \node at (7.5,1){\huge $=$};

  \draw (-2,0)--(4,0);
  \draw (-2,2)--(4,2);
  \draw (-1,-1)--(-1,3);
  \draw (1,-1)--(1,3);
  \draw (3,-1)--(3,3);
  \coordinate (p1) at (4,0);
  \coordinate (q1) at (6,2);
  \coordinate (p2) at (4,2);
  \coordinate (q2) at (6,0);
  \draw (p1) to [out=0,in=180] (q1);
  \draw (p2) to [out=0,in=180] (q2);
  \draw[fill=white] (6,2) circle (.3);
  \draw[fill=white] (6,0) circle (.3);
  \draw[fill=white] (4,2) circle (.3);
  \draw[fill=white] (4,0) circle (.3);
  \draw[line width=0.5mm,red,fill=white] (-2,0) circle (.3);
  \draw[line width=0.5mm,blue,fill=white] (-2,2) circle (.3);
  \node at (-2,0) {$1$};
  \node at (-2,2) {$n$};
  \node at (-2.7,0) {$y_1$};
  \node at (-2.7,2) {$x_1$};
  \node at (6,0){$+$};
  \node at (6,2){$+$};
  \node at (4,0){$+$};
  \node at (4,2){$+$};
  \node[fill=white] at (1,2) {weights $S^*$};
  \path[fill=gray, opacity = 0.3] (-1.5,1.5) rectangle (3.5, 2.5);
  \node[fill=white] at (1,0) {weights $S$};
  \path[fill=gray, opacity = 0.3] (-1.5,0.5) rectangle (3.5, -0.5);
  
  \end{tikzpicture}}
\end{center}
\caption{The first train argument in the proof of the Cauchy identity.}
\label{Cauchyops}
\end{figure}
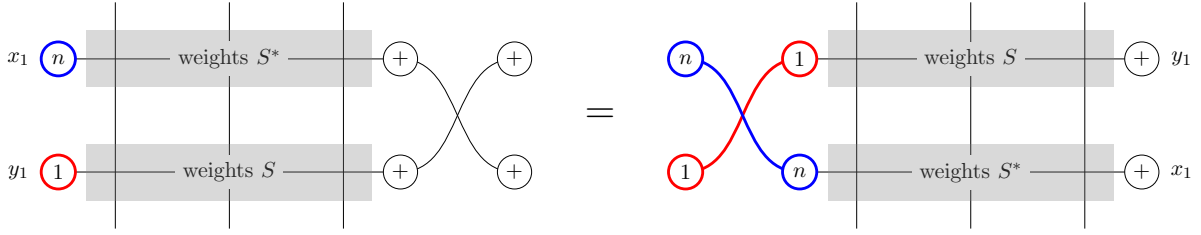
We apply Theorem \ref{SwapYBE} repeatedly to push the $R$-vertex to the left boundary, where it emerges with external edges assigned label $1$ (row $y_1$) and $n$ (row $x_1$) as in the right side of Figure \ref{Cauchyops}. Referring back to the $R$-vertices in Figure \ref{RweightsCauchy}, we then evaluate the partition functions of both sides: on the left hand side, the $R$-vertex could be of type $\tt{a}$ or type $\tt{c}_2$. Lemma \ref{nootherstupidcases} rules out type $\tt{c}_2$, so this $R$-vertex must be of type $\tt{a}$. Since the remaining boundary conditions to the left of the $R$-vertex mimic exactly those of $\frak{R}_{w_0}(\boldsymbol{x};\boldsymbol{y};\boldsymbol{z})$, this partition function is still $Z(\frak{R}_{w_0}(\boldsymbol{x};\boldsymbol{y};\boldsymbol{z}))$.

On the right hand side, there is only one $R$-vertex possible -- that of type $\tt{b}_1$ -- which has weight $x_1\oplus y_1$, and the boundary conditions on the left of that vertex have swapped one $S$ with one $S^*$. Thus we obtain the relation depicted in Figure~\ref{Cauchyops}:
\[ Z(\frak{R}_{w_0}(\boldsymbol{x};\boldsymbol{y};\boldsymbol{z})) = (x_1\oplus y_1) \cdot Z(\frak{R}_{w_0}(\boldsymbol{x};\boldsymbol{y};\boldsymbol{z}) \text{ with rows $x_1$ and $y_1$ swapped}).\]

We may continue to attach $R$-vertices and apply the train argument to move row $x_1$ downward, pushing it down past almost every $S$ row until it sits as the second-to-last row. By the same reasoning as above,
\[ Z(\frak{R}_{w_0}(\boldsymbol{x};\boldsymbol{y};\boldsymbol{z})) = \prod_{i<n}(x_1\oplus y_i) \cdot Z(\mathfrak{R}_{w_0}(\boldsymbol{x};\boldsymbol{y};\boldsymbol{z}) \text{ with row $x_1$ swapped with rows $y_1,y_2,\ldots,y_{n-1}$}).\]
If we examine the remaining boundary conditions, we see that on the left of the diagram, our boundary labels read $1 \; 2 \;\cdots\; n\!-\!1\; 1\; 2\; \cdots\; n\!-\!1\; n\; n$ from top to bottom, with parameters $x_n, x_{n-1},...,x_2,y_1,y_2,...,y_{n-1},x_1,y_n$. That is, if we chop off the bottom two rows, we have the boundary conditions for $\frak{R}_{w_0}(x_2,...,x_n; y_1,...,y_{n-1};\boldsymbol{z})$ in the $n-1$ case. Furthermore, the bottom two rows must have weight 1, since the $n$-colored strand cannot travel north or further east in the penultimate row (which has weights from $S^*$); it must turn immediately south and exit west out the last row in order to give an admissible state.

By induction, we then have that
\[Z(\frak{R}_{w_0}(\boldsymbol{x};\boldsymbol{y};\boldsymbol{z})) = \prod_{i<n}(x_1\oplus y_i) \prod_{k+j\leq n-1}(x_{j+1} \oplus y_k).\]
Therefore,
\[Z(\frak{R}_{w_0}(\boldsymbol{x};\boldsymbol{y};\boldsymbol{z})) = \prod_{i+j\leq n} (x_i\oplus y_j) = \mathcal{G}^{(\beta)}_{w_0}(\boldsymbol{x};\boldsymbol{y}). \qedhere \]
\end{proof}

\begin{example}
Let $n=3$. Figure \ref{cauchysteps} shows the entire step-by-step process from Lemma \ref{rainbowlemma}, including all the steps within the induction. In each step, we push an extra vertex attached from the right side of the system to the left. First, we attach a vertex between rows $y_1$ and $x_1$ and move through to the left via repeated application of the Yang-Baxter equation in Theorem \ref{RweightsCauchy}. This has the effect on the main rectangular part of the diagram of swapping rows $x_1$ and $y_1$ and the labels 1 and 3. We do the same to rows $x_2$ and $y_1$ followed by $x_1$ and $y_2$.
\end{example}

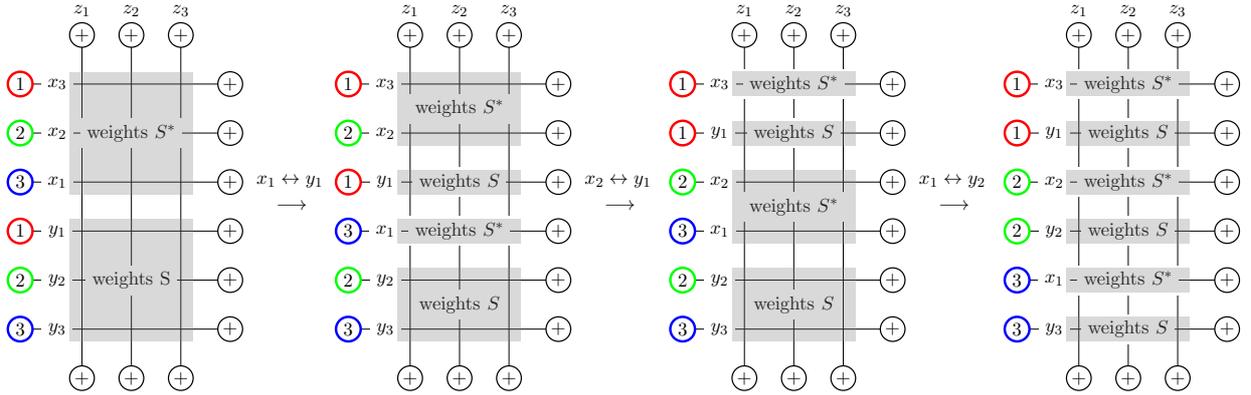
\begin{figure}[h]
\[\begin{array}{c@{\hspace{2pt}}c@{\hspace{2pt}}c@{\hspace{2pt}}c}
\scalebox{.65}{
\begin{tikzpicture}
%lines
\draw (-.5,.5)--(3.5,.5);
\draw (-.5,1.5)--(3.5,1.5);
\draw (-.5,2.5)--(3.5,2.5);
\draw (-.5,3.5)--(3.5,3.5);
\draw (-.5,4.5)--(3.5,4.5);
\draw (-.5,5.5)--(3.5,5.5);
\draw (.5,-.5)--(.5,6.5);
\draw (1.5,-.5)--(1.5,6.5);
\draw (2.5,-.5)--(2.5,6.5);
%top and bottom boundary
\draw[line width = 0.25mm,fill=white] (.5,-.5) circle (.25);
\draw[line width = 0.25mm,fill=white] (1.5,-.5) circle (.25);
\draw[line width = 0.25mm,fill=white] (2.5,-.5) circle (.25);
\draw[line width = 0.25mm,fill=white] (.5,6.5) circle (.25);
\draw[line width = 0.25mm,fill=white] (1.5,6.5) circle (.25);
\draw[line width = 0.25mm,fill=white] (2.5,6.5) circle (.25);
\node at (.5,-.5) {$+$};
\node at (1.5,-.5) {$+$};
\node at (2.5,-.5) {$+$};
\node at (.5,6.5) {$+$};
\node at (1.5,6.5) {$+$};
\node at (2.5,6.5) {$+$};
\node at (.5,7) {$z_1$};
\node at (1.5,7) {$z_2$};
\node at (2.5,7) {$z_3$};
%left 
\draw[line width=0.5mm,blue,fill=white] (-.75,.5) circle (.25);
\draw[line width=0.5mm,green,fill=white] (-.75,1.5) circle (.25);
\draw[line width=0.5mm,red,fill=white] (-.75,2.5) circle (.25);
\draw[line width=0.5mm,blue,fill=white] (-.75,3.5) circle (.25);
\draw[line width=0.5mm,green,fill=white] (-.75,4.5) circle (.25);
\draw[line width=0.5mm,red,fill=white] (-.75,5.5) circle (.25);
\node at (-.75,.5) {$3$};
\node at (-.75,1.5) {$2$};
\node at (-.75,2.5) {$1$};
\node at (-.75,3.5) {$3$};
\node at (-.75,4.5) {$2$};
\node at (-.75,5.5) {$1$};
\node[fill=white] at (0,.5) {$y_3$};
\node[fill=white] at (0,1.5) {$y_2$};
\node[fill=white] at (0,2.5) {$y_1$};
\node[fill=white] at (0,3.5) {$x_1$};
\node[fill=white] at (0,4.5) {$x_2$};
\node[fill=white] at (0,5.5) {$x_3$};
%right
\draw[line width = 0.25mm,fill=white] (3.5,.5) circle (.25);
\draw[line width = 0.25mm,fill=white] (3.5,1.5) circle (.25);
\draw[line width = 0.25mm,fill=white] (3.5,2.5) circle (.25);
\draw[line width = 0.25mm,fill=white] (3.5,3.5) circle (.25);
\draw[line width = 0.25mm,fill=white] (3.5,4.5) circle (.25);
\draw[line width = 0.25mm,fill=white] (3.5,5.5) circle (.25);
\node at (3.5,.5) {$+$};
\node at (3.5,1.5) {$+$};
\node at (3.5,2.5) {$+$};
\node at (3.5,3.5) {$+$};
\node at (3.5,4.5) {$+$};
\node at (3.5,5.5) {$+$};
\node[right] at (3.9,3.5) {$x_1 \leftrightarrow y_1$};
\node[right] at (4.3,3) {$\to$};
%zones
\node[fill = white] at (1.5,1.5) {weights S};
\path[fill = gray, opacity = .3] (0.25,0.25) rectangle (2.75, 2.75);
\node[fill = white] at (1.5,4.5) {weights $S^*$};
\path[fill = gray, opacity = .3] (0.25,3.25) rectangle (2.75, 5.75);
\end{tikzpicture}}&
\scalebox{.65}{\begin{tikzpicture}
%lines
\draw (-.5,.5)--(3.5,.5);
\draw (-.5,1.5)--(3.5,1.5);
\draw (-.5,2.5)--(3.5,2.5);
\draw (-.5,3.5)--(3.5,3.5);
\draw (-.5,4.5)--(3.5,4.5);
\draw (-.5,5.5)--(3.5,5.5);
\draw (.5,-.5)--(.5,6.5);
\draw (1.5,-.5)--(1.5,6.5);
\draw (2.5,-.5)--(2.5,6.5);
%top and bottom
\draw[line width = 0.25mm,fill=white] (.5,-.5) circle (.25);
\draw[line width = 0.25mm,fill=white] (1.5,-.5) circle (.25);
\draw[line width = 0.25mm,fill=white] (2.5,-.5) circle (.25);
\draw[line width = 0.25mm,fill=white] (.5,6.5) circle (.25);
\draw[line width = 0.25mm,fill=white] (1.5,6.5) circle (.25);
\draw[line width = 0.25mm,fill=white] (2.5,6.5) circle (.25);
\node at (.5,-.5) {$+$};
\node at (1.5,-.5) {$+$};
\node at (2.5,-.5) {$+$};
\node at (.5,6.5) {$+$};
\node at (1.5,6.5) {$+$};
\node at (2.5,6.5) {$+$};
\node at (.5,7) {$z_1$};
\node at (1.5,7) {$z_2$};
\node at (2.5,7) {$z_3$};
%%row labels
\draw[line width=0.5mm,blue,fill=white] (-.75,.5) circle (.25);
\draw[line width=0.5mm,green,fill=white] (-.75,1.5) circle (.25);
\draw[line width=0.5mm,blue,fill=white] (-.75,2.5) circle (.25);
\draw[line width=0.5mm,red,fill=white] (-.75,3.5) circle (.25);
\draw[line width=0.5mm,green,fill=white] (-.75,4.5) circle (.25);
\draw[line width=0.5mm,red,fill=white] (-.75,5.5) circle (.25);
\node at (-.75,.5) {$3$};
\node at (-.75,1.5) {$2$};
\node at (-.75,2.5) {$3$};
\node at (-.75,3.5) {$1$};
\node at (-.75,4.5) {$2$};
\node at (-.75,5.5) {$1$};
\node[fill=white] at (0,.5) {$y_3$};
\node[fill=white] at (0,1.5) {$y_2$};
\node[fill=white] at (0,2.5) {$x_1$};
\node[fill=white] at (0,3.5) {$y_1$};
\node[fill=white] at (0,4.5) {$x_2$};
\node[fill=white] at (0,5.5) {$x_3$};
%right
\draw[line width = 0.25mm,fill=white] (3.5,.5) circle (.25);
\draw[line width = 0.25mm,fill=white] (3.5,1.5) circle (.25);
\draw[line width = 0.25mm,fill=white] (3.5,2.5) circle (.25);
\draw[line width = 0.25mm,fill=white] (3.5,3.5) circle (.25);
\draw[line width = 0.25mm,fill=white] (3.5,4.5) circle (.25);
\draw[line width = 0.25mm,fill=white] (3.5,5.5) circle (.25);
\node at (3.5,.5) {$+$};
\node at (3.5,1.5) {$+$};
\node at (3.5,2.5) {$+$};
\node at (3.5,3.5) {$+$};
\node at (3.5,4.5) {$+$};
\node at (3.5,5.5) {$+$};
\node[right] at (3.9,3.5) {$x_2 \leftrightarrow y_1$};
\node[right] at (4.3,3) {$\to$};
%zones
\node[fill = white] at (1.5,1) {weights $S$};
\path[fill = gray, opacity = .3] (0.25,0.25) rectangle (2.75, 1.75);
\node[fill = white] at (1.5,2.5) {weights $S^*$};
\path[fill = gray, opacity = .3] (0.25,2.25) rectangle (2.75, 2.75);
\node[fill = white] at (1.5,3.5) {weights $S$};
\path[fill = gray, opacity = .3] (0.25,3.25) rectangle (2.75, 3.75);
\node[fill = white] at (1.5,5) {weights $S^*$};
\path[fill = gray, opacity = .3] (0.25,4.25) rectangle (2.75, 5.75);
\end{tikzpicture}}&
\scalebox{.65}{
\begin{tikzpicture}
%lines
\draw (-.5,.5)--(3.5,.5);
\draw (-.5,1.5)--(3.5,1.5);
\draw (-.5,2.5)--(3.5,2.5);
\draw (-.5,3.5)--(3.5,3.5);
\draw (-.5,4.5)--(3.5,4.5);
\draw (-.5,5.5)--(3.5,5.5);
\draw (.5,-.5)--(.5,6.5);
\draw (1.5,-.5)--(1.5,6.5);
\draw (2.5,-.5)--(2.5,6.5);
%top and bottom
\draw[line width = 0.25mm,fill=white] (.5,-.5) circle (.25);
\draw[line width = 0.25mm,fill=white] (1.5,-.5) circle (.25);
\draw[line width = 0.25mm,fill=white] (2.5,-.5) circle (.25);
\draw[line width = 0.25mm,fill=white] (.5,6.5) circle (.25);
\draw[line width = 0.25mm,fill=white] (1.5,6.5) circle (.25);
\draw[line width = 0.25mm,fill=white] (2.5,6.5) circle (.25);
\node at (.5,-.5) {$+$};
\node at (1.5,-.5) {$+$};
\node at (2.5,-.5) {$+$};
\node at (.5,6.5) {$+$};
\node at (1.5,6.5) {$+$};
\node at (2.5,6.5) {$+$};
\node at (.5,7) {$z_1$};
\node at (1.5,7) {$z_2$};
\node at (2.5,7) {$z_3$};
%row labels
\draw[line width=0.5mm,blue,fill=white] (-.75,.5) circle (.25);
\draw[line width=0.5mm,green,fill=white] (-.75,1.5) circle (.25);
\draw[line width=0.5mm,blue,fill=white] (-.75,2.5) circle (.25);
\draw[line width=0.5mm,green,fill=white] (-.75,3.5) circle (.25);
\draw[line width=0.5mm,red,fill=white] (-.75,4.5) circle (.25);
\draw[line width=0.5mm,red,fill=white] (-.75,5.5) circle (.25);
\node at (-.75,.5) {$3$};
\node at (-.75,1.5) {$2$};
\node at (-.75,2.5) {$3$};
\node at (-.75,3.5) {$2$};
\node at (-.75,4.5) {$1$};
\node at (-.75,5.5) {$1$};
\node[fill=white] at (0,.5) {$y_3$};
\node[fill=white] at (0,1.5) {$y_2$};
\node[fill=white] at (0,2.5) {$x_1$};
\node[fill=white] at (0,3.5) {$x_2$};
\node[fill=white] at (0,4.5) {$y_1$};
\node[fill=white] at (0,5.5) {$x_3$};
%right
\draw[line width = 0.25mm,fill=white] (3.5,.5) circle (.25);
\draw[line width = 0.25mm,fill=white] (3.5,1.5) circle (.25);
\draw[line width = 0.25mm,fill=white] (3.5,2.5) circle (.25);
\draw[line width = 0.25mm,fill=white] (3.5,3.5) circle (.25);
\draw[line width = 0.25mm,fill=white] (3.5,4.5) circle (.25);
\draw[line width = 0.25mm,fill=white] (3.5,5.5) circle (.25);
\node at (3.5,.5) {$+$};
\node at (3.5,1.5) {$+$};
\node at (3.5,2.5) {$+$};
\node at (3.5,3.5) {$+$};
\node at (3.5,4.5) {$+$};
\node at (3.5,5.5) {$+$};
\node[right] at (3.9,3.5) {$x_1 \leftrightarrow y_2$};
\node[right] at (4.3,3) {$\to$};
%zones
\node[fill = white] at (1.5,1) {weights $S$};
\path[fill = gray, opacity = .3] (0.25,0.25) rectangle (2.75, 1.75);
\node[fill = white] at (1.5,3) {weights $S^*$};
\path[fill = gray, opacity = .3] (0.25,2.25) rectangle (2.75, 3.75);
\node[fill = white] at (1.5,4.5) {weights $S$};
\path[fill = gray, opacity = .3] (0.25,4.25) rectangle (2.75, 4.75);
\node[fill = white] at (1.5,5.5) {weights $S^*$};
\path[fill = gray, opacity = .3] (0.25,5.25) rectangle (2.75, 5.75);
\end{tikzpicture}}&
\scalebox{.65}{
\begin{tikzpicture}
%lines
\draw (-.5,.5)--(3.5,.5);
\draw (-.5,1.5)--(3.5,1.5);
\draw (-.5,2.5)--(3.5,2.5);
\draw (-.5,3.5)--(3.5,3.5);
\draw (-.5,4.5)--(3.5,4.5);
\draw (-.5,5.5)--(3.5,5.5);
\draw (.5,-.5)--(.5,6.5);
\draw (1.5,-.5)--(1.5,6.5);
\draw (2.5,-.5)--(2.5,6.5);
%top and bottom
\draw[line width = 0.25mm,fill=white] (.5,-.5) circle (.25);
\draw[line width = 0.25mm,fill=white] (1.5,-.5) circle (.25);
\draw[line width = 0.25mm,fill=white] (2.5,-.5) circle (.25);
\draw[line width = 0.25mm,fill=white] (.5,6.5) circle (.25);
\draw[line width = 0.25mm,fill=white] (1.5,6.5) circle (.25);
\draw[line width = 0.25mm,fill=white] (2.5,6.5) circle (.25);
\node at (.5,-.5) {$+$};
\node at (1.5,-.5) {$+$};
\node at (2.5,-.5) {$+$};
\node at (.5,6.5) {$+$};
\node at (1.5,6.5) {$+$};
\node at (2.5,6.5) {$+$};
\node at (.5,7) {$z_1$};
\node at (1.5,7) {$z_2$};
\node at (2.5,7) {$z_3$};
%row labels
\draw[line width=0.5mm,blue,fill=white] (-.75,.5) circle (.25);
\draw[line width=0.5mm,blue,fill=white] (-.75,1.5) circle (.25);
\draw[line width=0.5mm,green,fill=white] (-.75,2.5) circle (.25);
\draw[line width=0.5mm,green,fill=white] (-.75,3.5) circle (.25);
\draw[line width=0.5mm,red,fill=white] (-.75,4.5) circle (.25);
\draw[line width=0.5mm,red,fill=white] (-.75,5.5) circle (.25);
\node at (-.75,.5) {$3$};
\node at (-.75,1.5) {$3$};
\node at (-.75,2.5) {$2$};
\node at (-.75,3.5) {$2$};
\node at (-.75,4.5) {$1$};
\node at (-.75,5.5) {$1$};
\node[fill=white] at (0,.5) {$y_3$};
\node[fill=white] at (0,1.5) {$x_1$};
\node[fill=white] at (0,2.5) {$y_2$};
\node[fill=white] at (0,3.5) {$x_2$};
\node[fill=white] at (0,4.5) {$y_1$};
\node[fill=white] at (0,5.5) {$x_3$};
%right
\draw[line width = 0.25mm,fill=white] (3.5,.5) circle (.25);
\draw[line width = 0.25mm,fill=white] (3.5,1.5) circle (.25);
\draw[line width = 0.25mm,fill=white] (3.5,2.5) circle (.25);
\draw[line width = 0.25mm,fill=white] (3.5,3.5) circle (.25);
\draw[line width = 0.25mm,fill=white] (3.5,4.5) circle (.25);
\draw[line width = 0.25mm,fill=white] (3.5,5.5) circle (.25);
\node at (3.5,.5) {$+$};
\node at (3.5,1.5) {$+$};
\node at (3.5,2.5) {$+$};
\node at (3.5,3.5) {$+$};
\node at (3.5,4.5) {$+$};
\node at (3.5,5.5) {$+$};
%zones
\node[fill = white] at (1.5,.5) {weights $S$};
\path[fill = gray, opacity = .3] (0.25,0.25) rectangle (2.75, .75);
\node[fill = white] at (1.5,1.5) {weights $S^*$};
\path[fill = gray, opacity = .3] (0.25,1.25) rectangle (2.75, 1.75);
\node[fill = white] at (1.5,2.5) {weights $S$};
\path[fill = gray, opacity = .3] (0.25,2.25) rectangle (2.75, 2.75);
\node[fill = white] at (1.5,3.5) {weights $S^*$};
\path[fill = gray, opacity = .3] (0.25,3.25) rectangle (2.75, 3.75);
\node[fill = white] at (1.5,4.5) {weights $S$};
\path[fill = gray, opacity = .3] (0.25,4.25) rectangle (2.75, 4.75);
\node[fill = white] at (1.5,5.5) {weights $S^*$};
\path[fill = gray, opacity = .3] (0.25,5.25) rectangle (2.75, 5.75);
\end{tikzpicture}}
\end{array}\]
\caption{Graphical depiction of steps in Lemma~\ref{rainbowlemma} when $n=3$.}
\label{cauchysteps}
\end{figure}

\begin{proof}[Proof of Proposition \ref{rainbowsystem}]
In the case $w=w_0$, this is Lemma \ref{rainbowlemma}. So we need only show that Cauchy lattice model satisfies the defining recursive relation of the $\beta$-Grothendieck polynomials, \[Z(\frak{R}_{ws_i}(\boldsymbol{x};\boldsymbol{y};\boldsymbol{z})) = \pi^{(\beta)}_{i}(Z(\mathfrak{R}_w(\boldsymbol{x};\boldsymbol{y};\boldsymbol{z}))), \hspace{20pt}\text{ when } \ell(ws_i) = \ell(w)-1.\]

Suppose the length condition holds. If we write $w$ in one-line notation as $c_1c_2\cdots c_n$, this occurs precisely when $c_i > c_{i+1}$. Since weights $S^*$ are a reflection of our original weights, these weights are solvable, and the $R$-vertex weights are a transformation of those in Figure \ref{RweightsDemazure}. Thus, as in Lemma \ref{permutationdecrease}, we can apply a train argument to rows $x_i$ and $x_{i+1}$. See the Figure \ref{dualRvertexweights} for the explicit weights of the $S^*$ row $R$-vertices, noting that the strand running southwest to northeast, labeled $x_i$ in those weights, has parameter $x_{i+1}$ in this case, and the other strand, labeled $x_j$ in the diagram, is $x_i$ in this case. If we attach an $R$-vertex to the left side of our system, we have two possible $R$-vertex types, $\tt{b}_2$ and $\tt{c}_2$. In the first case, the boundary conditions to the right of the $R$-vertex swap and we get partition function $(x_{i+1} - x_i) Z(\frak{R}_{ws_i}(\boldsymbol{x};\boldsymbol{y};\boldsymbol{z}))$. In the second case, the boundary conditions to the right of the $R$-vertex remain the same, yielding partition function $(1 + \beta x_{i+1}) Z(\mathfrak{R}_w(\boldsymbol{x};\boldsymbol{y};\boldsymbol{z}))$.

Moving the $R$-vertex through according to the Yang-Baxter equation, we have only one possibility for the $R$-vertex on the right side (type $\tt{a}$) so the partition function has weight $(1+\beta x_i) Z(\frak{R}_w(s_i\boldsymbol{x};\boldsymbol{y};\boldsymbol{z}))$. Solving for $Z(\frak{R}_{ws_i}(\boldsymbol{x};\boldsymbol{y}))$, we see that the recursive relation holds. Therefore, the partition functions of these two systems are equal, and thus we have that $Z(\frak{R}_w(\boldsymbol{x};\boldsymbol{y};\boldsymbol{z})) = \mathcal{G}^{(\beta)}_w(\boldsymbol{x};\boldsymbol{y})$ for all $w\in S_n$.
\end{proof}
\begin{figure}[h]
\centering
\scalebox{.95}{$\begin{array}{c@{\hspace{10pt}}c@{\hspace{30pt}}c@{\hspace{25pt}}c@{\hspace{10pt}}c@{\hspace{25pt}}}
\toprule
\tt{a} & \tt{b_1} & \tt{b_2} & \tt{c_1} & \tt{c_2} \\
\midrule
%  ++++ : 1+beta x_i
\begin{tikzpicture}[scale=0.7]
\draw[line width = .5mm, violet] (0,0) to [out = 0, in = 180] (2,2);
\draw[line width = .5mm, violet] (0,2) to [out = 0, in = 180] (2,0);
\draw[line width=0.5mm, violet, fill=white] (0,0) circle (.35);
\draw[line width=0.5mm, violet, fill=white] (0,2) circle (.35);
\draw[line width=0.5mm, violet, fill=white] (2,2) circle (.35);
\draw[line width=0.5mm, violet, fill=white] (2,0) circle (.35);
\node at (0,0) {$c$};
\node at (0,2) {$c$};
\node at (2,2) {$c$};
\node at (2,0) {$c$};
\end{tikzpicture}
&
\begin{tikzpicture}[scale=0.7]
\draw[line width = .5mm, red] (0,0) to [out = 0, in = 180] (2,2);
\draw[line width = .5mm, blue] (0,2) to [out = 0, in = 180] (2,0);
\draw[line width=0.5mm, red, fill=white] (0,0) circle (.35);
\draw[line width=0.5mm, blue, fill=white] (0,2) circle (.35);
\draw[line width=0.5mm, red, fill=white] (2,2) circle (.35);
\draw[line width=0.5mm, blue, fill=white] (2,0) circle (.35);
\node at (0,0) {$a$};
\node at (0,2) {$b$};
\node at (2,2) {$a$};
\node at (2,0) {$b$};
\end{tikzpicture}
&
\begin{tikzpicture}[scale=0.7]
\draw[line width = .5mm, blue] (0,0) to [out = 0, in = 180] (2,2);
\draw[line width = .5mm, red] (0,2) to [out = 0, in = 180] (2,0);
\draw[line width=0.5mm, blue, fill=white] (0,0) circle (.35);
\draw[line width=0.5mm, red, fill=white] (0,2) circle (.35);
\draw[line width=0.5mm, blue, fill=white] (2,2) circle (.35);
\draw[line width=0.5mm, red, fill=white] (2,0) circle (.35);
\node at (0,0) {$b$};
\node at (0,2) {$a$};
\node at (2,2) {$b$};
\node at (2,0) {$a$};
\end{tikzpicture}
&
\begin{tikzpicture}[scale=0.7]
\draw[line width = .5mm,blue] (0,2) to [out = 0, in = 120] (1,1);
\draw[line width = .5mm, blue] (2,2) to [out = 180, in=60] (1,1);
\draw[line width = .5mm, red] (0,0) to [out = 0, in = -120] (1,1);
\draw[line width = .5mm, red] (2,0) to [out = 180, in = -60] (1,1);
\draw[line width=0.5mm, red, fill=white] (0,0) circle (.35);
\draw[line width=0.5mm, blue, fill=white] (0,2) circle (.35);
\draw[line width=0.5mm, blue, fill=white] (2,2) circle (.35);
\draw[line width=0.5mm, red, fill=white] (2,0) circle (.35);
\node at (0,0) {$a$};
\node at (0,2) {$b$};
\node at (2,2) {$b$};
\node at (2,0) {$a$};
\end{tikzpicture}
&
\begin{tikzpicture}[scale=0.7]
\draw[line width = .5mm,red] (0,2) to [out = 0, in = 120] (1,1);
\draw[line width = .5mm, red] (2,2) to [out = 180, in=60] (1,1);
\draw[line width = .5mm, blue] (0,0) to [out = 0, in = -120] (1,1);
\draw[line width = .5mm, blue] (2,0) to [out = 180, in = -60] (1,1);
\draw[line width=0.5mm, blue, fill=white] (0,0) circle (.35);
\draw[line width=0.5mm, red, fill=white] (0,2) circle (.35);
\draw[line width=0.5mm, red, fill=white] (2,2) circle (.35);
\draw[line width=0.5mm, blue, fill=white] (2,0) circle (.35);
\node at (0,0) {$b$};
\node at (0,2) {$a$};
\node at (2,2) {$a$};
\node at (2,0) {$b$};
\end{tikzpicture}
\\
   \midrule
 1+\beta x_j - q^2(1+\beta x_i) & \beta^2q^2(x_i-x_j)  & x_i-x_j& (1-q^2)(1+\beta x_j)&  (1-q^2)(1+\beta x_i) \\
\bottomrule
\end{array}$}
    \caption{The row $R$-vertex weights that swap strands $i$ and $j$ both with weights $S^*$, where ${\color{red}a}<{\color{blue}b}$ and {\color{violet} $c$} is any color.}
    \label{dualRvertexweights}
\end{figure}
\begin{proof}[Proof of Theorem \ref{GenCauchy}]
By Proposition \ref{rainbowsystem}, $Z(\frak{R}_w(\boldsymbol{x};\boldsymbol{y};\boldsymbol{z})) = \mathcal{G}^{(\beta)}_w(\boldsymbol{x};\boldsymbol{y})$. On the other hand, we can evaluate this partition function in another way, by splitting the system $\frak{R}_w(\boldsymbol{x};\boldsymbol{y};\boldsymbol{z})$ at the middle (the dotted line in Figure \ref{Cauchyrainbow}) and evaluating each piece separately. Since each of the strands must pass up through this line in a different column, we can split the partition function into cases depending on which permutation appears on that line. Considering one of these cases, let $v^{-1}$ be the permutation at the dotted line. Then we see that through reflecting across $y=-1$ the bottom half is equivalent to the system $\frak{S}_{1,v^{-1}}(\boldsymbol{z},\boldsymbol{y})$, so the partition function of the bottom half is $\mathcal{G}^{(\beta)}_v(\boldsymbol{y};\boldsymbol{z})$. For the top half, the boundary conditions are precisely $\frak{S}_{v^{-1},w}(\boldsymbol{x},\ominus\boldsymbol{z})$, with a change of variables on the column parameters in the weights, so the partition function of the top half is $\mathcal{G}^{(\beta)}_{v^{-1},w}(\boldsymbol{x};\ominus\boldsymbol{z})$. Summing over possible midline permutations, we achieve the desired identity.
\end{proof}

In the case $w=w_0$, by Theorems~\ref{modelgivesgrotpolys} and~\ref{modelgivesdual}, $\mathcal{G}^{(\beta)}_{v^{-1},w_0}(\boldsymbol{x};\ominus \boldsymbol{z}) = \mathcal{H}^{(\beta)}_{vw_0}(\boldsymbol{x}; \ominus \boldsymbol{z})$. Thus Theorem~\ref{GenCauchy} becomes the following more familiar Cauchy identity involving Grothendieck polynomials and their duals. The three-variable version is first seen for Schubert polynomials in \cite{FominKirillov-schubert} and is stated for $\beta$-Grothendieck polynomials by Kirillov in \cite{Kirillov-quantum}, which credits the extended abstract \cite{FominKirillov-gpybe-abstract} of Fomin and Kirillov, suggesting that the three-variable version of the Cauchy identity may have been known earlier, though not written down to our knowledge. The more familiar two-variable version for $\beta$-Grothendieck polynomials is as seen in \cite{FominKirillov-gpybe}.

\begin{corollary}[Kirillov \cite{Kirillov-quantum}, Fomin-Kirillov \cite{FominKirillov-gpybe}]
\[ \mathcal{G}^{(\beta)}_{w_0}(\boldsymbol{x};\boldsymbol{y}) = \sum_{v\in S_n} \mathcal{G}^{(\beta)}_v(\boldsymbol{y};\boldsymbol{z})\mathcal{H}^{(\beta)}_{vw_0}(\boldsymbol{x}; \ominus \boldsymbol{z}).\]
In particular, if the column parameters are set to zero, 
\[ \mathcal{G}^{(\beta)}_{w_0}(\boldsymbol{x};\boldsymbol{y}) = \sum_{v\in S_n} \mathcal{G}^{(\beta)}_v(\boldsymbol{y})\mathcal{H}^{(\beta)}_{vw_0}(\boldsymbol{x}).\]
\end{corollary}

In general, diagrams of the type depicted in Figure~\ref{Cauchyrainbow} will result in Cauchy-style identities. By varying the boundary conditions, we obtain different sets of polynomials involved in such equations. Our final theorem of this section is the most general such identity we may prove by such a method.

\begin{theorem}[Generalized Cauchy Identity for Biaxial Polynomials]
\label{reallygeneralizedcauchy}
\[\mathcal{G}^{(\beta)}_{v,w}(\boldsymbol{x};\boldsymbol{y}) = \sum_{u \in S_n}\mathcal{G}^{(\beta)}_{v,u}(\boldsymbol{z};\boldsymbol{x})\mathcal{G}^{(\beta)}_{u,w}(\boldsymbol{y};\ominus \boldsymbol{z}).\]
\end{theorem}

\begin{proposition}\label{biaxialrainbowsystem}
For $v,w\in S_n$, let $\frak{R}_{v,w}(\boldsymbol{x};\boldsymbol{y};\boldsymbol{z})$ be the lattice model system $\frak{R}_w(\boldsymbol{x};\boldsymbol{y};\boldsymbol{z})$ from Proposition \ref{rainbowsystem} (see Figure \ref{Cauchyrainbow}), with one modification: the boundary labels on the bottom $n$ rows are now labeled by $v$ from top to bottom. Then $Z(\frak{R}_{v,w}(\boldsymbol{x};\boldsymbol{y};\boldsymbol{z})) = \mathcal{G}^{(\beta)}_{v,w}(\boldsymbol{x};\boldsymbol{y})$.
\end{proposition}

\begin{proof}
Since this system is precisely $\frak{R}_w(\boldsymbol{x};\boldsymbol{y};\boldsymbol{z})$ with a different permutation (namely $v$) across the lower half of the left boundary, the recursion established in the $x$ variables for $w$ in the proof of Lemma \ref{rainbowlemma} holds for $\frak{R}_{v,w}(\boldsymbol{x};\boldsymbol{y};\boldsymbol{z})$ by the same argument. So we have already that when $\ell(ws_i) = \ell(w)-1,$
\[ Z(\frak{R}_{v,ws_i}(\boldsymbol{x};\boldsymbol{y};\boldsymbol{z})) = \pi_i(Z( \frak{R}_{v,w}(\boldsymbol{x};\boldsymbol{y};\boldsymbol{z}))). \]
We may also establish a recursion in the $y$ variables using a train argument with the YBE for weights $S$. Suppose that $\ell(vs_i) = \ell(v) + 1$. If we write $v$ in one line notation $c_1c_2\ldots c_n$, this occurs precisely when $c_i < c_{i+1}$. Starting with the partition function of the following system, we observe that the $R$-vertex on the left has two possibilities: type $\tt{b}_2$ or type $\tt{c}_2$. In the former case, we will get a contribution of $(y_{i+1} - y_i)Z(\frak{R}_{vs_i,w}(\boldsymbol{x};\boldsymbol{y};\boldsymbol{z}))$, and in the latter a contribution of $(1+\beta y_i)Z(\frak{R}_{v,w}(\boldsymbol{x};\boldsymbol{y};\boldsymbol{z}))$.

After repeatedly applying the Yang-Baxter equation, the only possibility for the $R$-vertex on the right side is type $\tt{a}$, giving the partition function $(1+\beta y_i) Z(\frak{R}_{v,w}(\boldsymbol{x};s_i\boldsymbol{y};\boldsymbol{z}))$.

Combining both of these recursive steps with the base case $\frak{R}_{1,w_0}(\boldsymbol{x};\boldsymbol{y};\boldsymbol{z})$ proven in Lemma \ref{rainbowlemma}, we have precisely the defining conditions for $\mathcal{G}^{(\beta)}_{v,w}(\boldsymbol{x};\boldsymbol{y})$ in Definition~\ref{biaxial-polynomials}. Therefore, 
\[ Z(\mathfrak{R}_{v,w}(\boldsymbol{x};\boldsymbol{y};\boldsymbol{z})) = \mathcal{G}^{(\beta)}_{v,w}(\boldsymbol{x};\boldsymbol{y}). \qedhere \]
\end{proof}

\begin{proof}[Proof of Theorem \ref{reallygeneralizedcauchy}]
We evaluate $Z(\mathfrak{R}_{v,w}(\boldsymbol{x};\boldsymbol{y};\boldsymbol{z}))$ in two ways as in the proof of Theorem \ref{GenCauchy}. By Proposition \ref{biaxialrainbowsystem}, we know that $Z(\frak{R}_{v,w}(\boldsymbol{x};\boldsymbol{y};\boldsymbol{z})) = \mathcal{G}^{(\beta)}_{v,w}(\boldsymbol{x};\boldsymbol{y})$.

On the other hand, we can split the system $\mathfrak{R}_{v,w}(\boldsymbol{x};\boldsymbol{y};\boldsymbol{z})$ across its horizontal midline (the dotted line of Figure \ref{Cauchyrainbow}) and sum over permutations $u$ that appear on that split (reading off $u$ from left to right). On the top half, as in Theorem \ref{GenCauchy}, we have the system $\mathcal{G}^{(\beta)}_{u,w}(\boldsymbol{x};\ominus \boldsymbol{z})$. On the bottom half, we have a system with left boundary $v$, top boundary $u$, and other boundary edges all labeled $+$ and weights $S$. By reflecting across the line $y=-x$, this has the same partition function as $\mathfrak{S}_{v,u}(\boldsymbol{z},\boldsymbol{y})$. Thus, its partition function is $ \mathcal{G}^{(\beta)}_{v,u}(\boldsymbol{z};\boldsymbol{y})$.

Summing over permutations $u$ on the midline, we obtain the desired identity,
\[ \mathcal{G}^{(\beta)}_{v,w}(\boldsymbol{x};\boldsymbol{y}) = \sum_{u\in S_n} \mathcal{G}^{(\beta)}_{v,u}(\boldsymbol{z};\boldsymbol{y})\mathcal{G}^{(\beta)}_{u,w}(\boldsymbol{x};\ominus \boldsymbol{z}). \qedhere \]
\end{proof}

We now show that some of the terms in Theorem~\ref{reallygeneralizedcauchy} are in fact zero, so that the sum can be written over an interval in the Bruhat order.

\begin{proposition}
$\mathcal{G}^{(\beta)}_{v,w}(\boldsymbol{x};\boldsymbol{y}) = 0$ unless $v\leq w$ in the strong Bruhat order on $S_n$.
\end{proposition}

\begin{proof}
We prove this by observing the allowed crossings of colored paths travelling in through the top boundary of the chromatic model and out through the left boundary in an admissible state. We begin with strands labeled by $v = v(1),...,v(n)$ from left to right, and we end with strands labeled by $w = w(1),\ldots,w(n)$ from top to bottom. We view the strands travelling through the model as a braid, and each vertex involving two strands crossing corresponds to a simple reflection acting on the right of our permutation. The only vertices that cross strands, and therefore change the permutation are $\tt{b}_1$ and $\tt{b}_2$. Each $\tt{b}_1$ vertex increases the length of the permutation by adding an inversion, while each $\tt{b}_2$ vertex decreases the length of the permutation by removing an inversion.

When $q=0$, $\tt{b}_2$ is not an admissible vertex, so the permutation can only increase in length. Therefore, unless $v\leq w$ in the strong Bruhat order on $S_n$, the model has no admissible states, and so $\mathcal{G}^{(\beta)}_{v,w}(\boldsymbol{x};\boldsymbol{y}) = 0$.
\end{proof}

\begin{corollary}
Applying the previous proposition to Theorem \ref{reallygeneralizedcauchy} to remove zero terms we arrive at the following:
\[ \mathcal{G}^{(\beta)}_{v,w}(\boldsymbol{x};\boldsymbol{y}) = \sum_{\substack{u\in S_n\\v\leq u\leq w}} \mathcal{G}^{(\beta)}_{v,u}(\boldsymbol{z};\boldsymbol{y})\mathcal{G}^{(\beta)}_{u,w}(\boldsymbol{x};\ominus \boldsymbol{z}).\]
\end{corollary}

\section{A Branching Rule for Double $\beta$-Grothendieck Polynomials}\label{branching}
In this section, we will give a branching rule for the Grothendieck polynomial $\mathcal{G}^{(\beta)}_w(\boldsymbol{x};\boldsymbol{y})$. If $w\in S_n$, our rule gives a formula for $\mathcal{G}^{(\beta)}_w(\boldsymbol{x};\boldsymbol{y})$ in terms of Grothendieck polynomials for permutations in $S_{n-1}$. We choose the name ``branching rule'' because this process is reminiscent of branching rules from representation theory, and because for lattice models that encode characters of representations, the process we describe indeed does give a representation-theoretic branching rule (see \cite{BBFschur}).

In Section \ref{onerownonzerosection}, we will give a condition using the weak and strong Bruhat orders which determines when the one-row partition function of a modification of $\mathfrak{S}_{1,w^{-1}}(\boldsymbol{y},\boldsymbol{x})$ is non-zero. In Section \ref{interleavingsection}, we will show that this generalizes the interleaving condition for non-chromatic 5-vertex models, and in Section \ref{branchingrulesubsection}, we will use our generalized interleaving condition to determine a branching rule for $\mathcal{G}^{(\beta)}_w(\boldsymbol{x};\boldsymbol{y})$ (Corollary \ref{branchingrule}). Once again, as we are only considering the double $\beta$-Grothendieck polynomials, we set $q=0$ for this section, making the relevant model a 5-vertex model.

\subsection{Generalized Interleaving Condition} \label{onerownonzerosection}

We give a rule for when the one-row partition function is nonzero, which generalizes the interleaving condition from non-chromatic 5-vertex models to the chromatic case (see, for example, Lemma~5 of \cite{BBFschur}). We will work with a modification of the system $\mathfrak{S}_{1,w^{-1}}(\boldsymbol{y},\boldsymbol{x})$ for this computation, where we simply reflect the entire model along $y = -x$, including vertex weights, boundary conditions, and row and column parameters. We will denote this system $\mathfrak{S}_{1,w^{-1}}'(\boldsymbol{y},\boldsymbol{x})$ for the remainder of this section. As we will need to refer to specific vertex types in this section, the vertex weights and corresponding labels are given in Figure \ref{essentiallymodelP}. Note that by Proposition~\ref{hudson}, the partition function of this system is $\mathcal{G}_{w}^{(\beta)}(\boldsymbol{x},\boldsymbol{y})$.
We use the symbols $\le,\le_L\le_R$ to denote the strong Bruhat order, left (weak) Bruhat order, and right (weak) Bruhat order, respectively.

\begin{figure}
\centering
\scalebox{.95}{
$
\begin{array}{c@{\hspace{10pt}}c@{\hspace{10pt}}c@{\hspace{10pt}}c@{\hspace{10pt}}c@{\hspace{0pt}}}
\toprule
\tt{a}&\tt{b}_1&\tt{b_2}&\tt{c}_1&\tt{c}_2\\
\midrule
\begin{tikzpicture}
\coordinate (a) at (-.75, 0);
\coordinate (b) at (0, .75);
\coordinate (c) at (.75, 0);
\coordinate (d) at (0, -.75);
\coordinate (aa) at (-.75,.5);
\coordinate (cc) at (.75,.5);
\draw[line width=0.5mm, violet] (a)--(0,0);
\draw[line width=0.6mm, violet] (b)--(0,0);
\draw[line width=0.5mm, violet] (c)--(0,0);
\draw[line width=0.6mm, violet] (d)--(0,0);
\draw[line width=0.5mm, violet,fill=white] (a) circle (.25);
\draw[line width=0.5mm, violet,fill=white] (b) circle (.25);
\draw[line width=0.5mm, violet, fill=white] (c) circle (.25);
\draw[line width=0.5mm, violet, fill=white] (d) circle (.25);
\node at (0,1) { };
\node at (a) {$c$};
\node at (b) {$c$};
\node at (c) {$c$};
\node at (d) {$c$};
\end{tikzpicture}
&
\begin{tikzpicture}
\coordinate (a) at (-.75, 0);
\coordinate (b) at (0, .75);
\coordinate (c) at (.75, 0);
\coordinate (d) at (0, -.75);
\coordinate (aa) at (-.75,.5);
\coordinate (cc) at (.75,.5);
\draw[line width=0.5mm, blue] (a)--(c);
\draw[line width=0.6mm, red] (b)--(d);
\draw[line width=0.5mm,blue,fill=white] (a) circle (.25);
\draw[line width=0.5mm,blue,fill=white] (c) circle (.25);
\draw[line width=0.5mm,red,fill=white] (b) circle (.25);
\draw[line width=0.5mm,red,fill=white] (d) circle (.25);
\node at (0,1) { };
\node at (a) {$b$};
\node at (b) {$a$};
\node at (c) {$b$};
\node at (d) {$a$};
\end{tikzpicture}
&
\begin{tikzpicture}
\coordinate (a) at (-.75, 0);
\coordinate (b) at (0, .75);
\coordinate (c) at (.75, 0);
\coordinate (d) at (0, -.75);
\coordinate (aa) at (-.75,.5);
\coordinate (cc) at (.75,.5);
\draw[line width=0.5mm, red] (a)--(c);
\draw[line width=0.6mm, blue] (b)--(d);
\draw[line width=0.5mm,red,fill=white] (a) circle (.25);
\draw[line width=0.5mm,red,fill=white] (c) circle (.25);
\draw[line width=0.5mm,blue,fill=white] (b) circle (.25);
\draw[line width=0.5mm,blue,fill=white] (d) circle (.25);
\node at (0,1) { };
\node at (a) {$a$};
\node at (b) {$b$};
\node at (c) {$a$};
\node at (d) {$b$};
\end{tikzpicture}
%%%%%%%
& \begin{tikzpicture}
\coordinate (a) at (-.75, 0);
\coordinate (b) at (0, .75);
\coordinate (c) at (.75, 0);
\coordinate (d) at (0, -.75);
\coordinate (aa) at (-.75,.5);
\coordinate (cc) at (.75,.5);
\draw[line width=0.5mm, blue](a)--(0,0)--(b);
\draw[line width=0.5mm,blue,fill=white] (b) circle (.25);
\draw[line width=0.5mm,blue,fill=white] (a) circle (.25);
\draw[line width=0.5mm, red](d)--(0,0)--(c);
\draw[line width=0.5mm,red,fill=white] (c) circle (.25);
\draw[line width=0.5mm,red,fill=white] (d) circle (.25);
\node at (0,1) { };
\node at (a) {$b$};
\node at (b) {$b$};
\node at (c) {$a$};
\node at (d) {$a$};
\end{tikzpicture}
%%%%%%%
& \begin{tikzpicture}
\coordinate (a) at (-.75, 0);
\coordinate (b) at (0, .75);
\coordinate (c) at (.75, 0);
\coordinate (d) at (0, -.75);
\coordinate (aa) at (-.75,.5);
\coordinate (cc) at (.75,.5);
\draw[line width=0.5mm, red](a)--(0,0)--(b);
\draw[line width=0.5mm,red,fill=white] (b) circle (.25);
\draw[line width=0.5mm,red,fill=white] (a) circle (.25);
\draw[line width=0.5mm, blue](d)--(0,0)--(c);
\draw[line width=0.5mm,blue,fill=white] (c) circle (.25);
\draw[line width=0.5mm,blue,fill=white] (d) circle (.25);
\node at (0,1) { };
\node at (a) {$a$};
\node at (b) {$a$};
\node at (c) {$b$};
\node at (d) {$b$};
\end{tikzpicture}
%%%%%%%%%
\\
   \midrule
 1 & 0 & x_i\oplus y_j & 1+\beta (x_i\oplus y_j) & 1  \\
   \bottomrule
\end{array}
$}
\caption{The Boltzmann weights for $\mathfrak{S}_{1,w^{-1}}'(\boldsymbol{y},\boldsymbol{x})$ at a vertex in row $i$ and column $j$, where $\textcolor{red}{a}<\textcolor{blue}{b}$ and ${\color{violet} c}$ is any color and we have set $q=0$. We consider the $+$ label to be larger than any color, and the same weights hold when one or more labels are $+$.}
\label{essentiallymodelP}
\end{figure}

Let $w\in S_n$, and let $w_i := w^{-1}(i)$, so $\eta_w:= w_1w_2\ldots w_n$ is the one-line notation of $w^{-1}$. We then have a natural left action of $S_n$ on $\eta_w$ by \[s_i\cdot (w_1\ldots w_iw_{i+1}\ldots w_n) = w_1\ldots w_{i+1}w_i\ldots w_n;\] and under this action, \begin{equation}v\cdot \eta_w = \eta_{vw}.\label{actiononelineequation}\end{equation}

For $w\in S_n$, define the permutations \[w_{\max} := s_1s_2\ldots s_{k-1}w, \hspace{20pt} \text{where } k=w(1), \text{ and} \] \[w^- := w_{\max}|_{2,\ldots,n} \in S_{n-1},\] and let $w^+$ be the image of $w$ in $S_{n+1}$ via the map $\{1,\ldots,n\}\to \{2,\ldots,n+1\}$, $i\mapsto i+1$. The permutation $w^-$ is well-defined since $w_{\max}(1) = 1$. Note that $(w^-)^+ = w_{\max}$ and $(w^+)^- = w$.

In terms of one-line notation, $\eta_{w^-}$ is obtained from $\eta_w$ by removing 1 and decreasing all remaining entries by 1, $\eta_{w^+}$ is obtained from $\eta_w$ by adding 1 to the front and increasing all other entries by 1, and $\eta_{w_{\max}}$ is obtained from $\eta_w$ by shifting 1 to the front.

Further, define $w_{\min}\in S_n$ to be the permutation \[w_{\min} = s_1^{\delta_1}\cdots s_{n-1}^{\delta_{n-1}} w, \hspace{30pt} \delta_i = \begin{cases} 1, & \text{if }  s_is_{i+1}^{\delta_{i+1}}\cdots s_{n-1}^{\delta_{n-1}}w < s_{i+1}^{\delta_{i+1}}\cdots s_{n-1}^{\delta_{n-1}}w \\ 0, & \text{else.}\end{cases}\] In other words, $w_{\min}$ is the minimal permutation in the left Bruhat order obtained from left multiplying $w$ by a subexpression of $s_1\cdots s_{n-1}$. An equivalent definition for $\delta_i$ is \[\delta_i = \begin{cases} 1, & \text{if there exists } j>i \text{ such that } w_i>w_j \\ 0, & \text{else.}\end{cases}\]

\begin{example}
If we set $w = s_1s_2s_3s_2s_1s_5 = (14)(56)$, then $\eta_w = 423165$, $\eta_{w^-} = 31254, \eta_{w^+} = 1534276, \eta_{w_{\max}} = 142365$, and $\eta_{w_{\min}} = 142356$. 
\end{example}

We will prove the following result:
\begin{theorem} \label{interleaving}
Let $w,v\in S_n$ with $v(1) = 1$. The one-row partition function $T(w,v)$ in Figure \ref{modelPoneline} is nonzero if and only if $w_{\min} \le v\le_L w_{\max}$.
\end{theorem}

Our first step in proving this theorem is the following lemma.

\begin{figure}[h]
\begin{center}
\scalebox{0.8}{
\begin{tikzpicture}
  \coordinate (aa) at (1,1);
  \coordinate (ab) at (1,2);
  \coordinate (ac) at (1,3);
  \coordinate (ba) at (2,1);
  \coordinate (bb) at (2,2);
  \coordinate (bc) at (2,3);
  \coordinate (ca) at (3,1);
  \coordinate (cb) at (3,2);
  \coordinate (cc) at (3,3);
  \coordinate (da) at (4,1);
  \coordinate (db) at (4,2);
  \coordinate (dc) at (4,3);
  \coordinate (ea) at (5,1);
  \coordinate (eb) at (5,2);
  \coordinate (ec) at (5,3);
  \coordinate (fa) at (6,1);
  \coordinate (fb) at (6,2);
  \coordinate (fc) at (6,3);
  \coordinate (ga) at (7,1);
  \coordinate (gb) at (7,2);
  \coordinate (gc) at (7,3);
  \coordinate (ha) at (8,1);
  \coordinate (hb) at (8,2);
  \coordinate (hc) at (8,3);
  \coordinate (ia) at (9,1);
  \coordinate (ib) at (9,2);
  \coordinate (ic) at (9,3);
  \coordinate (ja) at (10,1);
  \coordinate (jb) at (10,2);
  \coordinate (jc) at (10,3);
  \coordinate (ka) at (11,1);
  \coordinate (kb) at (11,2);
  \coordinate (kc) at (11,3);
  \draw (ab)--(kb);
  \draw (ba)--(bc);
  \draw (da)--(dc);
  \draw (ha)--(hc);
  \draw (ja)--(jc);
  \draw[fill=white] (ba) circle (.4);
  \draw[fill=white] (da) circle (.4);
  \draw[fill=white] (ha) circle (.4);
  \draw[fill=white] (ja) circle (.4);
  \draw[fill=white] (bc) circle (.4);
  \draw[fill=white] (dc) circle (.4);
  \draw[fill=white] (hc) circle (.4);
  \draw[fill=white] (jc) circle (.4);
  \draw[fill=white] (ab) circle (.4);
  \draw[fill=white] (cb) circle (.4);
  \draw[fill=white] (eb) circle (.4);
  \draw[fill=white] (gb) circle (.4);
  \draw[fill=white] (ib) circle (.4);
  \draw[fill=white] (kb) circle (.4);
  \path[fill=white] (bb) circle (.4);
  \path[fill=white] (db) circle (.4);
  \path[fill=white] (fb) circle (.4);
  \path[fill=white] (hb) circle (.5);
  \path[fill=white] (jb) circle (.4);
  \node at (bb) {$T_{1,1}$};
  \node at (db) {$T_{1,2}$};
  \node at (fc) {$\ldots$};
  \node at (fb) {$\ldots$};
  \node at (fa) {$\ldots$};
  \node at (hb) {$T_{1,{n-1}}$};
  \node at (jb) {$T_{1,n}$};
  
  \node at (-1,2) {row:};
  \node at (-0.2,2) {$1$};
  \node at (bc) {$w_1$};
  \node at (dc) {$w_2$};
  \node at (fc) {$\ldots$};
  \node at (hc) {{\footnotesize $w_{n-1}$}};
  \node at (jc) {$w_n$};
  \node at (ab) {$1$};
  \node at (ba) {$v_2$};
  \node at (da) {$v_3$};
  \node at (fa) {$\ldots$};
  \node at (ha) {$v_n$};
  \node at (ja) {$+$};
  \node at (kb) {$+$};
  \node at (10,4) {$n$};
  \node at (8,4) {$n-1$};
  \node at (6,4) {$\ldots$};
  \node at (4,4) {$2$};
  \node at (2,4) {$1$};
  \node at (1,4.04) {column:};
\end{tikzpicture}}
\end{center}
\caption{The one-row partition function for $\mathfrak{S}_{1,w^{-1}}'(\boldsymbol{y},\boldsymbol{x})$, where $w_i := w^{-1}(i), v_i := v^{-1}(i)$.}
\label{modelPoneline}
\end{figure}
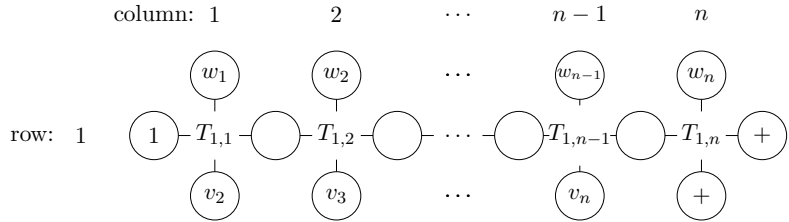

\begin{lemma}
For a given $w\in S_n$ the permutations $v\in S_n$ such that the partition function in Figure \ref{modelPoneline} is nonzero form the set \[A_w = \left\{v\in S_n \;|\;  v(1) = 1, \; v =  s_1^{\epsilon_1} \ldots s_{n-1}^{\epsilon_{n-1}} w ,\; \text{where } \epsilon_i \in \{0,1\},\; l(w) = l(v) + \sum \epsilon_i\right\}.\]
\end{lemma}

\begin{proof}

Figure \ref{modelPoneline} requires that $v(1)=1$, and a vertex of type $\tt{b}_2$ with $a <b< +$ in column $i$ corresponds to a simple reflection $s_i$ applied to $\eta_v$. Each $s_i$ must be length-increasing or the weight is 0 by the conditions on the colors for the vertices $\tt{b}_1$ and $\tt{b}_2$, and the length of $w$ is the length of $v$ plus the number of vertices of type $\tt{b}_2$ with $a <b< +$, which is $\sum_i \epsilon_i$.

The condition $v_i = w_j$ with $i\le j$ requires exactly that there is a type $\tt{b}_2$ vertex in columns $i,i+1,\ldots,j-1$, but not in column $i-1$ or column $j$. This corresponds to an action of $s_{j-1}s_{j-2}\ldots s_i$ on $\eta_v$. Factors of this form commute, so we can write $\eta_w = s_{n-1}^{\epsilon_{n-1}}\cdots s_1^{\epsilon_1}\cdot \eta_v$ for some $\epsilon_1,\ldots,\epsilon_{n-1}\in\{0,1\}$. By (\ref{actiononelineequation}), $w = s_{n-1}^{\epsilon_{n-1}}\cdots s_1^{\epsilon_1} v$. Given these conditions on $v$, there is a unique way to fill in the lattice with nonzero weight, so the partition function is nonzero.
\end{proof}

\begin{proof}[Proof of Theorem \ref{interleaving}]
It is clear from the definitions that $w_{\min}$ and $w_{\max}$ are elements of $A_w$. Let $v = s_1^{\epsilon_1}\ldots s_{n-1}^{\epsilon_{n-1}}w\in A_w$. We show that $w_{\min}\le v\le_L w_{\max}$.

If $\epsilon_i=1$, then the application of $s_i$ to $v_{i+1} := s_{i+1}^{\epsilon_{i+1}}\ldots s_{n-1}^{\epsilon_{n-1}}w$ must be length decreasing. Since $v_{i+1}^{-1}(i) = w_i$ and $v_{i+1}^{-1}(i+1) = w_j$ for some $j>i$, we must have $w_i>w_j$ for some $i<j$. Therefore, $\epsilon_i\le \delta_i$ for all $i$, so $w_{\min}\le v$.

On the other hand, $v(1)=1$, so if $w_k=1$ (i.e. $w(1)=k$), then we must have $\epsilon_i=1$ for $i<k$, but $\epsilon_k=0$. Therefore, $s_1s_2\cdots s_{k-1}$ commutes with the other $s_i^{\epsilon_i}$, so we can write \[v = s_{k+1}^{\epsilon_{k+1}}\ldots s_{n-1}^{\epsilon_{n-1}}s_1s_2\cdots s_{k-1}w = s_{k+1}^{\epsilon_{k+1}}\ldots s_{n-1}^{\epsilon_{n-1}}w_{\max},\] so $v\le_L w_{\max}$.

Conversely, suppose that $w_{\min}\le v\le_L w_{\max}$. Since $v\le_L w_{\max}$, $v(1)=1$. Furthermore, $v\le_L w_{\max}\le_L w$, so we can write $w = uv$, where $l(u)+l(v) = l(w)$. We want to show that $u$ has a reduced expression of the form $s_1^{\epsilon_1}\ldots s_{n-1}^{\epsilon_{n-1}}$. Consider the left (resp. right) Bruhat interval of elements less than $w$, which we call $[1,w]_L$ (resp. $[1,w]_R$). By \cite[Proposition~3.12]{Takigiku}, the map $\Phi_w:[1,w]_L\to [1,w]_R, \sigma\mapsto w\sigma^{-1}$ is an order-reversing bijection for both the weak and strong orders. In other words, we have \[\Phi_w(w_{\max}) = s_{k-1}\ldots s_1 \le_R \Phi_w(v)\le \Phi_w(w_{\min}) = s_{n-1}^{\delta_{n-1}}\ldots s_1^{\delta_1}.\] In other words, $\Phi_w(v)$ can be written $s_{n-1}^{\epsilon_{n-1}}\ldots s_1^{\epsilon_1}$, where $\epsilon_i\le \delta_i$, so \[v = \Phi_w(v)^{-1}w = s_1^{\epsilon_1}\ldots s_{n-1}^{\epsilon_{n-1}}w,\] and since everything above is a reduced expression, $v\in A_w$.
\end{proof}

The following general branching rule follows directly:

\begin{corollary} \label{generalbranching}
If $w\in S_n$, we have \[Z(\mathfrak{S}'_{1,w^{-1}}) = \sum_{w_{\min}\le v\le_L w_{\max}} Z(\mathfrak{S}'_{1, (v^-)^{-1}}) T(w,v),\] and every $T(w,v)$ that appears is nonzero.
\end{corollary}

\begin{remark}
Everything in this subsection depends only on which vertices are admissible, and not on their precise weights. Therefore, our results hold for any chromatic lattice model that has the same boundary conditions and admissible vertices as $\mathfrak{S}'_{1,w^{-1}}(\boldsymbol{y}, \boldsymbol{x})$. Thus, we can think of Corollary \ref{generalbranching} as a statement about an arbitrary chromatic 5-vertex model.
\end{remark}

\subsection{Reduction to the Non-Chromatic Case} \label{interleavingsection}

Now we show that Theorem \ref{interleaving} generalizes the interleaving condition for the non-chromatic 5-vertex model.

Recall that a Grassmannian permutation is a permutation $w$ with at most one descent. Say that this descent is at $b=b_w$, so $w(1)<w(2)<\ldots<w(b)>w(b+1)<w(b+2)<\ldots<w(n)$; equivalently, $1,\ldots,b$ forms a subword of $\eta_w$, and so does $b+1,\ldots,n$. We can associate to $w$ the partition $\lambda_w$, where $\ell(\lambda_w)\le b$ and $(\lambda_w)_{b+1-i} = w(i)-i$. 

\begin{definition} \label{interleavingdefinition}
We say that two partitions $\mu = (\mu_1,\ldots,\mu_b)$ and $\nu = (\nu_1,\ldots,\nu_b)$ satisfy the interleaving condition relative to $b$ if $\nu_b=0$, and for all $i$, $\mu_i \ge \nu_i \ge \nu_{i+1}$.
\end{definition}

\begin{remark}
Definition \ref{interleavingdefinition} is the condition needed for the following partition function to be nonzero: referring to Figure \ref{onelinenonchromatic}, we have that each part $\mu_i$ for $1\le i\le b$ corresponds to a $-$ in column $\mu_i+i$, and each part $\nu_i$ corresponds to a $-$ in column $\nu_i+i-1$ (column 0 is taken to be the circle to the left of the row). Note that we consider $\nu$ to have $b$ parts since this matches better with the chromatic model, whereas it is more standard \cite{BBFschur} to consider it to have $b-1$ parts.

\begin{figure}[h]
\begin{center}
\scalebox{0.8}{
\begin{tikzpicture}

  \coordinate (aa) at (1,1);
  \coordinate (ab) at (1,2);
  \coordinate (ac) at (1,3);
  \coordinate (ba) at (2,1);
  \coordinate (bb) at (2,2);
  \coordinate (bc) at (2,3);
  \coordinate (ca) at (3,1);
  \coordinate (cb) at (3,2);
  \coordinate (cc) at (3,3);
  \coordinate (da) at (4,1);
  \coordinate (db) at (4,2);
  \coordinate (dc) at (4,3);
  \coordinate (ea) at (5,1);
  \coordinate (eb) at (5,2);
  \coordinate (ec) at (5,3);
  \coordinate (fa) at (6,1);
  \coordinate (fb) at (6,2);
  \coordinate (fc) at (6,3);
  \coordinate (ga) at (7,1);
  \coordinate (gb) at (7,2);
  \coordinate (gc) at (7,3);
  \coordinate (ha) at (8,1);
  \coordinate (hb) at (8,2);
  \coordinate (hc) at (8,3);
  \coordinate (ia) at (9,1);
  \coordinate (ib) at (9,2);
  \coordinate (ic) at (9,3);
  \coordinate (ja) at (10,1);
  \coordinate (jb) at (10,2);
  \coordinate (jc) at (10,3);
  \coordinate (ka) at (11,1);
  \coordinate (kb) at (11,2);
  \coordinate (kc) at (11,3);
  \draw (ab)--(kb);
  \draw (ba)--(bc);
  \draw (da)--(dc);
  \draw (ha)--(hc);
  \draw (ja)--(jc);
  \draw[fill=white] (ba) circle (.25);
  \draw[fill=white] (da) circle (.25);
  \draw[fill=white] (ha) circle (.25);
  \draw[fill=white] (ja) circle (.25);
  \draw[fill=white] (bc) circle (.25);
  \draw[fill=white] (dc) circle (.25);
  \draw[fill=white] (hc) circle (.25);
  \draw[fill=white] (jc) circle (.25);
  \draw[fill=white] (ab) circle (.25);
  \draw[fill=white] (cb) circle (.25);
  \draw[fill=white] (eb) circle (.25);
  \draw[fill=white] (gb) circle (.25);
  \draw[fill=white] (ib) circle (.25);
  \draw[fill=white] (kb) circle (.25);
  \path[fill=white] (bb) circle (.4);
  \path[fill=white] (db) circle (.4);
  \path[fill=white] (fb) circle (.4);
  \path[fill=white] (hb) circle (.5);
  \path[fill=white] (jb) circle (.4);
  \node at (bb) {$T_{1,1}$};
  \node at (db) {$T_{1,2}$};
  \node at (fc) {$\ldots$};
  \node at (fb) {$\ldots$};
  \node at (fa) {$\ldots$};
  \node at (hb) {$T_{1,{n-1}}$};
  \node at (jb) {$T_{1,n}$};
  \node at (ja) {$+$};
  \node at (kb) {$+$};
  \draw [->,>=stealth] (2,3.7)--(10,3.7);
  \node at (6,4) {$\mu$};
  \draw [->,>=stealth] (1,0.3)--(8,0.3);
  \node at (4,0) {$\nu$};
\end{tikzpicture}}
\end{center}
\caption{The one-row partition function for $\mathfrak{S}'_{1,w^{-1}}$, where the top and bottom boundaries are Grassmannian permutations.}\label{onelinenonchromatic}
\end{figure}
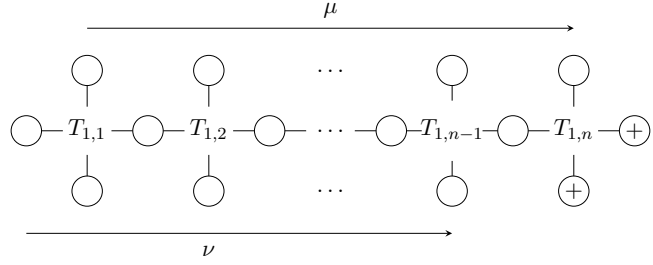
\end{remark}

\begin{lemma} \label{grassmannianbruhatorderequivalence}
If $u,v\in S_n$ are Grassmannian permutations, both with descent at $b$, then $u\le_L v$ if and only if $u\le v$.
\end{lemma}

\begin{proof}
Suppose that $u\le v$, and since $v$ is Grassmannian, $v$ has a reduced expression of the form \[\und{v} = s_{v(1)}s_{v(1)-1}\ldots s_1 s_{v(2)}s_{v(2)-1}\ldots s_2 \ldots s_{v(b)}s_{v(b)-1}\ldots s_b,\] and there exists some reduced expression $\und{u}$ that is a subword of $\und{v}$. Call $s_{v(k)}s_{v(k)-1}\ldots s_k$ the $k$-block of $\und{v}$, and call its subword that appears in $\und{u}$ the $k$-block of $u$.

We have $\und{v}\cdot \eta_1 = \eta_v$, and a key feature of a reduced expression of a Grassmannian permutation is that every application of a simple reflection swaps an element of $\{1,\ldots,b\}$ with an element of $\{b+1,\ldots,n\}$. For $\und{u}$ to still be Grassmannian, it must also have this property, which amounts to requiring that the $k$-block of $u$ is $s_{u(k)}s_{u(k)-1}\ldots s_k$.

Therefore, \[\und{u} = s_{u(1)}s_{u(1)-1}\ldots s_1 s_{u(2)}s_{u(2)-1}\ldots s_2 \ldots s_{u(b)}s_{u(b)-1}\ldots s_b,\] and notice that $s_{v(j)}s_{v(j)-1}\ldots s_{u(j)+1}$ commutes with every simple reflection in the $k$-block of $u$ for $k<j$. Therefore, \begin{align*} v &= s_{v(1)}s_{v(1)-1}\ldots s_1 \ldots s_{v(b)}s_{v(b)-1}\ldots s_b \\&= s_{v(1)}\ldots s_{u(1)+1}s_{u(1)}\ldots s_1 \ldots s_{v(b)}\ldots s_{u(b)+b}s_{u(b)}\ldots s_b \\&= s_{v(1)}\ldots s_{u(1)+1}\ldots s_{v(b)}\ldots s_{u(b)+b} s_{u(1)}\ldots s_1 \ldots s_{u(b)}\ldots s_b \\&= s_{v(1)}\ldots s_{u(1)+1}\ldots s_{v(b)}\ldots s_{u(b)+b} u,\end{align*} and this is a reduced expression, so $u\le_L v$.
\end{proof}

\begin{lemma} \label{grassmannianpartition}
If $u,v\in S_n$ are Grassmannian permutations, both with descent at $b$, then $u\le v$ is equivalent to the condition that $(\lambda_u)_i \le (\lambda_v)_i$ for all $1\le i\le b$.
\end{lemma}

\begin{proof}
The lengths of the $k$-blocks of $u$ and $v$ are $(\lambda_u)_k$ and $(\lambda_v)_k$, respectively. The reduced words from Lemma \ref{grassmannianbruhatorderequivalence} show that $u\le v$ if and only if every $k$-block of $v$ is at least as long as the corresponding $k$-block of $u$.
\end{proof}

The following proposition shows that for Grassmannian permutations, Theorem \ref{interleaving} is equivalent to the interleaving condition from the non-chromatic 5-vertex model.

\begin{proposition}
If $w,v\in S_n$ are Grassmannian permutations with a descent at $b$, then $\lambda_w$ and $\lambda_v$ satisfy the interleaving condition if and only if $w_{\min}\le v\le_L w_{\max}$.
\end{proposition}

\begin{proof}
We obtain $\eta_{w_{\max}}$ by moving 1 to the front of $\eta_w$, so $\lambda_{w_{\max}}$ is obtained by setting the $b$-th part of $\lambda_w$ to 0. By Lemmas \ref{grassmannianbruhatorderequivalence} and \ref{grassmannianpartition}, the condition $v\le_L w_{\max}$ therefore is equivalent to the condition that $(\lambda_v)_i\le (\lambda_w)_i$ for all $i$ and that $(\lambda_v)_b = 0$.

We obtain $\eta_{w_{\min}}$ from $\eta_w$ by moving every entry $i\in \{2,\ldots,b\}$ to $w(i-1)+1$, and moving 1 to the front. In other words, $\lambda_{w_{\min}}$ is the partition obtained from $\lambda_w$ by removing the first part. By Lemma \ref{grassmannianpartition}, the condition $w_{\min}\le v$ given that $v(1)=1$ is equivalent to the condition that $(\lambda_w)_{i+1} = (\lambda_{w_{\min}})_i\le (\lambda_w)_i$ for all $i$.
\end{proof}

\subsection{The Branching Rule} \label{branchingrulesubsection}

In this subsection, we will use Corollary \ref{generalbranching} to determine a branching rule for the $\beta$-Grothendieck polynomials $\mathcal{G}^{(\beta)}_w(\boldsymbol{x};\boldsymbol{y})$. Immediate from that corollary is that restriction from $S_n$ to $S_{n-1}$ gives $\mathcal{G}^{(\beta)}_w(\boldsymbol{x};\boldsymbol{y})$ as a linear combination of Grothendieck polynomials corresponding to permutations in the half-strong-half weak Bruhat interval $I_w:=w_{\min}\le v\le_L w_{\max}$.

\begin{proposition} \label{onerowpartitionfunction}
The one-row partition function of $\mathfrak{S}'_{1,w^{-1}}(\boldsymbol{y},\boldsymbol{x})$, using the weights from Figure \ref{essentiallymodelP}, is \[T(w,v) = \begin{cases} \prod_{i=1}^n d_i, & \text{if } v\in I_w\\ 0, & \text{if } v\notin I_w,\end{cases} \hspace{20pt} \text{where} \hspace{20pt} d_i = \begin{cases} x_1\oplus y_i, & \text{if } w(i) = v(i+1) \\ 1, & \text{if } w(i) < v(i+1) \\ 1+\beta(x_1\oplus y_i), & \text{if } w(i) > v(i+1). \end{cases}\]
\end{proposition}

\begin{proof}
By Theorem \ref{interleaving}, the one-row partition function $T(w,v)$ has exactly one nonzero state if $v\in I_w$, and none if $v\notin I_w$.

If $v\in I_w$, column $i$ of $T(w,v)$ contains a vertex of type $\tt{b}_2$ precisely when $w(i)=v(i+1)$, a vertex of type $\tt{c}_2$ precisely when $w(i)<v(i+1)$, and a vertex of type $\tt{c}_1$ precisely when $w(i)>v(i+1)$.
\end{proof}

\begin{corollary} \label{branchingrule}
The $\beta$-Grothendieck polynomials obey the following branching rule:

\begin{equation}\mathcal{G}^{(\beta)}_w(x_1,\ldots,x_n;y_1,\ldots,y_n) = \sum_{v\in I_w}  \left(\prod_{i=1}^n d_i\right) \mathcal{G}^{(\beta)}_{v^-}(x_2,\ldots,x_n;y_1,\ldots,y_{n-1}). \label{branchingruleequation}\end{equation} \end{corollary}

\begin{proof}
Let us view the row below the top boundary to be a fixed permutation $v$ (see Figure \ref{branchingrulefigure}). Note that for a given $v\in I_w$, the lower $n-1$ rows of the model have partition function $\mathcal{G}^{(\beta)}_{v^-}(x_2,\ldots,x_n;y_1,\ldots,y_{n-1})$. The result now follows from Theorem \ref{interleaving} and Proposition \ref{onerowpartitionfunction}.
\end{proof}

\begin{figure}[h]
\begin{center}
\scalebox{0.8}{
\begin{tikzpicture}
  \coordinate (ab) at (1,0);
  \coordinate (ad) at (3,0);
  \coordinate (af) at (5,0);
  \coordinate (ah) at (7,0);
  \coordinate (ba) at (0,1);
  \coordinate (bc) at (2,1);
  \coordinate (be) at (4,1);
  \coordinate (bg) at (6,1);
  \coordinate (bi) at (8,1);
  \coordinate (cb) at (1,2);
  \coordinate (cd) at (3,2);
  \coordinate (cf) at (5,2);
  \coordinate (ch) at (7,2);
  \coordinate (da) at (0,3);
  \coordinate (dc) at (2,3);
  \coordinate (de) at (4,3);
  \coordinate (dg) at (6,3);
  \coordinate (di) at (8,3);
  \coordinate (eb) at (1,4);
  \coordinate (ed) at (3,4);
  \coordinate (ef) at (5,4);
  \coordinate (eh) at (7,4);
  \coordinate (fa) at (0,5);
  \coordinate (fc) at (2,5);
  \coordinate (fe) at (4,5);
  \coordinate (fg) at (6,5);
  \coordinate (fi) at (8,5);
  \coordinate (gb) at (1,6);
  \coordinate (gd) at (3,6);
  \coordinate (gf) at (5,6);
  \coordinate (gh) at (7,6);
  \coordinate (ha) at (0,7);
  \coordinate (hc) at (2,7);
  \coordinate (he) at (4,7);
  \coordinate (hg) at (6,7);
  \coordinate (hi) at (8,7);
  \coordinate (ib) at (1,8);
  \coordinate (id) at (3,8);
  \coordinate (if) at (5,8);
  \coordinate (ih) at (7,8);
  \coordinate (bb) at (1,1);
  \coordinate (bd) at (3,1);
  \coordinate (bf) at (5,1);
  \coordinate (bh) at (7,1);
  \coordinate (db) at (1,3);
  \coordinate (dd) at (3,3);
  \coordinate (df) at (5,3);
  \coordinate (dh) at (7,3);
  \coordinate (fb) at (1,5);
  \coordinate (fd) at (3,5);
  \coordinate (ff) at (5,5);
  \coordinate (fh) at (7,5);
  \coordinate (hb) at (1,7);
  \coordinate (hd) at (3,7);
  \coordinate (hf) at (5,7);
  \coordinate (hh) at (7,7);
  \coordinate (bax) at (0,1.5);
  \coordinate (bcx) at (2,1.5);
  \coordinate (bex) at (4,1.5);
  \coordinate (bgx) at (6,1.5);
  \coordinate (bix) at (8,1.5);
  \coordinate (dax) at (0,3.5);
  \coordinate (dcx) at (2,3.5);
  \coordinate (dex) at (4,3.5);
  \coordinate (dgx) at (6,3.5);
  \coordinate (dix) at (8,3.5);
  \coordinate (fax) at (0,5.5);
  \coordinate (fcx) at (2,5.5);
  \coordinate (fex) at (4,5.5);
  \coordinate (fgx) at (6,5.5);
  \coordinate (fix) at (8,5.5);
  \coordinate (hax) at (0,7.5);
  \coordinate (hcx) at (2,7.5);
  \coordinate (hex) at (4,7.5);
  \coordinate (hgx) at (6,7.5);
  \coordinate (hix) at (8,7.5);
  \draw (ab)--(ib);
  \draw (ad)--(id);
  \draw (af)--(if);
  \draw (ah)--(ih);
  \draw (ba)--(bi);
  \draw (da)--(di);
  \draw (fa)--(fi);
  \draw (ha)--(hi);
  \draw[fill=white] (ab) circle (.3);
  \draw[fill=white] (ad) circle (.3);
  \draw[fill=white] (af) circle (.3);
  \draw[fill=white] (ah) circle (.3);
  \draw[fill=white] (ba) circle (.3);
  \draw[fill=white] (bc) circle (.3);
  \path[fill=white] (be) circle (.3);
  \draw[fill=white] (bg) circle (.3);
  \draw[fill=white] (bi) circle (.3);
  \draw[fill=white] (cb) circle (.3);
  \draw[fill=white] (cd) circle (.3);
  \draw[fill=white] (cf) circle (.3);
  \draw[fill=white] (ch) circle (.3);
  \draw[fill=white] (da) circle (.4);
  \draw[fill=white] (dc) circle (.3);
  \path[fill=white] (de) circle (.3);
  \draw[fill=white] (dg) circle (.3);
  \draw[fill=white] (di) circle (.3);
  \path[fill=white] (eb) circle (.3);
  \path[fill=white] (ed) circle (.3);
  \path[fill=white] (ef) circle (.3);
  \path[fill=white] (eh) circle (.3);
  \draw[fill=white] (fa) circle (.3);
  \draw[fill=white] (fc) circle (.3);
  \path[fill=white] (fe) circle (.3);
  \draw[fill=white] (fg) circle (.3);
  \draw[fill=white] (fi) circle (.3);
  \draw[fill=white] (gb) circle (.3);
  \draw[fill=white] (gd) circle (.3);
  \draw[fill=white] (gf) circle (.3);
  \draw[fill=white] (gh) circle (.3);
  \draw[fill=white] (ha) circle (.3);
  \draw[fill=white] (hc) circle (.3);
  \path[fill=white] (he) circle (.3);
  \draw[fill=white] (hg) circle (.3);
  \draw[fill=white] (hi) circle (.3);
  \draw[fill=white] (ib) circle (.3);
  \draw[fill=white] (id) circle (.3);
  \draw[fill=white] (if) circle (.4);
  \draw[fill=white] (ih) circle (.3);
  \path[fill=white] (bb) circle (.3);
  \path[fill=white] (bd) circle (.3);
  \path[fill=white] (bf) ellipse (.5cm and .3cm);
  \path[fill=white] (bh) circle (.3);
  \path[fill=white] (db) ellipse (.5cm and .3cm);
  \path[fill=white] (dd) ellipse (.5cm and .3cm);
  \path[fill=white] (df) ellipse (.7cm and .3cm);
  \path[fill=white] (dh) ellipse (.5cm and .3cm);
  \path[fill=white] (fb) circle (.3);
  \path[fill=white] (fd) circle (.3);
  \path[fill=white] (ff) ellipse (.5cm and .3cm);
  \path[fill=white] (fh) circle (.3);
  \path[fill=white] (hb) circle (.3);
  \path[fill=white] (hd) circle (.3);
  \path[fill=white] (hf) ellipse (.5cm and .3cm);
  \path[fill=white] (hh) circle (.3);
  \node at (bb) {$T_{n,1}$};
  \node at (bd) {$T_{n,2}$};
  \node at (bf) {$T_{n,n-1}$};
  \node at (bh) {$T_{n,n}$};
  \node at (db) {$T_{n-1,1}$};
  \node at (dd) {$T_{n-1,2}$};
  \node at (df) {$T_{n-1,n-1}$};
  \node at (dh) {$T_{n-1,n}$};
  \node at (fb) {$T_{2,1}$};
  \node at (fd) {$T_{2,2}$};
  \node at (ff) {$T_{2,n-1}$};
  \node at (fh) {$T_{2,n}$};
  \node at (hb) {$T_{1,1}$};
  \node at (hd) {$T_{1,2}$};
  \node at (hf) {$T_{1,n-1}$};
  \node at (hh) {$T_{1,n}$};
  \node at (eb) {$\vdots$};
  \node at (ed) {$\vdots$};
  \node at (ef) {$\vdots$};
  \node at (eh) {$\vdots$};
  \node at (be) {$\cdots$};
  \node at (de) {$\cdots$};
  \node at (fe) {$\cdots$};
  \node at (he) {$\cdots$};
  \node at (-1,8) {row:};
  \node at (-1.2,7) {1};
  \node at (-1.2,5) {2};
  \node at (-1.2,4) {$\vdots$};
  \node at (-1.2,1) {$n$};
  \node at (ib) {$w_1$};
  \node at (id) {$w_2$};
  \node at (if) {$w_{n-1}$};
  \node at (ih) {$w_n$};
  \node at (ha) {$1$};
  \node at (hi) {$+$};
  \node at (gb) {$v_2$};
  \node at (gd) {$\ldots$};
  \node at (gf) {$v_n$};
  \node at (gh) {$+$};
  \node at (fa) {$2$};
  \node at (fi) {$+$};
  \node at (da) {$n${\tiny $-$}1};
  \node at (di) {$+$};
  \node at (ba) {$n$};
  \node at (-1.2,3) {$n-1$};
  \node at (bi) {$+$};
  \node at (ab) {$+$};
  \node at (ad) {$+$};
  \node at (af) {$+$};
  \node at (ah) {$+$};
  \node at (7,9) {$n$};
  \node at (5,9) {$n-1$};
  \node at (4,9) {$\ldots$};
  \node at (3,9) {$2$};
  \node at (1,9) {$1$};
  \node at (0,9.04) {column:};
\end{tikzpicture}}

\end{center}
\caption{The branching rule computation in Corollary \ref{branchingrule}. The partition function $Z(\mathfrak{S}'_{1,w^{-1}})$ can be given as the sum over $v\in I_w$ of the partition function of this diagram. Fixing $v$ splits the diagram into two pieces. The top piece is $T(w,v)$, while the bottom piece is $\mathcal{G}_{v^-}^{(\beta)}(\boldsymbol{x};\boldsymbol{y})$.}
\label{branchingrulefigure}
\end{figure}
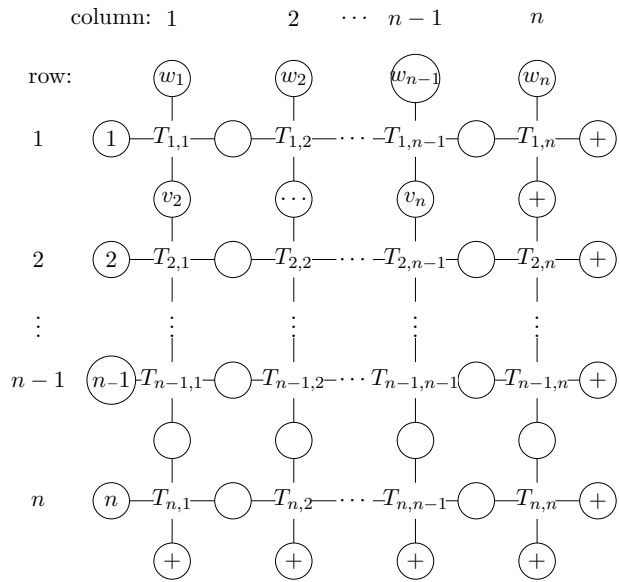

\begin{remark}
If we do the same process for $\mathfrak{S}'_{1,w}(\boldsymbol{x},\boldsymbol{y})$ and apply the involution (\ref{hudsoneq}), we obtain the following identity:

\[\mathcal{G}^{(\beta)}_w(x_1,\ldots,x_n;y_1,\ldots,y_n) = \sum_{w_{\min}\le v\le_R w_{\max}} \left(\prod_{i=1}^n f_i\right) \mathcal{G}^{(\beta)}_{v_-}(x_1,\ldots,x_{n-1};y_2,\ldots,y_n),\] where \[f_i = \begin{cases} x_i\oplus y_1, & \text{if } w^{-1}(i) = v^{-1}(i+1) \\ 1, & \text{if } w^{-1}(i) < v^{-1}(i+1) \\ 1+\beta(x_i\oplus y_1), & \text{if } w^{-1}(i) > v^{-1}(i+1). \end{cases} \hspace{20pt} and \hspace{20pt} v_- = ((v^{-1})^-)^{-1}.\]
We can alternatively obtain this formula by simply ``conjugating'' (\ref{branchingruleequation}) by (\ref{hudsoneq}).
\end{remark}

\begin{remark}
A similar branching rule to Corollary \ref{branchingrule} may be possible for the biaxial double $\beta$-Grothendieck polynomials. We have not attempted to find such a rule.
\end{remark}

\bibliographystyle{siam}
\bibliography{bibliography.bib}

\end{document}